\documentclass[twoside,11pt]{article}

\usepackage{amsmath}
\usepackage{jmlr2e}

\usepackage{bbding}

\usepackage{color}
\usepackage{booktabs}
\usepackage{enumerate}

\newcommand{\blue}[1]{\begin{color}{blue}#1\end{color}}

\newtheorem{assumption}{Assumption}

\usepackage[noend]{algpseudocode}

\usepackage{algorithmicx,algorithm}
\usepackage{changepage}

\usepackage[utf8x]{inputenc}
\usepackage{graphicx}
\usepackage{float} 
\usepackage{subfigure}
\usepackage[singlelinecheck=false]{caption}

\ShortHeadings{An Efficient HPR Algorithm for the WBP}{Zhang, Yuan, and Sun}

\begin{document}

\title{An Efficient HPR Algorithm for the Wasserstein Barycenter Problem with $O({Dim(P)}/\varepsilon)$ Computational Complexity}

\author{\name Guojun Zhang \email guojun.zhang@connect.polyu.hk \\
		\addr Department of Applied Mathematics\\
		The Hong Kong Polytechnic University\\
       Hung Hom, Kowloon, Hong Kong
       \AND
       \name Yancheng Yuan\thanks{Corresponding author.} \email yancheng.yuan@polyu.edu.hk \\
		\addr Department of Applied Mathematics\\
		The Hong Kong Polytechnic University\\
       Hung Hom, Kowloon, Hong Kong
       \AND
\name Defeng Sun \email defeng.sun@polyu.edu.hk \\
		\addr Department of Applied Mathematics\\
		The Hong Kong Polytechnic University\\
       Hung Hom, Kowloon, Hong Kong}

\editor{}

\maketitle

\begin{abstract}  
In this paper, we propose and analyze an efficient Halpern-Peaceman-Rachford (HPR) algorithm for solving the Wasserstein barycenter problem (WBP) with fixed supports. While the Peaceman-Rachford (PR) splitting method itself may not be convergent for solving the WBP, the HPR algorithm can achieve an $O(1/\varepsilon)$ non-ergodic iteration complexity with respect to the Karush–Kuhn–Tucker (KKT) residual. More interestingly, we propose an efficient procedure with linear time computational complexity to solve the linear systems involved in the subproblems of the HPR algorithm. As a consequence, the HPR algorithm enjoys an $O({\rm Dim(P)}/\varepsilon)$ non-ergodic computational complexity in terms of flops for obtaining an $\varepsilon$-optimal solution measured by the KKT residual for the WBP, where ${\rm Dim(P)}$ is the dimension of the variable of the WBP. This is better than the best-known complexity bound for the WBP. Moreover, the extensive numerical results on both the synthetic and real data sets demonstrate the superior performance of the HPR algorithm for solving the large-scale WBP.    
\end{abstract}


\section{Introduction}
Optimal transport (OT) \citep{monge1781memoire, kantorovich1942translocation} defines a Wasserstein distance between two probability distributions as the minimal cost of transportation. As an important application, the Wasserstein distance naturally leads to the concept of Wasserstein barycenter, which defines a mean of a set of complex objects (i.e., images, videos, texts, and so on) that can preserve their geometric structure\citep{agueh2011barycenters}. The WBP has made a significant impact in a variety of fields, including machine learning \citep{li2008real,ye2014scaling,peyre2019computational}, physics \citep{cotar2013density}, statistics \citep{bigot2018characterization}, economics \citep{chiappori2010hedonic,galichon2016optimal}, brain imaging \citep{gramfort2015fast}, and so on. However, the computational cost for computing the Wasserstein distance and finding the Wasserstein barycenters is expensive, especially for modern applications with an immense amount of data. In this paper, we focus on the WBP for discrete distributions with fixed supports. For this setting, \cite{anderes2016discrete} formulated the WBP a linear programming (LP) problem. Nonetheless, state-of-the-art linear programming solvers such as Gurobi, face challenges in solving the WBP even with a moderate number of objects and supports. 

To overcome the computational challenges, \cite{cuturi2013sinkhorn} proposed an entropic regularization to the OT, such that Sinkhorn’s algorithm is applicable to solving the entropy regularized OT problem and computing an approximate Wasserstein distance. \cite{cuturi2014fast} further applied the entropic regularization idea to study the WBP. Later, \cite{benamou2015iterative} proposed an iterative Bergman projection (IBP) method, which generalized Sinkhorn's algorithm, to solve the entropy regularized WBP. Along this direction, plenty of algorithms have been proposed, including gradient-type methods \citep{cuturi2016smoothed},  accelerated primal-dual gradient descent \citep{dvurechenskii2018decentralize,krawtschenko2020distributed}, the fast IBP \citep{lin2020fixed}, stochastic gradient descent \citep{claici2018stochastic, tiapkin2020stochastic}, distributed and parallel gradient descent \citep{staib2017parallel,uribe2018distributed,rogozin2021decentralized}.

When the regularization parameter in the entropy regularized WBP is moderate, say no less than $10^{-2}$, the aforementioned algorithms, such as the IBP, can be effective in finding an approximate solution to the entropy regularized WBP. However, it is known that one needs to choose a small regularization parameter in the regularized WBP to obtain a high-quality solution in real applications. This has been also observed in our numerical experiments. A small regularization parameter will usually cause numerical issues and a slower convergence for these algorithms, such as the IBP. Some stabilized and scaling techniques  have been proposed to improve the robustness of the entropic regularization based algorithms \citep{schmitzer2019stabilized}, but the efficiency of the stabilized algorithms is still not satisfying for solving the large-scale WBP. These issues may restrict the role of the entropic regularization approach in applications with high requirements of the solution quality \citep{bonneel2016wasserstein}.

\begin{table}[!ht]
\centering
\resizebox{\textwidth}{!}{%
\begin{tabular}{ccccccc}
\toprule[2pt]
\textbf{Algorithm}                                                           &  \textbf{Obj\textsubscript{P}} & \textbf{D\textsubscript{gap}} & \textbf{R\textsubscript{KKT}} & \textbf{Complexity}     & \textbf{Ergodic} & \textbf{Non-ergodic}                                \\ 
\toprule[2pt]
\begin{tabular}[c]{@{}c@{}}IBP\\  \citep{benamou2015iterative,kroshnin2019complexity}\end{tabular}                 & \Checkmark                                                                                                                                               &                                                       &                                                        & $\widetilde{O}(C^2 Tm^{2} / \varepsilon^{2})$    &  &   \Checkmark \\ [10pt]

\hline

\begin{tabular}[c]{@{}c@{}}PDAGD \\ \citep{kroshnin2019complexity}\end{tabular}                                    & \Checkmark                                                                                                                                               &                                                       &                                                        & $\widetilde{O}(C T m^{5 / 2}/\varepsilon)$     & \Checkmark  & \\ [10pt]
\hline
\begin{tabular}[c]{@{}c@{}}FastIBP \\ \citep{lin2020fixed}\end{tabular}                                            & \Checkmark                                                                                                                                                &                                                       &                                                        & $\widetilde{O}(C^{4/3}(T m^{7 / 3}/\varepsilon^{4 / 3}))$ &  & \Checkmark\\ [10pt]
\hline
\begin{tabular}[c]{@{}c@{}}Dual extrapolation with area-convexity \\ \citep{dvinskikh2021improved}\end{tabular}    &                                                                           &                                                                          \Checkmark                                                    &                                                        & $\widetilde{O}(C T m^{2}/ \varepsilon)$       &   \Checkmark &  \\[10pt]
\hline
\begin{tabular}[c]{@{}c@{}}Mirror-Prox \\ \citep{dvinskikh2021improved}\end{tabular}                               &                                                                           &                                                                         \Checkmark                                                    &                                                        & $\widetilde{O}(C T m^{5 / 2}/ \varepsilon)$       & \Checkmark & \\ [10pt]
\hline
\begin{tabular}[c]{@{}c@{}}Accelerated alternating minimization \\ \citep{guminov2021combination}\end{tabular}     & \Checkmark                                                                                                                                                 &                                                       &                                                        & $\widetilde{O}(C T m^{5 / 2} /{\varepsilon})$     &  & \Checkmark  \\ [10pt]
\hline
\begin{tabular}[c]{@{}c@{}}Accelerated Bergman primal-dual method\\  \citep{chambolle2022accelerated}\end{tabular} & \Checkmark                                                                                                                                                 &                                                       &                                                        & $\widetilde{O}(C T m^{5 / 2} / {\varepsilon})$      &   \Checkmark &
\\
\bottomrule[2pt]
\end{tabular}%
}
\caption{A summary of the known complexity of entropic regularization type algorithms for the WBP with $T$ sample distributions and $m$ supports. In this table, \textbf{Obj\textsubscript{P}}, \textbf{D\textsubscript{gap}} and \textbf{R\textsubscript{KKT}} mean that the error $\varepsilon > 0$ is measured by the primal objective function value gap, the duality gap, and the KKT residual, respectively. The constant $C$ only depends on the infinity norm of the cost matrices.}
\label{tab:complexity-summary}
\end{table}

The instability and unsatisfying efficiency of the entropic regularization based algorithm for solving the WBP in the relatively high accuracy regime motivate some researchers to move back to designing efficient algorithms for solving the WBP without entropic regularization. In this direction, the first-order splitting algorithms, in particular, the alternating direction method of multipliers (ADMM) \citep{glowinski1975approximation,gabay1976dual} and its variants are very popular. When applying the ADMM directly to solve the WBP, in each iteration, one needs to solve a huge-scale linear system, which is computationally challenging or even forbidden. The specific definition of this linear system can be found later in Section \ref{sec: def-WBP}. To avoid solving this huge linear system, \cite{ye2017fast} proposed a modified Bregman ADMM (mBADMM), where each step of the mBADMM has a closed-form solution for solving the WBP. Later, \cite{yang2021fast} proposed a symmetric Gauss-Seidel ADMM (sGS-ADMM) for solving the WBP, which is more efficient and robust than the mBADMM. Regarding the WBP as a multi-block problem and partially using the special structure of the WBP, the sGS-ADMM enjoys a low computational complexity in each iteration. It has been also demonstrated in \citep{yang2021fast} that the sGS-ADMM is more stable compared to the algorithms based on entropic regularization, such as IBP. The quality of the solution obtained by the sGS-ADMM can also outperform the one obtained by the IBP. However, we observed in our numerical experiments that the sGS-ADMM generally needs more iterations than the ADMM if the huge linear systems involved can be solved. This phenomenon has also been observed in our numerical experiments. Moreover, it has been demonstrated in \citep{yang2021fast} that the sGS-ADMM is much more efficient than the mBADMM. These observations motivate us to carefully investigate the linear system involved in the ADMM for solving the WBP. In this paper, we will propose a linear time complexity procedure to exactly solve the linear system. As a consequence, we can directly get a fast-ADMM for solving the WBP with a cheap per-iteration computational complexity. As a byproduct, we also get a fast-ADMM for solving the OT problem. 

The proposed efficient procedure for solving the linear system further motivates us to investigate the computational complexity for solving the WBP. The computational complexity of finding an approximate solution to the OT problem and the WBP with $\varepsilon$-precision has been extensively studied in recent years, in particular for the entropic regularization type algorithms. For the rest of this section, we consider the WBP with $T$ sample distributions and $m$ supports. \cite{kroshnin2019complexity} have established an $\widetilde{O}(Tm^2/\varepsilon^2)$ complexity bound for the IBP method and an $\widetilde{O}(Tm^{2.5}/\varepsilon)$ complexity bound for a variant of the primal-dual accelerated gradient (PDAGD) method (both in the sense of the primal objective function value gap). Here, the notation $\widetilde{O}(\cdot)$ hides only absolute constants and polylogarithmic factors. So far, the best computational complexity bound for the WBP is $\widetilde{O}(T m^{2} / \varepsilon)$ (in the sense of the duality gap), which is achieved by the dual extrapolation algorithm proposed by \cite{dvinskikh2021improved}. A summary of the complexity bounds of the entropic regularization type algorithms for solving the WBP is in Table \ref{tab:complexity-summary}.

\begin{table}[!ht]
\centering
\resizebox{\textwidth}{!}{%
\begin{tabular}{ccccccc}
\toprule[2pt]
\textbf{Algorithm}                                                           &  \textbf{Obj\textsubscript{P}} & \textbf{Obj\textsubscript{D}} & \textbf{R\textsubscript{KKT}} & \textbf{Complexity}     & \textbf{Ergodic} & \textbf{Non-ergodic}                                \\ 
\toprule[2pt]
\begin{tabular}[c]{@{}c@{}}sGS-ADMM (for dual WBP)\\ \citep{yang2021fast} + \citep{cui2016convergence}\end{tabular}                               &                                                                           &                                                                                                                          &                                             \Checkmark             & $O((D_1 + D_2) T m^{2} / \varepsilon^2)$       &  & \Checkmark\\ [10pt]
\hline
\begin{tabular}[c]{@{}c@{}}sGS-ADMM (for dual WBP) \\ \citep{yang2021fast} + \citep{cui2016convergence}\end{tabular}                               &                                                                           &      \Checkmark                                                                                                      &                                                        & $O((D_1^2 + D_2^2) T m^{2} / \varepsilon)$       & \Checkmark & \\ [10pt]
\hline
\begin{tabular}[c]{@{}c@{}}\blue{HPR} (for primal or dual WBP) \\ (\blue{This paper})\end{tabular}     &                                                                         &                                                                         &                                                                                                         \Checkmark    & \blue{$O(D_1 T m^2 / \varepsilon)$}     &  & \Checkmark  \\ [10pt]
\hline
\begin{tabular}[c]{@{}c@{}}\blue{HPR} (for primal WBP) \\ (\blue{This paper})\end{tabular}     &                                                                        \Checkmark &                                                                         &                                                                                                              & \blue{$O(D_1^2T m^2 / \varepsilon)$}     &  & \Checkmark  \\ [10pt]
\hline
\begin{tabular}[c]{@{}c@{}}\blue{HPR} (for dual WBP) \\ (\blue{This paper})\end{tabular}     &                                                                         &  \Checkmark                                                                       &                                                                                                              & \blue{$O(D_1^2T m^2 / \varepsilon)$}     &  & \Checkmark  \\ [10pt]
\bottomrule[2pt]
\end{tabular}%
}
\caption{A summary of the complexity of the HPR algorithm and the sGS-ADMM algorithm for the WBP with $T$ sample distributions and $m$ supports. In the table, \textbf{Obj\textsubscript{P}}, \textbf{Obj\textsubscript{D}}, and \textbf{R\textsubscript{KKT}} mean that the error $\varepsilon > 0$ is measured by the primal objective function value gap, the dual objective function value gap, and the KKT residual, respectively. The constants $D_1$ and $D_2$ only depend on the distance of the initial point to the solution set. The additional constant $D_2$ in the complexity bound of the sGS-ADMM comes from the additional proximal term, which is automatically generated by the algorithm. We regard the LP reformulation of the WBP and its dual problem as the primal WBP and the dual WBP, respectively.}
\label{tab:complexity-summary-ADMM}
\end{table}

On the other hand, the discussion of the computational complexity of the operator splitting methods, such as the ADMM type algorithms for solving the WBP is not that rich. In this direction, we can obtain an $O(Tm^2/\varepsilon^2)$ computational complexity in terms of the KKT residual by combining the per-iteration computational complexity of the sGS-ADMM \citep{yang2021fast} and the non-ergodic $O(1/\sqrt{k})$ iteration convergence rate of the majorized ADMM \citep{cui2016convergence}. This is perhaps the only known complexity bound so far which is comparable to the best known complexity bound of the entropy regularized type algorithms. Actually, there are some known iteration complexity results of the ADMM. \cite{monteiro2013iteration} first proved the $O(1/k)$ ergodic convergence rate (regarding the KKT-type residual) of the ADMM with unit dual step size for a class of linearly constrained convex programming problems with a separable objective function. Later, \cite{davis2016convergence} established an $O(1/\sqrt{k})$ non-ergodic iteration complexity bound of the ADMM with unit dual step size with respect to the primal feasibility and the primal objective function value gap. The key bottleneck for establishing an attractive complexity bound comes from the expensive Cholesky decomposition, which is $O(T^3m^3)$, although it only needs to do once. Thus, for the WBP, the ergodic computational complexity and the non-ergodic computational complexity of the ADMM method can be known to be $O((T^3m^3) + (Tm^2/\varepsilon))$ and $O((T^3m^3) + (Tm^2/\varepsilon^2))$, respectively. Instead, the fast-ADMM can benefit from the linear time complexity procedure for solving the involved linear system. Thus, the ergodic computational complexity and the non-ergodic computational complexity of the fast-ADMM method can be improved to be $O(Tm^2/\varepsilon)$ and $O(Tm^2/\varepsilon^2)$, respectively. In real applications, the non-ergodic complexity is more important since the non-ergodic sequence can preserve sparsity. In this paper, we will propose a more appealing algorithm that enjoys a more important $O(Tm^2/\varepsilon)$ non-ergodic complexity. This new bound is even better than the best-known ergodic computational complexity of $\widetilde{O}(Tm^2/\varepsilon)$ \citep{dvinskikh2021improved}.

The new appealing non-ergodic computational complexity bound for the WBP will be achieved by an HPR algorithm, which applies the Halpern iteration \citep{halpern1967fixed} to the PR splitting method \citep{lions1979splitting}. While it is not clear to us when the PR applied to the WBP converges, the convergence of the HPR follows directly from \citep{wittmann1992approximation}. Starting from any initial point, the HPR enjoys an $O(1/k)$ iteration complexity with respect to the fixed point residual of the corresponding PR operator \citep{lieder2021convergence}. When the HPR algorithm is applied to solving the two-block convex optimization problem with linear constraints, we will establish an $O(1/k)$ non-ergodic iteration complexity in terms of the KKT residual, and the primal objective function value gap. Here, we want to mention that, \cite{kim2021accelerated} proposed an accelerated ADMM and proved an $O(1/k)$ convergence rate with respect to the primal feasibility only. To this end, we will prove that the HPR algorithm enjoys an $O(Tm^2/\varepsilon)$ computational complexity guarantee for obtaining an $\varepsilon$-optimal solution with respect to the KKT residual of the WBP with $T$ distributions and $m$ supports. We briefly compare the computational complexity bounds of the HPR algorithm and the sGS-ADMM algorithm and summarize the results in Table \ref{tab:complexity-summary-ADMM}. As we can see in the table, in terms of the non-ergodic complexity, the HPR improves an $O(1/\epsilon)$ compared to the sGS-ADMM, which is substantial.  Here, we want to mention that, the constants of the complexity bounds of the HPR algorithm and the sGS-ADMM algorithm only depend on the distance of the initial point to the solution set, which is different from the constant in the complexity bounds of the entropy regularized algorithms (in Table \ref{tab:complexity-summary}).  we will numerically justify later in this paper that these constants can be comparable.

We highlight the main contributions of this paper as follows:

\begin{enumerate}
    \item We propose a linear time complexity procedure for exactly solving the involved linear systems in the ADMM for solving the WBP. As a byproduct, we can also design a linear time complexity procedure for similar linear systems involved in solving the OT problem. As a consequence, we can have a fast-ADMM for solving the large-scale WBP and the OT problem with a cheap per-iteration computational complexity. 
    \item We propose an HPR algorithm for solving the maximal monotone inclusion problems. When it is applied to solve the two-block convex programming problems with linear constraints, we establish an $O(1/\epsilon)$ non-ergodic iteration complexity with respect to the KKT residual and the objective function value gap.  
     \item We prove the HPR algorithm enjoys the $O(Tm^2/\varepsilon)$ non-ergodic computational complexity for obtaining an $\varepsilon$-optimal solution with respect to the KKT residual of the WBP with $T$ distributions and $m$ supports. The complexity bound with respect to the KKT residual measure is important for the primal-dual optimization algorithms since it is widely used as a stopping criterion. 

    \item We demonstrate the superior numerical performance of the HPR algorithm for obtaining high-quality solutions to the WBP on both synthetic data and real data. 
\end{enumerate}

\par 
\paragraph{Organization.} 
The rest of the paper is organized as follows. In section 2. we introduce the model of the WBP. To make the ADMM computationally affordable for solving the WBP,  we propose a linear time complexity procedure for solving the linear system involved.  In section 3, we first introduce the HPR for solving the maximal monotone inclusion problem. When the HPR algorithm is applied to solving the two-block convex optimization problem
with linear constraints, we will establish an $O(1/\varepsilon)$ non-ergodic iteration complexity in terms of the KKT residual and the primal objective function value gap. Then, we will prove an $O(Tm^2/\varepsilon)$ computational complexity in terms of flops for obtaining an $\varepsilon$-optimal solution measured by the KKT residual of the WBP with $T$ distributions and
$m$ supports.   Section 4 is devoted to demonstrating the efficiency and robustness of the HPR for solving WBP with extensive numerical experiments on both synthetic and real datasets. We conclude this paper in Section 5.

\paragraph{Notation.}
 We use $\mathbb{X}, \mathbb{Y}, \mathbb{Z}$ to denote finite-dimensional real Euclidean spaces equipped with the inner product $\langle \cdot, \cdot \rangle$ and its induced norm $\|\cdot\|$. In particular, we denote the $n$-dimensional real Euclidean space as $\mathbb{R}^n$. For any $x \in \mathbb{R}^n$, $y \in \mathbb{R}^n$, we define $\langle x, y \rangle := \sum_{i=1}^n x_i y_i$ and $\|x\| := \sqrt{\sum_{i=1}^n x_i^2}$, respectively. We also denote the nonnegative orthant of $\mathbb{R}^n$ as $\mathbb{R}^n_{+}$. A vector $x \in \mathbb{R}^n$ is a column vector by default. For a given linear mapping $A:\mathbb{X}\to \mathbb{Y}$, $A^{*}: \mathbb{Y}\to \mathbb{X}$ is the adjoint of $A$. We denote $\|A\| := \sup_{\|x\| \leq 1}\|Ax\|$ as the spectral norm of $A$. We denote the transpose of a matrix $B \in \mathbb{R}^{m \times n}$ as $B^{\top} \in \mathbb{R}^{n \times m}$. $\|B\|_F := \sqrt{{\rm trace}(BB^{\top})}$ and $\|B\|_{\infty} := \max_{1 \leq i \leq m, 1 \leq j \leq n} |B_{ij}|$ are the Frobenius norm of $B$ and the infinity norm of $B$, respectively. Here, ${\rm trace}(BB^{\top})$ is the trace of $BB^{\top}$. Let $\mathbf{1}_{m}$ (resp. $\mathbf{0}_m$)  denote the $m$ dimensional vector with all entries being 1 (resp. 0). $\otimes $ stands for Kronecker product. We denote the vectorization of a matrix $X \in \mathbb{R}^{m \times n}$ as $\operatorname{vec}(X) \in \mathbb{R}^{mn}$. For a collection of matrices $\{A_{1},\ldots,A_{m}\}$, we denote the block diagonal matrix with diagonal blocks $A_i$ as $\operatorname{diag}(A_{1},\ldots,A_{m})$. For a closed convex set $C$, we denote the indicator function of $C$ and the Euclidean projector over $C$ as $\delta_C$ and $\Pi_C(x) := \arg\min_{s \in C} ~ \|x-s\|$, respectively. Let $f: \mathbb{X} \to (-\infty, +\infty]$ be a proper closed and convex function. We denote the effective domain of $f$ and the proximal mapping of $f$ at $x \in \mathbb{X}$ as ${\rm dom}(f) := \{x \in \mathbb{X} ~|~ f(x) < +\infty\}$ and ${\rm Prox}_f(x) := \arg\min_{y \in \mathbb{X}} ~ \{f(y) + \frac{1}{2}\|y - x\|^2\}$, respectively. The Fenchel conjugate function of $f$ is defined as $f^*(x) := \sup_{y \in \mathbb{X}} ~ \langle y, x \rangle - f(y)$. Let $D \subseteq \mathbb{X}$ be a nonempty closed convex set. We call $\boldsymbol{T}: D \rightarrow D$ a nonexpansive operator if $\|\boldsymbol{T}(x) - \boldsymbol{T}(y)\| \leq \|x - y\|$ for any $x, y \in D$. For a nonexpansive operator $\boldsymbol{T}: D \rightarrow D$, we denote the set of its fixed points as $\operatorname{Fix}(\boldsymbol{T}) := \{x \in D ~|~ x = \boldsymbol{T}(x)\}$. 

\section{Preliminaries}
In this section, we first introduce the model of the Wasserstein barycenter problem. Then we proposed a linear time complexity procedure for the linear system involved in solving the Wasserstein barycenter problem.   

\subsection{Wasserstein Barycenter Problem}
\label{sec: def-WBP}
Consider the following discrete probability distribution with finite support points:
$$
\mathcal{P}:=\left\{\left(a_{i}, \boldsymbol{q}_{i}\right) \in \mathbb{R}_{+} \times \mathbb{R}^{d}: i=1, \cdots, m\right\},
$$
where  $\left\{\boldsymbol{q}_{1}, \cdots, \boldsymbol{q}_{m}\right\}$ are the support points and $\textbf{a}:=(a_{1}, \cdots, a_{m})$ is the associated probability satisfying $\sum_{i=1}^{m} a_{i}=1$. Given two discrete distributions $\mathcal{P}^{u}=\left\{\left(a_{i}^{u}, \boldsymbol{q}_{i}^{u}\right): i=1, \cdots, m_{u}\right\}$ and $\mathcal{P}^{v}=\left\{\left(a_{i}^{v}\right.\right.$, $\left.\left.\boldsymbol{q}_{i}^{v}\right): i=1, \cdots, m_{v}\right\}$, the $p$-Wasserstein distance between $\mathcal{P}^{u}$ and $\mathcal{P}^{v}$ is defined as the optimal objective function value of the following optimal transport problem: 
\begin{equation}\label{model:OT}
  \begin{array}{cl}
     \left(\mathcal{W}_{p}(\mathcal{P}^{u},\mathcal{P}^{v})\right)^p
       :=& \min \limits_{X\in \mathbb{R}^{m_u\times m_v}}  \langle X,\mathcal{D}(\mathcal{P}^{u},\mathcal{P}^{v})\rangle \\
      \mathrm{s.t.} &X^{\top}\mathbf{1_{m_u}}=\mathbf{a^{v}},\\
      &X \mathbf{1_{m_v}}=\mathbf{a^{u}}, \\      
      &X\geq \mathbf{0},
\end{array}  
\end{equation}
where $\mathcal{D}(\mathcal{P}^{u},\mathcal{P}^{v}) \in \mathbb{R}^{m_u \times m_v}$ is the distance matrix with $\mathcal{D}(\mathcal{P}^{u},\mathcal{P}^{v})_{ij} = \|{q}_{i}^{u}-{q}_{j}^{v}\|^{p}_p$ and $p\geq 1$. Based on the Wasserstein distance, Agueh and Carlier \citep{agueh2011barycenters} proposed the Wasserstein barycenter problem. Specifically, given a collection of discrete probability distributions $\left\{\mathcal{P}^{t}\right\}_{t=1}^{T}$ with $\mathcal{P}^{t}=\left\{\left(a_{i}^{t}, \boldsymbol{q}_{i}^{t}\right)\right.$ : $\left.i=1, \cdots, m_{t}\right\}$, a $p$-Wasserstein barycenter $\mathcal{P}^{c}:=\left\{\left(a_{i}^{c}, \mathbf{q}_{i}^{c}\right): i=1, \cdots, m\right\}$ with $m$ support points is a minimizer of the following optimization problem:
\begin{equation}\label{model:free}
    \min _{\mathcal{P}^{c} \in \Xi^{p}\left(\mathbb{R}^{d}\right)} \sum_{t=1}^{T} \omega_{t}\left(\mathcal{W}_{p}(\mathcal{P}^{c}, \mathcal{P}^{t})\right)^{p},
\end{equation}
where $\Xi^{p}\left(\mathbb{R}^{d}\right)$ denotes the set of all discrete probability distributions on $\mathbb{R}^{d}$ with finite $p$-th moment, and {the} weight vector $\left(\omega_{1}, \cdots, \omega_{T}\right)$ satisfies $\sum_{t=1}^{T} \omega_{t}=1$ and $\omega_{t}>0, t=1, \cdots, T$. Note that problem \eqref{model:free} is a non-convex multi-marginal OT problem, in which one needs to find the optimal support $\mathcal{Q}^{c}:=\{\mathbf{q}_i^{c},i=1,\ldots,m\}$ and the optimal weight vector $\mathbf{a}^{c}$ simultaneously.

In many real applications, the support $\mathcal{Q}^{c}$ can be specified empirically. As a result, one only needs to compute the optimal weight vector $\mathbf{a}^{c}$ via solving the optimization problem \eqref{model:free} with finite specified supports. In this paper, we focus on the WBP with specified supports. From now on, we assume that the support $\mathcal{Q}^{c}$ is given. Under this setting, the WBP can be formulated as the following linear programming:
\begin{equation}\label{model:fixedX}
\begin{array}{cl}
   	\min \limits_{\mathbf{a}^c \in\mathbb{R}^{m},\left\{X^{t}\right\}_{t=1}^{T}\in \mathbb{R}^{m \times m_{t}}} &  \sum_{t=1}^{T}\left\langle D^{t}, X^{t}\right\rangle \\
    \mathrm{ s.t.}   &\left(X^{t}\right)^{\top} \mathbf{1}_{m}=\mathbf{a}^{t},\quad  t=1, \cdots, T,\\
    & X^{t} \mathbf{1}_{m_{t}}=\mathbf{a}^{c},\quad t=1, \cdots, T,\\   
    & X^{t} \geq 0,\quad  t=1, \cdots, T,\\
    &\langle\mathbf{a}^{c},\mathbf{1}_m\rangle=1,

\end{array}
\end{equation}
where $D^{t}:=\omega_{t} \mathcal{D}\left(\mathcal{P}^{c} , \mathcal{P}^{t} \right)$. We can write \eqref{model:fixedX} as the following standard form linear programming problem: 
\begin{equation}\label{model:standLP0}
\begin{array}{ll}
\underset{x}{\min} & \langle{c}, {x}\rangle + \delta_{K}(x)\\
         \text { s.t. } &\hat{A} {x}=\hat{b},
\end{array}
\end{equation}
where
\begin{enumerate}
    \item ${x}=\left(\operatorname{vec}(X^{1}) ; \ldots ; \operatorname{vec}(X^{T});\mathbf{a}^c\right)$, $K$ is the nonnegative orthant with compatible dimension,
    \item $\hat{b}=\left(\mathbf{a}^{1}; \mathbf{a}^{2} ; \ldots ; \mathbf{a}^{T}; \mathbf{0}_{m} ; \ldots ; \mathbf{0}_{m};1\right), $
${c}=\left(\operatorname{vec}(D^{1}) ; \ldots ; \operatorname{vec}(D^{T}) ; \mathbf{0}_{m}\right)$,
\item $\hat{A}=\left[\begin{array}{cc} A_{1} & \mathbf{0}  \\ \hat{A}_{2} & \hat{A}_{3}\\ \mathbf{0} & \mathbf{1}_m^{\top}   \end{array}\right], A_{1} = \operatorname{diag}\left(I_{m_{1}} \otimes \mathbf{1}_{m}^{\top}, \ldots, I_{m_{T}} \otimes \mathbf{1}_{m}^{\top}\right)$, \\
$\hat{A}_{2} = \operatorname{diag}\left(\mathbf{1}_{m_{1}}^{\top} \otimes I_{m}, \ldots, \mathbf{1}_{m_{t}}^{\top} \otimes I_{m}\right),$ and $\hat{A}_{3} = -\mathbf{1}_{T} \otimes I_{m}.$
\end{enumerate}

Define 
\begin{equation}
\label{def: M-and-N}
M:=\sum_{i=1}^{T} m_{i}+T(m-1)+1, \quad N:=m \sum_{i=1}^{T} m_{i}+m.
\end{equation}

Let $A_{2}$ and $A_{3}$ be the matrices which are obtained from $\hat{A}_{2}$ and $\hat{A}_{3}$ by removing the $1$-st, $(m+1)$-th, \ldots, $(T-1)m+1$-th rows of $\hat{A}_{2}$ and $\hat{A}_{3}$, respectively. That is 
$$
{A}_{2} = \operatorname{diag}\left(\mathbf{1}_{m_{1}}^{\top} \otimes [\mathbf{0}_{m-1},I_{m-1}], \ldots, \mathbf{1}_{m_{t}}^{\top} \otimes [\mathbf{0}_{m-1},I_{m-1}]\right),\text{ and } {A}_{3} = -\mathbf{1}_{T} \otimes [\mathbf{0}_{m-1},I_{m-1}].
$$
Define 
\begin{equation}\label{Mat:A&b}
 A:=\left[\begin{array}{cc} A_{1} & \mathbf{0}  \\ {A}_{2} & A_{3}\\ \mathbf{0} & \mathbf{1}_m^{\top}  \end{array}\right]\in \mathbb{R}^{M\times N},\quad {b}:=\left(\mathbf{a}^{1}; \mathbf{a}^{2} ; \ldots ; \mathbf{a}^{T}; \mathbf{0}_{m-1} ; \ldots ; \mathbf{0}_{m-1};1\right)\in \mathbb{R}^{M}.  
\end{equation}
\cite{ge2019interior} made the following useful observation: 
\begin{proposition}\citep[Lemma 3.1]{ge2019interior}
 Consider $A\in \mathbb{R}^{M\times N}$ and $b\in \mathbb{R}^{M}$ defined in \eqref{Mat:A&b}. Then 
 $A$ has full row rank, and $\{x \in \mathbb{R}^N ~|~ Ax = b\} = \{x \in \mathbb{R}^N ~|~ \hat{A}x = \hat{b}\}$.
\end{proposition}
As a result, the linear programming problem \eqref{model:standLP0} is equivalent to 
\begin{equation}\label{model:standLP}
\begin{array}{ll}
\underset{x \in \mathbb{R}^N}{\min} & \langle{c}, {x}\rangle  + \delta_{K}(x)\\
         \text { s.t. } &{A} {x}={b}.
\end{array}
\end{equation}
The dual problem of \eqref{model:standLP} is 
\begin{equation}\label{model:dualLP}
  \underset{ y \in \mathbb{R}^M, s \in \mathbb{R}^N}{\min} \left\{-\langle{b}, {y}\rangle+\delta_{K}^{*}(-s)  \mid A^{*} y+s=c\right\},
  \end{equation}
where $\delta^*_{K}(\cdot)$ is the support function of $K$. The KKT conditions associated with \eqref{model:standLP} and \eqref{model:dualLP} are given by \begin{equation}\label{KKT}
    A^{*}y+s=c,\quad Ax=b,\quad  K \ni  x \perp s \in K,
\end{equation} 
where $x \perp s$ means $x$ is perpendicular to $s$, i.e., $\langle x, s \rangle = 0$.

One can apply the operator splitting algorithm like ADMM to solve \eqref{model:dualLP} to calculate the Wasserstein barycenter. The computational bottleneck is solving the linear system $AA^{*}y=R$ for a given $R\in \mathbb{R}^{M}$. Here, the dimension of the matrix $AA^{*}$ is $M \times M$ and $M$ is defined in \eqref{def: M-and-N}. For the WBP, $M$ can be extremely large. As a result, even if one Cholesky decomposition is not computationally affordable. In practice, the conjugate gradient method is widely used to solve this large linear system. Instead, in the next subsection, we will derive a linear time complexity procedure for this linear system. 

\subsection{A Linear Time Complexity Procedure for Solving $AA^{*}y=R$}

 Next, we will propose an efficient procedure for solving $AA^*y = R$, where $A$ is the matrix defined in \eqref{Mat:A&b} and $R \in \mathbb{R}^M$ is any given vector. The computational complexity for solving the linear system is only $O\left(Tm+\sum_{t=1}^{T}m_t\right)$. For notation convenience, we denote $M_1:=\sum_{t=1}^{T} m_{t}$ and $M_2:=T(m-1)$. By direct calculations, ${A} A^{*}$ can be written in the following form:
    \begin{equation}\label{Mat:WBP_AAT}
	    {A} A^{*}=
	\left[\begin{array}{cc} A_{1} & \mathbf{0}  \\ {A}_{2} & A_{3}\\ \mathbf{0} & \mathbf{1}_m^{\top}  \end{array}\right]\left[\begin{array}{ccc} A_{1}^{*} & {A}_{2}^{*} & \mathbf{0}\\ \mathbf{0}& A_{3}^{*} &\mathbf{1}_{m}  \end{array}\right]
	=\left[\begin{array}{ccc}
		E_{1} & E_{2}& \mathbf{0}  \\
		E_{2}^{*} & E_{3}+E_{4} & E_{5}\\ \mathbf{0} & E_{5}^{*} &m 
	\end{array}\right],
	\end{equation}	
	where	
	\begin{enumerate}
	    \item $E_{1}:=A_1A_1^{*} = \operatorname{diag}\left(m I_{m_1},\ldots, m I_{m_T} \right)\in   \mathbb{R}^{M_1 \times M_1}$,
	    \item $E_{2}:=A_1A_2^{*} = \operatorname{diag}\left(\mathbf{1}_{m_1}\mathbf{1}_{m-1}^{\top},\ldots, \mathbf{1}_{m_T}\mathbf{1}_{m-1} ^{\top} \right)  \in \mathbb{R}^{M_1 \times M_{2}}$,       \item $E_{3}:={A}_2{A}_2^{*}= \operatorname{diag}\left(m_1 I_{m-1},\ldots, m_{T} I_{m-1} \right)\in   \mathbb{R}^{M_2 \times M_2}$,  
	     \item $E_{4}:={A}_3{A}_3^{*}= (\mathbf{1}_{T} \mathbf{1}_{T}^{\top}) \otimes I_{m-1}\in \mathbb{R}^{M_{2} \times M_{2}}$,
	     \item $E_5:=A_3\mathbf{1}_m=-\mathbf{1}_{T}\otimes\mathbf{1}_{m-1}\in \mathbb{R}^{M_{2}}$.
	\end{enumerate}
 To better explore the structure of $AA^{*}y=R$, we rewrite it equivalently as
\begin{equation}\label{equ:normal}
	AA^{*}y=\left[\begin{array}{ccc}
		E_{1} & E_{2}& \mathbf{0}  \\
		E_{2}^{*} & E_{3}+E_{4} & E_{5}\\ \mathbf{0} & E_{5}^{*} &m 
	\end{array}\right]\left[\begin{array}{c}
		y_1  \\
		y_2 \\
		y_3 
	\end{array}\right]=\left[\begin{array}{c}
		R_1  \\
		R_2 \\
		R_3
	\end{array}\right].
\end{equation}
where $y:=(y_1; y_2; y_3)\in \mathbb{R}^{M_1}\times\mathbb{R}^{M_2} \times\mathbb{R}$ and $ R:= (R_1; R_2; R_3)\in \mathbb{R}^{M_1}\times\mathbb{R}^{M_2} \times\mathbb{R}$. To further explore the block structure of the linear system, we can denote $y_1:= (y_1^1; \dots; y_1^T) \in \mathbb{R}^{m_1} \times \cdots \times \mathbb{R}^{m_T}$, $y_2 = (y_2^1; \dots; y_2^T) \in \mathbb{R}^{m-1} \times \cdots \times \mathbb{R}^{m-1}$. Correspondingly, we write $R_1 = (R_1^1; \dots; R_1^T)$ and $R_2 = (R_2^1; \dots; R_2^T)$. The next proposition shows the solution $y$ to the linear system \eqref{equ:normal}.
\begin{proposition}
Consider $A\in \mathbb{R}^{M\times N}$ defined in \eqref{Mat:A&b}. Given $R\in \mathbb{R}^{M}$, the solution $y$ to  $AA^{*}y=R$ in the form \eqref{equ:normal} is given by
\begin{eqnarray}
        &&  y_2^{t}=  \frac{1}{m_t}\left( \hat {y}_2^{t}-\hat {y}_2^{a} \right),\quad  t=1,\ldots,T, \label{y2}\\
	&&  y_1^{t}=\frac{R_1^{t}}{m}- \frac{\mathbf{1}^{\top}_{m-1}y_2^t}{m} \mathbf{1}_{m_t},\quad t=1,\dots,T, \label{y1} \\
	&&  y_3=\frac{1}{m}(R_3+ \mathbf{1}_{M_2}^{\top}y_2),  \label{y3} 	  
\end{eqnarray}
where
\begin{enumerate}
    \item $ \hat {y}_2^{t}:=R_2^{t}+( \mathbf{1}_{m-1}^{\top}R^t_2- \mathbf{1}_{m_t}^{\top}R_1^t + R_3)\mathbf{1}_{m-1},\quad t=1,\ldots,T,$
    \item $\hat {y}_2^{a}:=\sum_{t=1}^{T} \frac{\bar{m}}{m_{t}} \hat {y}_2^{t},$ and $\bar{m}:=(1+\sum_{t=1}^{T}\frac{1}{m_t})^{-1}$.
\end{enumerate}
\end{proposition}
\begin{proof} By some direct calculations, we can solve \eqref{equ:normal} equivalently as  
\begin{eqnarray}
	&&y_1^{t}=(E_1^{-1}(R_1-E_2y_2))^{t}=\frac{R_1^{t}}{m}- \frac{\mathbf{1}^{\top}_{m-1}y_2^t}{m} \mathbf{1}_{m_t},\quad t=1,\dots,T, \\
	&&y_3=\frac{1}{m}(R_3-E_5^{*}y_2)=\frac{1}{m}(R_3+ \mathbf{1}_{M_2}^{\top}y_2),  \\
	&&(E_3-E_{2}^{*}E_1^{-1}E_2+E_4-\frac{1}{m}E_5E_5^{*})y_2=R_2-E_{2}^{*}E_1^{-1}R_1-\frac{1}{m}E_5R_3.  \label{equ:normal_2}   
\end{eqnarray}
As a result, the key is to obtain $y_2$ by solving \eqref{equ:normal_2}. For convenience, denote $\hat{E}_{3}:=E_{3}-E_{2}^{*}(E_{1})^{-1} E_{2}$ and $\hat{E}_{4}:=E_4-\frac{1}{m}E_5E_5^{*}$. Then, the linear system \eqref{equ:normal_2} can be rewritten as 
\begin{equation*}  
    (\hat{E}_3 + \hat{E}_4)y_2 =\hat{R}_2,
\end{equation*}
where $\hat{R}_2:=R_2-E_{2}^{*}E_1^{-1}R_1-\frac{1}{m}E_5R_3$. Define $Q:=I_{m-1}-1/m( \mathbf{1}_{m-1}\mathbf{1}_{m-1}^{\top})$. By \eqref{Mat:WBP_AAT} and some simple calculations, we have
\begin{equation}\label{Mat-E4}
  \hat{E}_{4}=\mathbf{1}_{T} \mathbf{1}_{T}^{\top}\otimes Q,
\end{equation}
and
\[  \hat{E}_3= \operatorname{diag} (m_1Q, \dots, m_TQ). \]
Moreover, by the Sherman–Morrison-Woodbury formula, we directly get
\begin{equation}\label{Mat-invE3}
    \hat{E}_3^{-1}= \operatorname{diag} (1/m_{1}Q^{-1}, \dots, 1/m_T Q^{-1}),
\end{equation}
where $Q^{-1}=(I_{m-1}+\mathbf{1}_{m-1} \mathbf{1}_{m-1}^{\top})$. Denote $\hat{Q} := Q^{1/2}$ such that $\hat{Q}\hat{Q}=Q$, and $\bar{Q}:=\mathbf{1}_{T}\otimes \hat {Q}$. Then we have $\hat{E}_{4}=\bar{Q} \bar{Q}^{*}$. Using the Sherman-Morrison-Woodbury formula, we can obtain 
$$
(\hat{E}_3 + \hat{E}_4)^{-1}=\hat{E}_3^{-1}-\hat{E}_3^{-1}\bar{Q}W\bar{Q}^{*}\hat{E}_3^{-1},
$$
where $W:=(I_{m-1}+\bar{Q}^{*}\hat{E}_3^{-1}\bar{Q})^{-1}=(I_{m-1}+\sum_{i=1}^{T}\hat{Q}(E_3^{t})^{-1}\hat{Q})^{-1}.$
It follows from \eqref{Mat-invE3} that 
$$\begin{aligned}
	\bar{Q}W\bar{Q}^{*}&=\mathbf{1}_{T}\mathbf{1}_{T}^{\top}\otimes \hat{Q}W\hat{Q}\\&=\mathbf{1}_{T}\mathbf{1}_{T}^{\top}\otimes(Q^{-1}+\sum_{i=1}^{T}\hat{E}_3^{-1})^{-1}\\&=\mathbf{1}_{T}\mathbf{1}_{T}^{\top}\otimes((1+\sum_{t=1}^{T}\frac{1}{m_t})Q^{-1})^{-1}\\&=
	\mathbf{1}_{T}\mathbf{1}_{T}^{\top}\otimes(1+\sum_{t=1}^{T}\frac{1}{m_t})^{-1}Q.
\end{aligned}
$$
Denote $\bar{m}:=(1+\sum_{t=1}^{T}\frac{1}{m_t})^{-1}$. Therefore, we have 
$$
\begin{aligned}
	(\hat{E}_3 + \hat{E}_4)^{-1}&=\hat{E}_3^{-1}-\hat{E}_3^{-1}\bar{Q}W\bar{Q}^{*}\hat{E}_3^{-1}\\
	&=\left[\left[\begin{array}{ccc}
		\frac{1}{m_1} &  & \\
		&  \ddots&\\
		& & \frac{1}{m_T} 	
	\end{array}\right]-\left[\begin{array}{ccc}
		\frac{\bar{m}}{m_1m_1}  &\ldots  &\frac{\bar{m}}{m_1m_T} \\
		&  \ddots&\\\frac{\bar{m}}{m_Tm_1} &\ldots &\frac{\bar{m}}{m_Tm_T} 	
	\end{array}\right] \right]\otimes Q^{-1}.
\end{aligned}   
$$
Hence,
\begin{equation}\label{y2_0}
	y_2=\operatorname{Vec}\left(\left[Q^{-1}\hat{R}^{1}_2,\ldots, Q^{-1} \hat{R}^{T}_2\right] \left[\left[\begin{array}{ccc}
		\frac{1}{m_1} &  & \\
		&  \ddots&\\
		& & \frac{1}{m_T} 	
	\end{array}\right]-\left[\begin{array}{ccc}
		\frac{\bar{m}}{m_1m_1}  &\ldots  &\frac{\bar{m}}{m_1m_T} \\
		&  \ddots&\\\frac{\bar{m}}{m_Tm_1} &\ldots &\frac{\bar{m}}{m_Tm_T} 	
	\end{array}\right] \right]      \right).
\end{equation} 
Recall that $\hat{R}_2=R_2-E_{2}^{*}E_1^{-1}R_1-\frac{1}{m}E_5R_3$. By the definition of $E_1,E_2,$ and $E_5$ in \eqref{Mat:WBP_AAT}, we have 
$$
\hat{R}_2=\left[\begin{array}{c}
	R_2^{1}-\frac{\sum_{j=1}^{m_1} (R_{1j}^1)-R_3}{m}\mathbf{1}_{m-1} \\
	\vdots\\
	R_2^{T}-\frac{\sum_{j=1}^{m_T} (R_{1j}^T)-R_3}{m}\mathbf{1}_{m-1}
\end{array}\right].
$$
It follows that 
\begin{equation*}
Q^{-1}\hat{R}^{t}_2=R_2^{t}+( \mathbf{1}_{m-1}^{\top}R^t_2- \mathbf{1}_{m_t}^{\top}R_1^t + R_3)\mathbf{1}_{m-1},\quad t=1,\ldots,T.
\end{equation*} 
Define $ \hat {y}_2^{t}:=R_2^{t}+( \mathbf{1}_{m-1}^{\top}R^t_2- \mathbf{1}_{m_t}^{\top}R_1^t + R_3)\mathbf{1}_{m-1}, t=1,\ldots,T,$
and Define $\hat {y}_2^{a}:=\sum_{t=1}^{T} \frac{\bar{m}}{m_{t}} \hat {y}_2^{t}$.
From \eqref{y2_0}, we have 
\begin{equation*}
    y_2^{t}=  \frac{1}{m_t}\left( \hat {y}_2^{t}-\hat {y}_2^{a} \right),\quad  t=1,\ldots,T. 
\end{equation*}
This completes the proof. 
\end{proof}
Now, we can summarize the procedure for solving equation $AA^{*}y=R$ in Algorithm \ref{alg:normal}.
\begin{algorithm}[H] 
	\caption{A linear time complexity solver for the linear system  $AA^{*}y=R$ }\label{alg:normal}
	\begin{algorithmic}
	    \State{Input: $R \in \mathbb{R}^{M}$.}
		\State {Step 1. Compute $y_2$ by \eqref{y2}.}
		\State {Step 2. Compute $y_1$ by \eqref{y1}.}
		\State {Step 3. Compute $y_3$ by \eqref{y3}.}
		\State {Output: $y=(y_1,y_2,y_3)\in \mathbb{R}^{M_1}\times\mathbb{R}^{M_2} \times\mathbb{R}$.}
	\end{algorithmic}
\end{algorithm}

\begin{proposition}\label{lem:compAAT}
 	The computational complexity of Algorithm \ref{alg:normal} in terms of flops is $7Tm+3\sum_{t=1}^{T}m_t+O(T)$.    
\end{proposition}
\begin{proof}
The complexity of each step in Algorithm \ref{alg:normal} can be summarized as follows: 
	$$
	\begin{array}{l}
		{\rm Step 1}: 6Tm+\sum_{t=1}^{T}m_t+O(T),\\
		{\rm Step 2}: Tm+2\sum_{t=1}^{T}m_t+O(T),\\
		{\rm Step 3}: O(T).
		\end{array}
	$$
	Summing them up, we obtain that the overall computational complexity of Algorithm \ref{alg:normal} is $7Tm+3\sum_{t=1}^{T}m_t+O(T)$.
\end{proof}

\begin{remark}
This complexity analysis can be extended to the OT problem. It is easy to see that the OT problem \eqref{model:OT} can be reformulated into the following form:
\begin{equation}\label{model:standLP_OT}
\begin{array}{ll}
\underset{x}{\min} & \langle{c}, {x}\rangle + \delta_{K}(x)\\
         \text { s.t. } &A x=b,
\end{array}
\end{equation}
where
\begin{enumerate}
    \item ${x}=\operatorname{vec}(X)$, $K$ is the nonnegative orthant with compatible dimension,
    \item $b=\left(\mathbf{a}^{v}; [\mathbf{0}_{m_u-1},I_{m_u-1}]\mathbf{a}^{u}\right), $
$c=\operatorname{vec}(\mathcal{D}(\mathcal{P}^{u},\mathcal{P}^{v}))$,
\item ${A}=\left[\begin{array}{c}
	I_{m_{v}} \otimes \mathbf{1}_{m_u}^{\top}\\
    \mathbf{1}_{m_{v}}^{\top} \otimes [\mathbf{0}_{m_u-1},I_{m_u-1}]
\end{array}\right].$
\end{enumerate}
From \citep[Lemma 7.1]{dantzig2003linear}, we know that $A$ is full row rank. The solution to the involved linear system 
\begin{equation*}
	AA^{*}y=AA^{*}\left[\begin{array}{c}
		y_1  \\
		y_2 
	\end{array}\right]=\left[\begin{array}{c}
		R_1  \\
		R_2 
	\end{array}\right]
\end{equation*}
can be obtained by 
\begin{equation}\label{AA*-OT}
    \left\{\begin{array}{ll}
y_1&=\frac{R_1}{m_u}+\frac{1} {m_v} \left(\frac{m_u-1}{m_u}\mathbf{1}_{m_v}^{\top}R_{1} -\mathbf{1}_{m_u-1}^{\top}R_{2}\right)	\mathbf{1}_{m_v}, \\
y_2&=\frac{R_2}{m_v}+ \frac{1} {m_v} \left(\mathbf{1}_{m_u-1}^{\top}R_{2}  -\mathbf{1}_{m_v}^{\top}R_{1}\right) \mathbf{1}_{m_u-1},\\
\end{array} \right.
\end{equation}
where $R:= (R_1,R_2)\in \mathbb{R}^{m_{v}}\times \mathbb{R}^{(m_{u}-1)}$. The computational complexity of solving this system by \eqref{AA*-OT} is $O(m_{u}+m_{v})$. 
\end{remark}

\section{A Halpern-Peaceman-Rachford Algorithm}
We start this section by introducing the HPR algorithm, which applies the Halpern iteration to the PR splitting method for finding a solution ${w}^{*} \in \mathbb{X}$ to the following inclusion problem:
\begin{equation}\label{pro:inclusion}
     0\in  \boldsymbol{M}_{1} w+\boldsymbol{M}_{2} w,
\end{equation}
where $\boldsymbol{M}_{1}: \mathbb{X} \rightrightarrows \mathbb{X}$, $\boldsymbol{M}_{2}: \mathbb{X} \rightrightarrows \mathbb{X}$, and $\boldsymbol{M}_1 + \boldsymbol{M}_2$ are all maximal monotone operators. We denote the zeros of $\boldsymbol{M}_1 + \boldsymbol{M}_2$ as ${\rm Zer}(\boldsymbol{M}_1 + \boldsymbol{M}_2)$.

For any given maximal monotone operator $\boldsymbol{M}: \mathbb{X} \rightrightarrows \mathbb{X}$, its resolvent $
\boldsymbol{J}_{\boldsymbol{M}}:=(\boldsymbol{I}+\boldsymbol{M})^{-1}$ is single-valued \citep{Minty62} and firmly nonexpansive, where $\boldsymbol{I}$ is the identity operator. Moreover, the reflected resolvent $
\boldsymbol{R}_{\boldsymbol{M}}:=2\boldsymbol{J}_{\boldsymbol{M}}-\boldsymbol{I}$ of $\boldsymbol{M}$ is nonexpansive  \citep[Corollary 23.11]{bauschke2011convex}. Let $\sigma > 0$ be any given parameter. Let $\eta^0 \in \mathbb{X}$ be any initial point. Then the PR splitting method \citep{lions1979splitting} solves \eqref{pro:inclusion} iteratively as
\begin{equation}\label{PR}
 \begin{aligned}
\eta^{k+1} = \mathbf{T}_{\sigma}^{\rm PR}(\eta^k) := \boldsymbol{R}_{\sigma \boldsymbol{M}_{1}}\circ\boldsymbol{R}_{\sigma \boldsymbol{M}_2}(\eta^{k}), \quad \forall k\geq 0,
\end{aligned}   
\end{equation}
where ``$\circ$'' is the operator composition. It is not difficult to verify that $\mathbf{T}_{\sigma}^{\rm PR}: \mathbb{X} \rightrightarrows \mathbb{X}$ is nonexpansive. If $\eta^* \in \operatorname{Fix}(\mathbf{T}_{\sigma}^{\rm PR})$, then $w^*= \boldsymbol{J}_{\sigma\boldsymbol{M}_{2}}(\eta^*)$ is a solution to \eqref{pro:inclusion} \citep{lions1979splitting}. 

On the other hand, the Halpern iteration  \citep{halpern1967fixed} is a popular method for finding a fixed point of the nonexpansive operator. When we apply the Halpern iteration to the PR splitting method for solving \eqref{pro:inclusion}, it has the following simple iterative scheme:
\begin{equation}\label{Halpern-PR}
  \eta^{k+1}:=\lambda_k \eta^{0}+\left(1-\lambda_k\right) \mathbf{T}_{\sigma}^{\rm PR}(\eta^k),\forall k\geq 0,  
\end{equation}
where $\eta^{0}\in\mathbb{X}$ is any given initial point and $\lambda_k\ \in [0,1]$ is a specified parameter. While it is not clear to us when the PR applied to the inclusion problem \eqref{pro:inclusion} converges, with suitable choices of $\{\lambda_k\}_{k=0}^{\infty}$,  the sequence $\{\eta^{k}\}_{k=0}^{\infty}$ generated by the Halpern iteration \eqref{Halpern-PR} converges to a fixed point of the nonexpansive operator $\mathbf{T}^{\rm PR}_{\sigma}$, which is a direct result of the following theorem. 
\begin{theorem}\citep[Theorem 2]{wittmann1992approximation}\label{Th:Halpern}
Let $D$ be a nonempty closed convex subset of $\mathbb{X}$, and let $\mathbf{T}: D \rightarrow D$ be a nonexpansive operator such that $\operatorname{Fix} (\mathbf{T}) \neq \varnothing$. Let $\{\lambda_k\}_{k=0}^{\infty}$ be a sequence in $[0,1]$ such that the following hold:
$$\lambda_{k} \rightarrow 0,\quad  \sum_{k=0}^{\infty} \lambda_{k}=+\infty, \quad \sum_{k=0}^{\infty}\left|\lambda_{k+1}-\lambda_{k}\right|<+\infty.
$$
Let $\eta^{0}\in D$ and set
$$
  \eta^{k+1}:=\lambda_k \eta^{0}+\left(1-\lambda_k\right) \mathbf{T} (\eta^k),\forall k\geq 0.   
$$
Then $\eta^{k} \rightarrow \Pi_{\operatorname{Fix}(\mathbf{T})}(\eta^0)$.
\end{theorem}

Recently, \cite{lieder2021convergence} showed that, if we take $\lambda_k = 1/(k+2)$ for $k\geq 0$, the Halpern iteration will give the following best possible convergence rate regarding the residual:
\[
\|\eta^k - \mathbf{T}^{\rm PR}_{\sigma}(\eta^k)\| \leq \frac{2\|\eta^0 - \bar{\eta}\|}{k+1}, \quad \forall k \geq 0 ~ and ~ \bar{\eta} \in \operatorname{Fix}(\mathbf{T}^{\rm PR}_{\sigma}).
\]
Thus, we choose $\lambda_k = 1/(k+2)$ for $k\geq 0$ for the HPR algorithm in this paper for solving \eqref{pro:inclusion}. The HPR algorithm is presented in Algorithm \ref{alg:HPR}. The convergence result of HPR for solving \eqref{pro:inclusion} is summarized in Corollary \ref{Corollary:convergence-HPR}. 

\begin{algorithm}[H] 
	\caption{An HPR algorithm for the problem \eqref{pro:inclusion}}
	\label{alg:HPR}
	\begin{algorithmic}
	    \State{Input: $ \eta^{0}\in \mathbb{X}$. For $k=0,1,\ldots$ }
	    \State{  
	   \begin{equation*}
 \begin{aligned}
&w^{k+1} = \boldsymbol{J}_{\sigma \boldsymbol{M}_{2}}(\eta^{k}), \\
&x^{k+1} = \boldsymbol{J}_{\sigma \boldsymbol{M}_{1}}\left(2 w^{k+1}-\eta^{k}\right), \\
&v^{k+1}= 2 x^{k+1} - \left(2w^{k+1} - \eta^k\right), \\
&\eta^{k+1}=\frac{1}{k+2} \eta^{0}+\frac{k+1}{k+2}v^{k+1}.
\end{aligned}   
\end{equation*}
	    }
	\end{algorithmic}
\end{algorithm}

\begin{corollary}\label{Corollary:convergence-HPR} 
Let $\boldsymbol{M}_1$, $\boldsymbol{M}_2$ and $\boldsymbol{M}_1 + \boldsymbol{M}_2$ be maximal monotone operators on $\mathbb{X}$ satisfying $\operatorname{Zer}(\boldsymbol{M}_1+\boldsymbol{M}_2)\neq \emptyset $. Let $\{w^{k}, x^{k}, v^{k}, \eta^{k}\}_{k=1}^{\infty}$ be the sequence generated by Algorithm \ref{alg:HPR}. Let $\eta^* = \Pi_{\operatorname{Fix}(\mathbf{T^{\rm PR}_{\sigma}})}(\eta^0)$ and $w^* = \boldsymbol{J}_{\sigma\mathbf{M}_2}(\eta^*)$. Then we have 
$$
 \eta^k \rightarrow \eta^*, \quad v^{k}\rightarrow \eta^{*},\quad 
    x^{k} \rightarrow w^{*}, \text{ and  } w^{k} \rightarrow w^{*},
$$
and $w^{*}$ is a solution to the problem \eqref{pro:inclusion}.
\end{corollary}
\begin{proof}
Since $\operatorname{Fix}(\mathbf{T}^{\rm PR}_{\sigma})$ is a closed convex set \citep[Corollary 4.24]{bauschke2011convex}, $\Pi_{\operatorname{Fix}(\mathbf{T}^{\rm PR}_{\sigma})}(\eta^0)$ is unique. The convergence of the sequence $\{\eta^{k}\}_{k=0}^{\infty}$ comes from Theorem \ref{Th:Halpern}. Since for all $k \geq 0$,
$$\eta^{k+1}=\frac{1}{k+2} \eta^{0}+\frac{k+1}{k+2}v^{k+1},$$
we can directly obtain $v^{k}\rightarrow \eta^{*}$. It follows from $x^{k+1}-w^{k+1}=1/2(v^{k+1}-\eta^{k}),\forall k\geq 0$ in Algorithm \ref{alg:HPR} that  
\begin{equation}
    x^{k}-w^{k} \rightarrow 0.
\end{equation} 
Since $\boldsymbol{J}_{\sigma \boldsymbol{M}_{2}}$ is nonexpansive and $w^{k+1} = \boldsymbol{J}_{\sigma \boldsymbol{M}_{2}}(\eta^k)$, 
we directly obtain $\lim_{k \to \infty} x^k = \lim_{k \to \infty} w^k = \boldsymbol{J}_{\sigma \boldsymbol{M}_{2}}(\eta^*) = w^*$. This completes the proof.
\end{proof}

\subsection{An HPR Algorithm for the Two-block Convex Programming Problems}
Now, we consider the following two-block convex optimization problem with linear constraints:
\begin{equation}\label{primal}
\begin{array}{cc}
      \min _{y \in \mathbb{Y} , s \in \mathbb{Z}} & f_{1}(y)+f_{2}(s)\\
     \text{subject to}& {B}_{1}y+{B}_{2}s=c,
\end{array}
\end{equation}
where $f_{1}: \mathbb{Y} \to (-\infty, +\infty]$ and $f_{2}: \mathbb{Z} \to (-\infty, +\infty]$ are proper closed convex functions, which may take extended value; $B_1: \mathbb{Y} \to \mathbb{X}$ and $B_2: \mathbb{Z} \to \mathbb{X}$ are two given linear mappings, and $c\in \mathbb{X}$ is a given vector. It is clear that the optimization problem \eqref{model:dualLP}, which is the dual problem of the LP, is a special case of \eqref{primal}. Given $\sigma>0$, the augmented Lagrange function corresponding to \eqref{primal} is 
\[
L_{\sigma}(y, s; x) := f_1(y) + f_2(s) + \langle x, B_1 y + B_2 s - c \rangle+\frac{\sigma}{2}\|B_1 y + B_2 s - c\|^2,
\]
where $x \in \mathbb{X}$ is the multiplier. The dual problem of \eqref{primal} is
\begin{equation}\label{dual}
  \max _{x \in \mathbb{X}}\left\{-f_{1}^{*}(-B_{1}^* x)-f_{2}^{*}(-B_{2}^* x)-\langle c, x\rangle\right\}, 
\end{equation}
where $f_1^*$ and $f_2^*$ are the Fenchel conjugate of $f_1$ and $f_2$, respectively; $B_1^*: \mathbb{X} \to \mathbb{Y}$ and $B_2^*: \mathbb{X} \to \mathbb{Z}$ are the adjoint of $B_1$ and $B_2$, respectively. 

Since $f_{1}$ and $f_{2}$ are proper closed convex functions, there exist two self-adjoint and positive semidefinite operators $\Sigma_{f_{1}}$ and $\Sigma_{f_{2}}$ such that for all $y, \hat{y} \in \operatorname{dom}(f_{1}), \phi \in \partial f_{1}(y)$, and $\hat{\phi} \in \partial f_{1}(\hat{y})$,
$$f_{1}(y) \geq f_{1}(\hat{y})+\langle\hat{\phi}, y-\hat{y}\rangle+\frac{1}{2}\|y-\hat{y}\|_{\Sigma_{f_{1}}}^{2} \text{ and } \langle \phi-\hat{\phi}, y-\hat{y}\rangle \geq\|y-\hat{y}\|_{\Sigma_{f_{1}}}^{2},$$
and for all $s, \hat{s} \in \operatorname{dom}(f_{2}), \varphi \in \partial f_{2}(s)$, and $\hat{\varphi} \in \partial f_{2}(\hat{s})$
$$\quad f_{2}(s) \geq f_{2}(\hat{s})+\langle\hat{\varphi}, s-\hat{s}\rangle+\frac{1}{2}\|s-\hat{s}\|_{\Sigma_{f_{2}}}^{2} \text{ and } \langle \varphi-\hat{\varphi}, s-\hat{s}\rangle \geq\|s-\hat{s}\|_{\Sigma_{f_{2}}}^{2}.$$
In this paper, we assume the following assumptions:
\begin{assumption}
\label{Assump: ass1}For $i = 1, 2$, the following conditions hold:
\begin{itemize}
    \item[(A1.1)] 
    $ \operatorname{ri}(\operatorname{dom} f_{i}^{*})\bigcap\operatorname{Range}(B_{i}^*) \neq \varnothing, 
    $
    \item[(A1.2)]
$\operatorname{ri}\left(\operatorname{dom}(f_{1}^{*} \circ (-B_{1}^{*}))\right) \cap \operatorname{ri}\left(\operatorname{dom}(f_{2}^{*} \circ (-B_{2}^{*}))\right) \neq \varnothing$.

\item[(A1.3)] The solution set of the optimization problem \eqref{primal} is nonempty.
\end{itemize}
\end{assumption}

\begin{assumption}
\label{ass: Assump-solvability}
Both $\Sigma_{f_1} + B_1^*B_1$ and $\Sigma_{f_2} + B_2^*B_2$ are positive definite.
\end{assumption}

Under Assumption \ref{Assump: ass1}, $(y^*, s^*) \in \mathbb{Y} \times \mathbb{Z}$ is a solution to the optimization problem \eqref{primal} and $x^* \in \mathbb{X}$ is a solution to the optimization problem  \eqref{dual} if and only if the following KKT system is satisfied:
\begin{equation}\label{KKT-OP}
   {B}_{1} y^* + {B}_{2} s^*=c,\quad -B_{2}^{*} x^* \in \partial f_{2}(s^*),\quad -B_{1}^{*} x^* \in \partial f_{1}(y^*).
\end{equation} 
The residual mapping associated with the KKT system is
\[
\mathcal{R}(y, s, x) = \left(
\begin{array}{c}
y - {\rm Prox}_{f_1}(y - B_1^*x)\\
s - {\rm Prox}_{f_2}(s - B_2^*x)\\
 c-B_1 y - B_2 s  
\end{array}
\right), \quad (y, s, x) \in \mathbb{Y} \times \mathbb{Z} \times \mathbb{X}.
\]
Then $(y^*, s^*, x^*)$ satisfies the KKT system \eqref{KKT-OP} if and only if $\mathcal{R}(y^*, s^*, x^*) = 0$.

If we take $\mathbf{M}_1 = \partial (f_1^* \circ (-B_1^*)) + c$ and $\mathbf{M}_2 = \partial (f_2^* \circ (-B_2^*))$, then, the inclusion problem \eqref{pro:inclusion} is equivalent to the optimization problem \eqref{dual}. As a result, we can apply Algorithm \ref{alg:HPR} to solve \eqref{dual}, which is presented in Algorithm \ref{alg:HPR-OP-0}.
\begin{algorithm}[H] 
	\caption{An HPR algorithm for solving the convex optimization problem \eqref{primal}}
	\label{alg:HPR-OP-0}
	\begin{algorithmic}[1]
		\State {Input: $y^0 \in {\rm dom}(f_1)$, 
  $x^0 \in \mathbb{X}$, and $\sigma>0$.}
  \State{Initialization: $\hat{x}^0 := x^0$, $\eta^0 := \hat{x}^0 + \sigma(B_1 y^0 - c)$.}
  \State{For $k=0,1, \ldots$}
		\State {Step 1. $s^{k+1} =\underset{s\in \mathbb{Z}}{\arg \min }\left\{f_{2}(s)+\langle\eta^{k}, {B}_{2} s\rangle+\frac{\sigma}{2}\|{B}_{2} s\|^{2}\right\}$.}
		\State{Step 2. $w^{k+1} =\eta^{k}+\sigma {B}_{2} s^{k+1}$.}
		\State {Step 3. ${y}^{k+1} =\underset{y\in \mathbb{Y}}{\arg \min }\left\{f_{1}(y)+\langle\eta^{k}+2 \sigma {B}_2 s^{k+1}, {B}_{1}{y}-{c}\rangle+\frac{\sigma}{2}\|{B}_{1}y -{c}\|^{2}\right\} $.
		} 
		\State {Step 4. $x^{k+1} =\eta^{k}+\sigma({B}_{1} {y}^{k+1}-c)+2 \sigma {B}_{2} s^{k+1}$.}		 
		\State {Step 5. ${v}^{k+1} =\eta^{k}+2\sigma\left({B}_{1} {y}^{k+1}+B_{2} s^{k+1}-c\right)$.}
            \State {Step 6. $ \eta^{k+1} =\frac{1}{k+2} \eta^{0}+\frac{k+1}{k+2}v^{k+1}$.}

	\end{algorithmic}
\end{algorithm}

In Algorithm \ref{alg:HPR-OP-0}, if we define
\begin{equation}\label{x&hatx}
   \hat{x}^{k+1}:=\eta^{k+1}-\sigma(B_{1}{y}^{k+1}-c), \quad \forall k\geq 0,
\end{equation}
we can discard the sequences $\{\eta^{k}\}_{k=0}^{\infty},\{w^{k}\}_{k=1}^{\infty},$ and $\{v^{k}\}_{k=1}^{\infty}$. This leads to Algorithm \ref{alg:HPR-OP}. We show the equivalence of Algorithm \ref{alg:HPR}, Algorithm \ref{alg:HPR-OP-0}, and Algorithm \ref{alg:HPR-OP} in Proposition \ref{prop:alg1-2-3}.
\begin{algorithm}[H] 
	\caption{An HPR algorithm for solving two-block convex optimization problem \eqref{primal}}
	\label{alg:HPR-OP}
	\begin{algorithmic}[1]
		\State {Input: $y^{0} \in {\rm dom}(f_1)$, $x^{0} \in \mathbb{X}$, and $\sigma>0$.} 
       \State{Initialization: $\hat{x}^{0} := x^0$.}
       \State{For $k=0,1, \ldots$}
		\State {Step 1. $s^{k+1} = \underset{s\in \mathbb{Z}}{\arg \min }\left\{L_{\sigma}(y^{k}, s; \hat{x}^{k})\right\}$.}
		\State{Step 2. $x^{k+\frac{1}{2}}=\hat{x}^k+\sigma (B_{1}{y}^{k}+B_{2}s^{k+1}-c) $.}
		\State {Step 3. ${{y}}^{k+1}= \underset{y\in \mathbb{Y}}{\arg \min }\left\{L_{\sigma}(y, s^{k+1}; x^{k+\frac{1}{2}})\right\} $.
		} 
		\State {Step 4. $x^{k+1}= {x}^{k+\frac{1}{2}}+\sigma (B_{1}{y}^{k+1}+B_{2}s^{k+1}-c)$.}		 
		\State {Step 5. $\hat{x}^{k+1}= \left(\frac{1}{k+2} \hat{x}^{0} + \frac{k+1}{k+2} x^{k+1}\right) + \frac{\sigma}{k+2}\left[(B_1{y}^{0}-c)    -(B_{1}{y}^{k+1}-c)\right] $.}
	\end{algorithmic}
\end{algorithm}

\begin{proposition}
\label{prop:alg1-2-3}
   Suppose that Assumption \ref{Assump: ass1} and Assumption \ref{ass: Assump-solvability} hold. Thus, all the subproblems in Algorithm \ref{alg:HPR-OP-0} and Algorithm \ref{alg:HPR-OP} are solvable. Take $\mathbf{M}_1 = \partial (f_1^* \circ (-B_1^*)) + c$ and $\mathbf{M}_2 = \partial (f_2^* \circ (-B_2^*))$. The following statements hold:
   \begin{enumerate}
       \item  If $\eta^{0}\in \mathbb{X}$ in Algorithm \ref{alg:HPR} is identical to the $\eta^0$ in Algorithm \ref{alg:HPR-OP-0}, then the sequence $\{w^{k},x^{k},v^{k},\eta^{k}\}_{k=1}^{\infty}$  generated by Algorithm \ref{alg:HPR} coincides with the sequence $\{w^{k},x^{k},v^{k},\eta^{k}\}_{k=1}^{\infty}$ generated by Algorithm \ref{alg:HPR-OP-0}. 
       \item Starting from the same initial points $y^0 \in {\rm dom}(f_1)$ and $x^0 \in \mathbb{X}$, the sequence $\{s^{k},y^{k},x^{k}\}_{k=1}^{\infty}$  generated by Algorithm \ref{alg:HPR-OP-0} coincides with the sequence $\{s^{k},y^{k},x^{k}\}_{k=1}^{\infty}$ generated by Algorithm \ref{alg:HPR-OP}.
   \end{enumerate}
\end{proposition}

\begin{proof}
See Appendix \ref{proof:alg1-2-3}. 
\end{proof}

Now, we prove the convergence of Algorithm \ref{alg:HPR-OP} for solving the optimization problem \eqref{primal}.
\begin{corollary}\label{Corollary:convergence-HPR-OP}
Assume that Assumption \ref{Assump: ass1} and Assumption \ref{ass: Assump-solvability} hold. Let $\{y^{k},s^{k},x^{k}\}_{k=1}^{\infty}$ and $\{\hat{x}^{k},x^{k+\frac{1}{2}}\}_{k=0}^{\infty}$ be the sequence generated by Algorithm \ref{alg:HPR-OP}. Then, we have 
$$y^{k} \rightarrow y^{*},  \quad  s^{k} \rightarrow s^{*}, \quad x^{k}\rightarrow x^{*},\quad  x^{k+\frac{1}{2}}\rightarrow x^{*}, \text{ and } \hat{x}^{k}\rightarrow x^{*},  $$
where $(y^*,s^*)$ is a solution to the problem \eqref{primal} and $x^{*}$ is a solution to the problem \eqref{dual}. 
\end{corollary}
\begin{proof}
First of all, Assumption \ref{Assump: ass1} ensures the existence of $(y^*, s^*, x^*)$ \citep[Corollary 28.2.2]{rockafellar1970convex}. Moreover, Assumption \ref{ass: Assump-solvability} ensures that both $\Sigma_{f_1} + \sigma B_1^*B_1$ and $\Sigma_{f_2} + \sigma B_2^*B_2$ are positive definite for any $\sigma > 0$. As a result, the subproblems of Algorithm \ref{alg:HPR-OP} are all solvable. Given any $y^{0} \in {\rm dom}(f_1)$ and $x^{0} \in \mathbb{X}$, let the sequence $\{w^{k},v^{k},\eta^{k}\}_{k=1}^{\infty}$ be generated by Algorithm \ref{alg:HPR-OP-0} with $\eta^{0}:={x}^{0}+\sigma (B_{1}y^{0}-c)$. From Proposition \ref{prop:alg1-2-3} and Corollary \ref{Corollary:convergence-HPR}, we know $\{{x}^{k}\}_{k=1}^{\infty}$ is convergent. Next, we prove the convergence of $\{s^k\}_{k=1}^{\infty}$. According to Algorithm \ref{alg:HPR-OP-0} and Proposition \ref{prop:alg1-2-3}, we have for all $k \geq 0$,
\begin{equation}\label{update-s}
  s^{k+1} =\underset{s}{\arg \min }\left\{f_{2}(s)+\left\langle\eta^{k}, {B}_{2} s\right\rangle+\frac{\sigma}{2}\|{B}_{2} s\|^{2}\right\}. 
\end{equation}
Denote $\hat{f_2}(s):=f_2(s)+\frac{\sigma}{2}\|{B}_{2} s\|^{2}$, which is a strongly convex function. Thus, $\hat{f}_{2}^{*}$ is essentially smooth \citep[Theorem 26.3]{rockafellar1970convex}. The first-order optimality condition of \eqref{update-s} implies
\begin{equation}\label{first-order}
  0\in \partial \hat{f}_2(s^{k+1})+B_{2}^{*}\eta^{k},\forall k\geq 0.
\end{equation}
Since $\hat{f}_2$ is a proper closed convex function, by \citep[Theorem 23.5]{rockafellar1970convex}, \eqref{first-order} is equivalent to
\begin{equation}
  s^{k+1}= \nabla \hat{f}_2^{*}(-B_{2}^{*}\eta^{k}),\forall k\geq 0.
\end{equation}
It follows from the convergence of $\{\eta^{k}\}_{k=1}^{\infty}$ by Corollary \ref{Corollary:convergence-HPR} and the continuity of $\nabla \hat{f}_2^{*}$ \citep[Theorem 25.5]{rockafellar1970convex} that  $\{s^{k}\}_{k=1}^{\infty}$ is convergent. Similarly, we can obtain the convergence of $\{y^{k}\}_{k=1}^{\infty}$.  
As a result, the convergence of $\{y^{k}\}_{k=1}^{\infty},\{{s}^{k}\}_{k=1}^{\infty}$, and $\{{x}^{k}\}_{k=1}^{\infty}$ yields the convergence of $\{\hat{x}^{k},x^{k+\frac{1}{2}}\}_{k=0}^{\infty}$. 

Assume that $(y^*,s^*,x^*)$ is the limit point of sequence $\{y^{k},s^{k},x^{k}\}_{k=1}^{\infty}$. Since for all $k \geq 0$,
$${v}^{k+1} =\eta^{k}+2\sigma({B}_{1} {{y}}^{k+1}+{B}_{2} s^{k+1}-{c})$$
from Algorithm \ref{alg:HPR-OP-0},
we can obtain \begin{equation}\label{Lemma:conv-HPR-OP-1}
    B_{1}y^{*}+B_{2}s^{*}=c
\end{equation}
from Corollary \ref{Corollary:convergence-HPR} by taking the limit. It follows that $\lim_{k \to \infty}x^{k+\frac{1}{2}}=\lim_{k \to \infty} \hat{x}^{k}= x^{*}$.
By Algorithm \ref{alg:HPR-OP-0}, we have $$
-B_{2}^{*}w^{k+1}\in \partial f_{2}(s^{k+1}),\quad -B_{1}^{*}x^{k+1}\in \partial f_{1}(y^{k+1}),\quad w^{k+1} =\eta^{k}+\sigma {B}_{2} s^{k+1},\quad \forall k\geq 0. 
$$
It follows from Corollary \ref{Corollary:convergence-HPR} and \citep[Theorem 24.4]{rockafellar1970convex} that 
\begin{equation}\label{Lemma:conv-HPR-OP-2}
-B_{2}^{*}w^{*}\in \partial f_{2}(s^{*}),\quad -B_{1}^{*}x^{*}\in \partial f_{1}(y^{*}),\quad x^{*}=w^{*}.     
\end{equation}
Together with \eqref{Lemma:conv-HPR-OP-1}, we have 
\begin{equation*}
-B_{2}^{*}x^{*}\in \partial f_{2}(s^{*}),\quad -B_{1}^{*}x^{*}\in \partial f_{1}(y^{*}),\quad 
 B_{1}y^{*}+B_{2}s^{*}-c=0.    
\end{equation*}
This completes the proof \citep[Corollary 28.3.1]{rockafellar1970convex}. 
\end{proof}

\begin{remark}
Assumption \ref{ass: Assump-solvability} ensures the solvability of the subproblems of Algorithm \ref{alg:HPR-OP} and it is necessary for the convergence of the sequence $\{(y^k, s^k)\}_{k \geq 1}$. Assumption \ref{ass: Assump-solvability} holds automatically if either $B_i$ is injective or $f_i$ is strongly convex (for $i = 1, 2$). \cite{han2018linear} gives an example where the sufficient conditions just mentioned fail to hold, but Assumption \ref{ass: Assump-solvability} can still hold.
\end{remark}

\subsection{Iteration Complexity Analysis}
In this subsection, we analyze the iteration complexity of  Algorithm \ref{alg:HPR-OP} for solving the optimization problem \eqref{primal}. Given the sequence $\{y^k,s^k\}_{k=1}^{\infty}$ generated by Algorithm \ref{alg:HPR-OP}, define 
 \begin{equation}\label{fun:h}
      h(y^{k+1},s^{k+1}):=f_{1}(y^{k+1})+f_{2}(s^{k+1})-f_{1}(y^*)-f_{2}(s^*), \forall k\geq 0,
 \end{equation}
 where $(y^*, s^*)$ is the limit point of the sequence $\{y^k,s^k\}_{k=1}^{\infty}$. Define
 $$d_{f_{1}}(w):=f_{1}^{*}(-B_{1}^{*} w)+\langle w, c\rangle  \quad \text{and} \quad d_{f_{2}}(w):=f_{2}^{*}(-B_{2}^{*} w), \quad \forall w \in \mathbb{X}. $$ 
We first prove the following lemma before we establish the iteration complexity of the HPR. 
 
 \begin{lemma}\label{Lemma:bound-h}
Suppose that Assumption \ref{Assump: ass1} and Assumption \ref{ass: Assump-solvability} hold. Take $\mathbf{M}_1 = \partial (f_1^* \circ (-B_1^*)) + c$ and $\mathbf{M}_2 = \partial (f_2^* \circ (-B_2^*))$. Let $\{y^{k},s^{k}\}_{k=1}^{\infty}$ be the sequence generated by Algorithm \ref{alg:HPR-OP} and the sequence $\{w^{k},x^{k},v^{k},\eta^{k}\}_{k=1}^{\infty}$ be generated by Algorithm \ref{alg:HPR-OP-0} with $\eta^{0}:=\hat{x}^{0}+\sigma (B_{1}y^{0}-c)$. Then, we have for all $k \geq 0$,
{\begin{equation}
\label{eq: h-equation}
\small
4 \sigma h(y^{k+1}, s^{k+1})=-4 \sigma \left(d_{f_{1}}(x^{k+1})+d_{f_{2}}(w^{k+1})-d_{f_{1}}(w^{*})-d_{f_{2}}(w^{*})\right) 
+\|\eta^{k}-v^{k+1}\|^{2}+2\langle \eta^{k}-v^{k+1}, v^{k+1}\rangle.
\end{equation}}
Moreover, the following bounds hold for all $k \geq 0$:
\begin{equation}
\label{eq: h-upper_bound}
    4 \sigma h(y^{k+1}, s^{k+1}) \leq \|\eta^{k}-\left(\eta^{*}-w^{*}\right)\|^2-\|v^{k+1}-\left(\eta^{*}-w^{*}\right)\|^2,
\end{equation}
and
\begin{equation}
\label{eq: h-lower_bound}
 h(y^{k+1}, s^{k+1}) \geq\langle B_{1}y^{k+1}+B_{2} s^{k+1}-c,-w^{*}\rangle,  
\end{equation}
where $\eta^* = \Pi_{\operatorname{Fix}(\mathbf{T^{\rm PR}_{\sigma}})}(\eta^0)$ and $w^{*}={J}_{\sigma M_{2}}\left(\eta^{*}\right)$.
 \end{lemma}
\begin{proof}
We first prove the equation \eqref{eq: h-equation}. For all $k \geq 0$, from Algorithm \ref{alg:HPR-OP-0} we have
\begin{equation}\label{subgradient}
-B_{2}^{*}w^{k+1}\in \partial f_{2}(s^{k+1}),\quad -B_{1}^{*}x^{k+1}\in \partial f_{1}(y^{k+1}).    
\end{equation}
It follows from \citep[Theorem 23.5]{rockafellar1970convex} that for all $k\geq 0$,
$$
d_{f_{1}}(x^{k+1})=\langle x^{k+1}, c\rangle -\langle B_{1}^{*}x^{k+1},y^{k+1}\rangle-f_{1}(y^{k+1}),  \quad  d_{f_{2}}(w^{k+1})= -\langle B_{2}^{*}w^{k+1},s^{k+1}\rangle-f_{2}(s^{k+1}).
$$
Therefore, for all $k\geq 0$,
$$
\begin{array}{ll}
&-d_{f_{1}}(x^{k+1})-d_{f_{2}}(w^{k+1})\\
=& f_{1}(y^{k+1})+f_{2}(s^{k+1})-\langle x^{k+1}, c\rangle +\langle B_{1}^{*}x^{k+1},y^{k+1}\rangle+\langle B_{2}^{*}w^{k+1},s^{k+1}\rangle\\
=&f_{1}(y^{k+1})+f_{2}(s^{k+1})+\langle x^{k+1},B_{1}y^{k+1}+B_{2}s^{k+1}-c\rangle -\langle x^{k+1}-w^{k+1},B_{2}s^{k+1}\rangle.
\end{array}
$$
From Algorithm \ref{alg:HPR-OP-0}, we also have for all $k\geq 0$, $x^{k+1}-w^{k+1}=\sigma\left(B_{1}y^{k+1}+B_{2}s^{k+1}-c\right) = \frac{1}{2}(v^{k+1} - \eta^k)$. Thus, for all $k\geq 0$,
\begin{equation}\label{df-dg}
    -d_{f_{1}}(x^{k+1})-d_{f_{2}}(w^{k+1})=f_{1}(y^{k+1})+f_{2}(s^{k+1})+\frac{1}{\sigma}\langle x^{k+1}-\sigma B_{2}s^{k+1},x^{k+1}-w^{k+1}\rangle.
\end{equation}
Recall that $x^{k+1}-\sigma B_{2}s^{k+1}=\eta^{k}+\sigma(B_{1}y^{k+1}+B_{2}s^{k+1}-c)=\eta^{k}+x^{k+1}-w^{k+1}$.
Therefore, for all $k\geq 0$,
$$
x^{k+1}-\sigma B_{2}s^{k+1} = \eta^{k}+\frac{1}{2 }(v^{k+1}-\eta^{k}) = \frac{1}{2 }(\eta^{k}-v^{k+1})+v^{k+1},
$$
and 
\begin{equation}\label{inner}
  \begin{aligned}
\frac{1}{\sigma}\langle x^{k+1}-\sigma B_{2}s^{k+1},x^{k+1}-w^{k+1}\rangle &=\frac{1}{\sigma}\left\langle \frac{1}{2 }(\eta^{k}-v^{k+1})+v^{k+1}
,\frac{1}{2}(v^{k+1}-\eta^{k}) \right\rangle \\
&=-\frac{1}{4 \sigma}\|v^{k+1}-\eta^{k}\|^{2}-\frac{1}{2 \sigma }\langle \eta^{k}-v^{k+1}, v^{k+1}\rangle.
\end{aligned}  
\end{equation}
According to Corollary \ref{Corollary:convergence-HPR}, $w^{*}$ is a solution to problem \eqref{pro:inclusion} and problem \eqref{dual} under Assumption 1. We also have the limit point $(y^*, s^*)$ of the sequence $\{y^k,s^k\}_{k=1}^{\infty}$ is a solution to problem \eqref{primal} from Corollary \ref{Corollary:convergence-HPR-OP}. It follows from \citep[Theorem 28.4]{rockafellar1970convex} that
\begin{equation}\label{df*+dg*}
 -d_{f_{1}}\left(w^{*}\right)-d_{f_{2}}\left(w^{*}\right)=f_{1}\left(y^{*}\right)+f_{2}\left(s^{*}\right).
\end{equation}
 Therefore, it follows from \eqref{df-dg}, \eqref{inner}, and \eqref{df*+dg*} that the equation \eqref{eq: h-equation} holds.

Next, we prove the upper bound of $h(y^{k+1}, s^{k+1})$ for all $k \geq 0$. According to Algorithm \ref{alg:HPR} and Corollary \ref{Corollary:convergence-HPR} that $(\eta^* - w^*)/\sigma \in \partial d_{f_2}(w^*)$ and $(w^* - \eta^*)/\sigma \in \partial d_{f_1}(w^*)$. Hence, 
\[
d_{f_1}(x^{k+1}) - d_{f_1}(w^*) \geq \langle (w^* - \eta^*)/\sigma, x^{k+1} - w^* \rangle, \quad d_{f_2}(w^{k+1}) - d_{f_1}(w^*) \geq \langle (\eta^* - w^*)/\sigma, w^{k+1} - w^* \rangle.
\]
Thus, 
{
\begin{equation*}
\small
  -4 \sigma \left(d_{f_{1}}(x^{k+1})+d_{f_{2}}(w^{k+1})-d_{f_{1}}(w^{*})-d_{f_{2}}(w^{*})\right)\leq 4 \langle x^{k+1}-w^{k+1}, \eta^{*}-w^{*} \rangle=2\langle v^{k+1}-\eta^{k}, \eta^{*}-w^{*} \rangle.  
\end{equation*}
}
Together with \eqref{eq: h-equation}, we obtain the inequality \eqref{eq: h-upper_bound}.

Finally, according to \eqref{KKT-OP} and Corollary \ref{Corollary:convergence-HPR}, we have for all $k \geq 0$,
\[
f_1(y^{k+1}) - f_1(y^*) \geq \langle - B_1^* w^*, y^{k+1} - y^* \rangle, \quad f_2(s^{k+1}) - f_2(s^*) \geq \langle - B_2 w^*, s^{k+1} - s^* \rangle.
\]
Thus, the inequality \eqref{eq: h-lower_bound} holds since $B_1 y^* + B_2 s^* = c$. This completes the proof of the lemma.
\end{proof}

Now, we are ready to prove the convergence rate of the HPR algorithm for solving the optimization problem \eqref{primal}. 

\begin{theorem}\label{Th:complexity}
Suppose that Assumption \ref{Assump: ass1} and Assumption \ref{ass: Assump-solvability} hold. Take $\mathbf{M}_1 = \partial (f_1^* \circ (-B_1^*)) + c$ and $\mathbf{M}_2 = \partial (f_2^* \circ (-B_2^*))$. Let $\{y^{k+1},s^{k+1},x^{k+1}\}_{k=0}^{\infty}$ be the sequence generated by Algorithm \ref{alg:HPR-OP}. Let $(y^{*},x^{*})$ be the limit point of $\{(y^{k+1},x^{k+1})\}_{k=0}^{\infty}$. Then for all $k \geq 0$, we have the following bounds:
\begin{equation}
\label{eq: complexity-bound-KKT}
\|\mathcal{R}(y^{k+1},s^{k+1},x^{k+1})\| \leq\frac{1}{k+1} \left(\frac{\sigma \|B_2^{*}\| + 1}{\sigma}\left(\left\|{x}^{0}-x^{*} \right\| + \sigma \left\| B_{1}y^{0} - B_{1}y^{*}\right\|\right)\right).
\end{equation}
Moreover, we have for $k\geq 0$ that
$$
h(y^{k+1}, s^{k+1}) \geq -\frac{1}{k+1}\left(\frac{\|x^{*}\|}{\sigma} \left( \left\|{x}^{0}-x^{*}\right\| +\sigma \left\|B_{1}y^{0}-B_{1}y^{*}\right\|\right)\right)
$$ 
and
$$
h(y^{k+1}, s^{k+1}) \leq \frac{1}{k+1} \left(\frac{\left(\left\|{x}^{0}-x^{*} \right\| +\sigma \left\| B_{1}y^{0}-B_{1}y^{*}\right\|\right)^2 + \left\|x^{*}\right\|\left(\left\|{x}^{0}-x^{*} \right\| + \sigma \left\| B_{1}y^{0}-B_{1}y^{*}\right\|\right)}{\sigma}\right).
$$

\end{theorem}
\begin{proof}
Take the sequence $\{w^{k},v^{k},\eta^{k}\}_{k=1}^{\infty}$ generated by Algorithm \ref{alg:HPR-OP-0} with $\eta^{0}:={x}^{0}+\sigma (B_{1}y^{0}-c)$. From Corollary \ref{Corollary:convergence-HPR}, Proposition \ref{prop:alg1-2-3}, and Corollary \ref{Corollary:convergence-HPR-OP}, we have the sequence $\{\eta^{k}\}_{k=1}^{\infty}$ converges to $\eta^{*}:= \Pi_{\operatorname{Fix}(\mathbf{T^{\rm PR}_{\sigma}})}(\eta^0)$, and $\eta^* = x^{*} +\sigma (B_{1}y^{*}-c)$. Hence, we have 
$$\|\eta^{0}-\eta^{*}\|=\left\|{x}^{0}-x^{*} +\sigma (B_{1}y^{0}-B_{1}y^{*})\right\|.$$
It together with the convergence rate of Halpern-Iteration \citep[Thoerem 2.1]{lieder2021convergence} implies that
\begin{equation}\label{Halpern}
  \|v^{k+1}-\eta^{k}\|\leq \frac{2\|\eta^{0}-\eta^{*}\|}{k+1}= \frac{2\left\|{x}^{0}-x^{*} +\sigma (B_{1}y^{0}-B_{1}y^{*})\right\|}{k+1}, \quad \forall k\geq 0.
\end{equation}
According to Algorithm \ref{alg:HPR-OP-0} and Proposition \ref{prop:alg1-2-3}, we have 
\begin{equation}\label{complexity:feasibility}
  \|(B_{1}{y}^{k+1}+B_2{s}^{k+1}-c)\|=\frac{1}{2\sigma}\|v^{k+1}-\eta^{k}\|\leq \frac{\left\|{x}^{0}-x^{*} +\sigma (B_{1}y^{0}-B_{1}y^{*})\right\|}{\sigma(k+1)},  \quad \forall k\geq 0.
\end{equation}
Moreover, from Algorithm \ref{alg:HPR-OP-0}, we also have for all $k \geq 0$,
$$
-B_{2}^{*}w^{k+1}\in \partial f_{2}(s^{k+1}),\quad -B_{1}^{*}x^{k+1}\in \partial f_{1}(y^{k+1}),  
$$
and 
$$ 
x^{k+1}=w^{k+1}+\sigma ({B}_{1} y^{k+1}+{B}_{2} s^{k+1}-c).
$$
This yields that for all $k \geq 0$
\begin{equation*}
        y^{k+1}-{\rm Prox}_{f_1}(y^{k+1}-B_{1}^{*}x^{k+1})=0,  
\end{equation*}
and 
\begin{equation*}
\begin{array}{ll}
     &    \|s^{k+1}-{\rm Prox}_{f_2}(s^{k+1}-B_{2}^{*}x^{k+1})\|\\
     = &\|{\rm Prox}_{f_2}(s^{k+1}-B_{2}^{*}w^{k+1})               -{\rm Prox}_{f_2}(s^{k+1}-B_{2}^{*}x^{k+1})\|\\ 
    \leq & \| \sigma B_{2}^{*}({B}_{1} y^{k+1}+{B}_{2} s^{k+1}-c) ) \| \\
    \leq & \sigma \| B_{2}^{*} \|\|{B}_{1} y^{k+1}+{B}_{2} s^{k+1}-c\|.
\end{array}
\end{equation*}
As a result, for $k \geq 0$, we have 
\[
\begin{array}{lcl}
\|\mathcal{R}(y^{k+1}, s^{k+1}, x^{k+1})\| &\leq& \sqrt{\sigma^2\|B_2^*\|^2 + 1} \|B_1 y^{k+1} + B_2 s^{k+1} - c\|\\
& \stackrel{\eqref{complexity:feasibility}
} \leq & (\sigma \|B_2^{*}\| + 1) \frac{\left\|{x}^{0}-x^{*} +\sigma (B_{1}y^{0}-B_{1}y^{*})\right\|}{\sigma (k+1)}\\
&\leq& \frac{1}{k+1} \left(\frac{\sigma \|B_2^{*}\| + 1}{\sigma}\left(\left\|{x}^{0}-x^{*} \right\| + \sigma \left\| B_{1}y^{0} - B_{1}y^{*}\right\|\right)\right).
\end{array}
\]

Next, we show the convergence rate regarding the objective function value. We will first show that 
\begin{equation}
\label{eq: boundness-vk}
\|v^{k}-\eta^{*}\|\leq \|\eta^{0}-\eta^{*}\|, \quad \forall k \geq 1.
\end{equation}
We prove it by induction. When $k=1$, we have $\|v^1-\eta^{*}\|\leq \|\eta^{0}-\eta^{*}\|$ by the non-expansiveness of $R_{\sigma \boldsymbol{M}_{1}}\circ R_{\sigma \boldsymbol{M}_{2}}$. Assume \eqref{eq: boundness-vk} holds for some $k \geq 1$. Then, we have 
$$
\begin{array}{ll}
     \|v^{k+1}-\eta^{*}\|&\leq \|\eta^{k}-\eta^{*}\|  \\
                         &= \|\frac{1}{k+1}\eta^{0}+\frac{k}{k+1}v^{k}-\eta^{*}\| \\
                         &\leq \frac{1}{k+1}\|\eta^{0}-\eta^{*}\|+\frac{k}{k+1}\|v^{k}-\eta^{*}\|\\
                         &\leq \|\eta^{0}-\eta^{*}\|,
\end{array}
$$
which proves \eqref{eq: boundness-vk}.
It follows that
\begin{equation}\label{Th1-1}
\|v^{k+1}-(\eta^{*}-w^{*})\|  \leq\|\eta^{0}-\eta^{*}\|+\|w^{*}\|, \quad \forall k\geq 0.    
\end{equation}
Similarly, we have 
\begin{equation}\label{Th1-2}
\|\eta^{k}-(\eta^{*}-w^{*})\| \leq\|\eta^{0}-\eta^{*}\|+\|w^{*}\|, \quad \forall k\geq 0.     
\end{equation}
From \eqref{eq: h-upper_bound}, we have for all $k \geq 0$, 
$$
\begin{array}{lcl}
h(y^{k+1}, s^{k+1}) & \leq & \frac{1}{4 \sigma } \left(\|\eta^{k}-(\eta^{*}-w^{*})\|^2-\|v^{k+1}-(\eta^{*}-w^{*})\|^2\right) \\
&=& \frac{1}{4 \sigma } \left\langle \eta^{k}+v^{k+1}-2\left(\eta^{*}-w^{*}\right)   ,  \eta^{k}-v^{k+1}   \right\rangle\\
& \leq & \frac{1}{4\sigma}\left(\|\eta^{k}-(\eta^{*}-w^{*})\|+\|v^{k+1}-(\eta^{*}-w^{*})\| \right)\|\eta^{k}-v^{k+1}\|   \\
& \stackrel{\eqref{Th1-1}\eqref{Th1-2} }{\leq} & \frac{1}{2\sigma}\left(\|\eta^{0}-\eta^{*}\|+\|w^{*}\| \right)\|\eta^{k}-v^{k+1}\|      \\
& \stackrel{\eqref{Halpern}}{\leq} &
\frac{\left(\left\|{x}^{0}-x^{*} +\sigma (B_{1}y^{0}-B_{1}y^{*})\right\|+\|w^{*}\| \right)\left\|{x}^{0}-x^{*} +\sigma (B_{1}y^{0}-B_{1}y^{*})\right\|}{\sigma(k+1)}\\
& \leq & \frac{1}{k+1} \left(\frac{\left(\left\|{x}^{0}-x^{*} \right\| +\sigma \left\| B_{1}y^{0}-B_{1}y^{*}\right\|\right)^2 + \left\|w^{*}\right\|\left(\left\|{x}^{0}-x^{*} \right\| + \sigma \left\| B_{1}y^{0}-B_{1}y^{*}\right\|\right)}{\sigma}\right).
\end{array}
$$
For the lower bound, from \eqref{eq: h-lower_bound} and Algorithm \ref{alg:HPR-OP-0}, we have for all $k\geq 0$,
$$
\begin{aligned}
h(y^{k+1}, s^{k+1}) &\geq\left\langle B_{1}y^{k+1}+B_{2}s^{k+1}-c,-w^{*}\right\rangle\\
                    &{=}\frac{1}{2\sigma}\langle v^{k+1}-\eta^{k},-w^{*}\rangle\\
                    &\geq -\frac{1}{2\sigma}\|v^{k+1}-\eta^{k}\|\|w^{*}\|\\
                    &\stackrel{\eqref{Halpern}}{\geq} -\frac{\left\|{x}^{0}-x^{*} +\sigma (B_{1}y^{0}-B_{1}y^{*})\right\|\|w^{*}\| }{\sigma (k+1)}\\
                    & \geq -\frac{1}{k+1}\left(\frac{\|w^{*}\|}{\sigma} \left( \left\|{x}^{0}-x^{*}\right\| +\sigma \left\|B_{1}y^{0}-B_{1}y^{*}\right\|\right)\right).
\end{aligned}
$$
From Corollary \ref{Corollary:convergence-HPR} and Proposition \ref{prop:alg1-2-3}, we have $w^{*}=x^{*}$, which completes the proof.
\end{proof}

\begin{remark} Here, we make some remarks.
\begin{itemize}
    \item[1.] We emphasize that $\|B_2^*\|$ is the spectral norm of $B_2^*$ and $\|B_2^{*}\| = 1$ for the WBP. 
    \item[2.] If all the subproblems in Algorithm \ref{alg:HPR-OP} are solvable, the $O(1/k)$ convergence rate in terms of the KKT residual and the primal objective function value gap still holds true in the theorem without Assumption \ref{ass: Assump-solvability}. This assumption is only necessary for the convergence of the sequence $\{(y^k, s^k)\}_{k \geq 1}$.
 \end{itemize}   
\end{remark}

\subsection{A Fast Implementation of the HPR for Solving the WBP}
In this section, we will present a fast implementation of the HPR for solving the WBP. In particular, we will show that, for the WBP, all subproblems of the HPR have a closed-form solution. We will also analyze the per-iteration computational complexity of the HPR for solving the WBP. An HPR for solving the linear programming problem \eqref{model:dualLP} is presented in Algorithm \ref{alg:HPR-WBP}, which is a direct application of Algorithm \ref{alg:HPR-OP}. 
\begin{algorithm}[H] 
	\caption{An HPR algorithm for solving dual linear programming \eqref{model:dualLP} }
	\label{alg:HPR-WBP}
	\begin{algorithmic}[1]
		\State {Input: $x^{0}\in \mathbb{R}^{N}$, $y^{0} \in \mathbb{R}^{M}$, and $\sigma>0$.}
  \State{Initialization: $\hat{x}^{0} := x^0$, $s^0 := c - A^*y^0$.}
  \State{For $k = 0, 1, \dots$}
		\State {Step 1. Compute $ s^{k+1}=\Pi_{K}\left(c-A^{*}y^{k}-\hat{x}^{k}/\sigma\right).$  }
		\State {Step 2. Compute $ x^{k+\frac{1}{2}}=\hat{x}^{k}+ \sigma\left(s^{k+1}+A^{*} y^{k}-c\right).$}
		\State {Step 3. Compute $y^{k+1}$ by applying Algorithm \ref{alg:normal} to solve the following linear system:
  \begin{equation}
  \label{eq: linear-system-HPR-WBP}
 A A^{*}y = b / \sigma-A({x}^{k+\frac{1}{2}} / \sigma+s^{k+1}-c).
  \end{equation}}
		\State {Step 4. Compute
			${x}^{k+1}={x}^{k+\frac{1}{2}}+ \sigma\left(s^{k+1}+A^{*} y^{k+1}-c\right).
			$}
		\State {Step 5. Compute $\hat{x}^{k+1}= \left(\frac{1}{k+2} \hat{x}^{0} + \frac{k+1}{k+2} x^{k+1}\right) + \frac{\sigma}{k+2}\left[(A^{*}{y}^{0}-c)    -(A^{*}{y}^{k+1}-c)\right].$}
	\end{algorithmic}
\end{algorithm}

Now, we are ready to present the per-iteration computational complexity $O(m\sum_{t=1}^T m_t)$ of Algorithm \ref{alg:HPR-WBP} in the next lemma.  This implies that the per-iteration computational complexity of the HPR algorithm is comparable to the IBP algorithm \citep{benamou2015iterative} for solving the WBP.
\begin{lemma}\label{compleixty:per}
The per-iteration computational complexity of Algorithm \ref{alg:HPR-WBP} in terms of flops is $26m\sum_t m_t+O(Tm+\sum_{t=1}^{T}m_t)$.
\end{lemma}
\begin{proof} 
The computational complexity of $k-$th ($\forall k\geq 0$) step of Algorithm \ref{alg:HPR-WBP} can be summarized as follows:  
\par{Step 1.} Since $s^{k+1}:= \Pi_{K}(c-A^{*}y^{k}-\hat{x}^{k}/\sigma) = \max (0,c-A^{*}y^{k}-\hat{x}^{k}/\sigma)$, and the matrix $A$ has $2(m-1)\sum_{t=1}^T m_t+T(m-1)+m$ non-zero entries, the complexity for updating $s^{k+1}$ is $7m\sum_{t=1}^{T}m_t+O(mT)$. Here, we omit $O(Tm)$ since $m_t \geq 1$ for $1 \leq t \leq T$. 
\par{Step 3.} The complexity of forming $R=\left(b / \sigma-A(x^{k+\frac{1}{2}} / \sigma+s^{k+1}-c)\right)$ is $6m\sum_{t=1}^Tm_t+O(Tm)$ due to the structure of $A$. From Proposition \ref{lem:compAAT}, we know that the complexity for solving the equation $AA^{*}y=R$ is $O(Tm+\sum_{t=1}^{T}m_t)$. Therefore, the complexity of Step 3 is $6m\sum_{t=1}^{T}m_t+O(Tm+\sum_{t=1}^{T}m_t)$. 
\par{Steps 2, 4, and 5.} By some simple calculation, we can see the complexity for these three steps is $13m\sum_{t=1}^{T}m_t+O(Tm)$. 

Summing them up, we know the per-iteration computational complexity of Algorithm \ref{alg:HPR-WBP} in terms of flops is $26m\sum_t m_t+O(Tm+\sum_{t=1}^{T}m_t)$.
\end{proof}

Now, we can obtain the overall computational complexity of the HPR for solving the WBP.

\begin{theorem}\label{complexity}
 Let $\{y^{k},s^{k},x^{k}\}_{k=0}^{\infty}$ be the sequence generated by Algorithm \ref{alg:HPR-WBP}.
 For any given tolerance $\varepsilon>0$, HPR needs at most 
 \[
 \frac{1}{\varepsilon} \left(\frac{1 + \sigma}{\sigma}\left(\left\|x^0 - x^*\right\| + \sigma\left\|s^0 - s^*\right\|\right)\right) - 1
 \]
iterations to return a solution to WBP such that the KKT residual $\|\mathcal{R}(y^{k+1},s^{k+1},x^{k+1}) \| \leq \varepsilon$, where $(x^{*},s^{*})$ is the limit point of the sequence $\{x^{k},s^{k}\}_{k=0}^{
\infty}$. In particular, the overall computational complexity of HPR to achieve this accuracy in terms of flops is 
$$O\left(\left(\frac{1 + \sigma}{\sigma}\left(\left\|x^{0}-x^{*}\right\| + \sigma\left\|s^{0}-s^{*}\right\|\right)\right) \frac{m\sum_{t=1}^{T} m_t}{\varepsilon}\right).$$
\end{theorem}
\begin{proof}
Since $B_2 = I$, it follows from Theorem \ref{Th:complexity} that, for the linear programming \eqref{model:dualLP}, we have 
    \[
\|\mathcal{R}(y^{k+1},s^{k+1},x^{k+1})\| \leq  \frac{1}{k + 1}\left(\frac{1 + \sigma}{\sigma}\left(\left\|x^{0}-x^{*}\right\| + \sigma\left\|s^{0}-s^{*}\right\|\right)\right), \quad \forall k\geq 0.
\]
Therefore, $\|\mathcal{R}(y^{k+1},s^{k+1},x^{k+1})\| \leq \varepsilon$ if
$$
k\geq \frac{1}{\varepsilon}\left(\frac{1 + \sigma}{\sigma}\left(\left\|x^{0}-x^{*}\right\| + \sigma\left\|s^{0}-s^{*}\right\|\right)\right) - 1.
$$
According to Lemma \ref{compleixty:per}, the per-iteration computational complexity in terms of the flops of the HPR algorithm for solving the WBP is $O(m \sum_{t = 1}^T m_t)$. The overall computational complexity of HPR to achieve this accuracy in terms of flops is $$O\left(\left(\frac{1 + \sigma}{\sigma}\left(\left\|x^{0}-x^{*}\right\| + \sigma\left\|s^{0}-s^{*}\right\|\right)\right) \frac{m\sum_{t=1}^{T} m_t}{\varepsilon}\right).$$ This completes the proof.
\end{proof}

\begin{remark}
Actually, we can reformulate the primal LP model of the WBP \eqref{model:standLP} equivalently as
\begin{equation}
\label{eq: WBP_LP_primal_equiv}
\begin{array}{ll}
\min_{x \in \mathbb{R}^N, z \in \mathbb{R}^N} & (\langle c, x \rangle + \delta_K(x)) + \delta_{\Omega}(z)\\
{\rm s.t.} & x - z = 0,
\end{array}
\end{equation}
where $\delta_{\Omega}$ is the indicator function of the set $\Omega := \{z \in \mathbb{R}^N ~|~ Az - b = 0\}$. Since the optimization problem \eqref{eq: WBP_LP_primal_equiv} is in the form of \eqref{primal}, the HPR algorithm can be applied to solve it. A key step in the HPR algorithm for solving \eqref{eq: WBP_LP_primal_equiv} is to compute the projection of a given point $\widehat{z} \in \mathbb{R}^N$ onto the set $\Omega$, which is given by
\[
\Pi_{\Omega}(\widehat{z}) = \widehat{z} - A^*(AA^*)^{-1}(A\widehat{z} - b).
\]
The projection $\Pi_{\Omega}(\widehat{z})$ can be efficiently computed by applying Algorithm \ref{alg:normal}. As a consequence, we can also obtain the $O((m \sum_{t = 1}^T m_t)/\epsilon)$ computational complexity of the HPR algorithm for solving the WBP in terms of the primal objective function value gap. In this paper, we prefer to apply the HPR algorithm to solve the dual problem \eqref{model:dualLP} partially because it can be more memory efficient. 
\end{remark}

\section{Numerical Experiments}
In this section, we present the numerical performance of the HPR algorithm for solving the WBP with fixed supports on both synthetic and real data sets.  We use the 2-Wasserstein distance in all our experiments. We will compare the performance of the HPR algorithm with the fast-ADMM, IBP \citep{benamou2015iterative,schmitzer2019stabilized}, and the commercial software Gurobi. All the numerical experiments in this paper are obtained by running Matlab R2022a on a desktop with Intel(R) Core i7-10700HQ CPU @2.90GHz and 32GB of RAM.
\subsection{Implementation Details}
We apply Algorithm \ref{alg:normal} to solve the linear system involved in ADMM and call the algorithm fast-ADMM. We adopt the following stopping criterion based on the relative KKT residual for the HPR algorithm and the fast-ADMM:
\begin{equation}\label{relative-KKT-res}
  {\rm KKT_{res}}=\max \left\{\frac{\|b-A x\|}{1+\|b\|},\frac{\|\min(x,0)\|}{1+\|x\|}, \frac{\left\|A^{T} y+s-c\right\|}{1+\|c\|+\|s\|}, \frac{\left\|s-\Pi_{K}(s-x)\right\|}{1+\|x\|+\|s\|}\right\} \leq  {10^{-5}}. 
\end{equation}
In the implementation, we check the relative KKT residual every 50 steps. For the ADMM algorithm, we will follow the algorithm proposed in \citep{li2020asymptotically}, where the dual step size $\gamma$ can be chosen in $(0,2)$. Based on our numerical testing, we observe that the fast-ADMM with $\gamma=1.9$ has better performance. Hence, we set $\gamma = 1.9$ as the default dual step size in the fast-ADMM. Also, we consider a hybrid of HPR and fast-ADMM called HPR-hybrid. The specific hybrid strategy is as follows:
\[
\left\{\begin{array}{lc}
     \text{Run fast-ADMM},      & \mbox{if $k\leq 800$ and  KKT\textsubscript{res} $\geq 2*10^{-4}$},\\
     \text{Run HPR },  & \text{ otherwise,} \\
\end{array}   \right.\]
where $k$ is the iteration number. In addition, we find that restarting is very useful for improving the performance of HPR. In this paper, we adopt the following restart strategy for HPR:
\[
\left\{\begin{array}{lc}
     \mbox{every 50 iterations}, & \mbox{if $k \leq 500$}, \\
     {\rm KKT_{res}^{old}}>{\rm KKT_{res}^{now}}\quad  \mbox{or} \quad  \operatorname{mod}(k,500)==0,        &        \mbox{if $k>500$},
\end{array}   \right.
\]
here ${\rm KKT_{res}^{old}}$ is the relative KKT residual of the last checking and ${\rm KKT_{res}^{now}}$ is the relative KKT residual of the current checking. Note that we will check the relative KKT residual every $50$ iteration. 

For IBP, the regularization parameters are chosen from $\{0.01, 0.001\}$ for the synthetic data sets and $\{0.01,0.001,0.0005\}$ for real data sets, respectively. If sample distributions share the same ground cost $\{\mathcal{D}(\mathcal{P}^{c},\mathcal{P}^{t})\}_{t=1}^{T}$, then we run the IBP implemented in the POT toolbox\footnote{https://pythonot.github.io}, which is a standard and highly efficient numerical solver for the WBP. Otherwise, we 
    use the Matlab code of the IBP implemented by \cite{yang2021fast}. For convenience, we call them POT and IBP, respectively. When the regularization parameter is less than 0.01, we will run both the IBP and its stabilized version simultaneously, and will only report the faster results in terms of running time. We adopt the default stopping criteria in the POT and IBP and set $\operatorname{Tol}_{\mathrm{IBP}}= 10^{-6}$. 
\par 
In this paper, we also compare with the interior point method (IPM) implemented in Gurobi (v9.50 with an academic license). We adopt the default stopping criteria in the Gurobi and set the tolerance to be $10^{-6}$. We disable the pre-solving phase as well as the cross-over strategy. There are three reasons for choosing the aforementioned settings. First, as observed from our experiments, other methods (such as the primal/dual simplex method) implemented in Gurobi are in general not as efficient as the IPM. Second, enabling the pre-solving stage did not improve the performance of the IPM in our numerical tests. Third, the cross-over strategy is usually too costly for our experiments, and we do not require a basic solution.
\par 

We set the maximum number of iterations as 10000 for all algorithms.
The maximum running time (for each run) of all algorithms is set to be \textbf{one hour}. For the evaluation of the quality of the solution, we report the ${\rm KKT_{res}}$ for Gurobi, fast-ADMM, HPR, and HPR-hybrid. To compare with IBP, we will use '\textbf{relative obj gap}' and '\textbf{relative primal feasibility error}'. Specifically, '\textbf{relative obj gap}' stands for the relative objective value gap which is defined by
\begin{equation}\label{normalized-obj}
   \textbf {relative obj gap:}=\frac{\left|\sum_{t=1}^{T} \left\langle D^{t},{X}^{t}\right\rangle-\sum_{t=1}^{T}\left\langle D^{t}, X^{t}_{\mathrm{g}}\right\rangle\right|}{\left|\sum_{t=1}^{T} \left\langle D^{t}, X^{t}_{\mathrm{g}}\right\rangle\right|+1}, 
\end{equation}
where $\left({X}^{1}, \ldots, {X}^{T}\right)$ is the solution obtained by the algorithm, $\left(X^{1}_{\mathrm{g}}, \ldots, X^{T}_{\mathrm{g}}\right)$ denotes the solution obtained by Gurobi;  and $D^{t}=\omega_{t} \mathcal{D}\left(\mathcal{P}^{c} , \mathcal{P}^{t} \right),t=1,\dots,T$. “\textbf{relative primal feasibility error}” denotes the value of 
$$
 \max \left\{\frac{\|b-A x\|}{1+\|b\|},\frac{\|\min(x,0)\|}{1+\|x\|}\right\}. 
$$
In order to compare the efficiency of the algorithms, we will report the computational time (in seconds) and the number of iterations.
\subsection{Experiments on Synthetic Data}
In this subsection, we randomly generated a set of discrete probability distributions $\{\mathcal{P}^{t}\}_{t=1}^{T}$ with $\mathcal{P}^{t}=\{(a_{i}^{t}, \boldsymbol{q}_{i}^{t}) \in \mathbb{R}_{+} \times \mathbb{R}^{d}: i=1, \cdots, m_{t}\}$ and $\sum_{i=1}^{m_{t}} a_{i}^{t}=1$. Specifically, we generate the supports $\left\{\boldsymbol{q}_{i}^{t}: i=1, \cdots, m_{t}, t=1, \cdots, T\right\}$ whose entries are drawn from a Gaussian mixture distribution via the Matlab commands following \citep{yang2021fast}:
\begin{equation*}
	\begin{aligned}
			&\text { gm\_num }=5 ; \text { gm\_mean }=[-20 ;-10 ; 0 ; 10 ; 20] ; \\
		&\text { sigma = zeros }(1,1, \text { gm\_num }) ; \text { sigma }(1,1,:)=5 * \text { ones }\left(\mathrm{gm}_{-} \text {num, } 1\right) ; \\
		&\text { gm\_weights = rand(gm\_num, } 1) ; \\
		&\text { distrib = gmdistribution(gm\_mean, sigma, gm\_weights).}
	\end{aligned}
\end{equation*}
The associated weight vector $\left(a_{1}^{t}, \cdots, a_{m_{t}}^{t}\right)$ for $\mathcal{P}^{t},t=1,\ldots, T$ are drawn from the standard uniform distribution on the open interval $(0,1)$, and then normalize it such that $\sum_{i=1}^{m_{t}} a_{i}^{t}=1$. Similarly, we also randomly generate the weight vectors $(\omega_{1},\ldots,\omega_{T})$. After generating all $\{\mathcal{P}^{t}\}_{t=1}^{T}$,  to compute a Wasserstein barycenter $\mathcal{P}^{c}=\left\{\left(a^{c}_{i}, \boldsymbol{q}^{c}_{i}\right) \in \mathbb{R}_{+} \times\right.$ $\left.\mathbb{R}^{d}: i=1, \cdots, m\right\}$, we first use the k-means\footnote{In our experiments, we call the Matlab function “kmeans”, which is built-in statistics and machine learning toolbox.} method to select $m$ points from $\left\{\boldsymbol{q}_{i}^{t}: i=1, \cdots, m_{t}, t=1, \cdots, T\right\}$ to be the support points of $\mathcal{P}^{c}$. For each $t=1,\dots,T$, the distance matrix is obtained by 
\begin{equation}\label{cost:D}
    \mathcal{D}(\mathcal{P}^{c},\mathcal{P}^{t})_{ij} = \|{q}_{i}^{c}-{q}_{j}^{t}\|^{2}_2
\end{equation}
 for $i=1,\dots,m$ and $j=1,\dots,m_t$ and is normalized to have $\max_{t = 1}^T \|\mathcal{D}(\mathcal{P}^{c},\mathcal{P}^{t})\|_{\infty}=1$. Then we run the previously mentioned methods to solve WBP with fixed supports in \eqref{model:fixedX}. It should be mentioned that $\mathcal{D}\left(\mathcal{P}^{c}, \mathcal{P}^{1} \right)=\dots=\mathcal{D}\left(\mathcal{P}^{c},\mathcal{P}^{T} \right)$ does not hold for this random data. As a result, POT is not applicable and we omit the comparison with POT in this subsection. In addition, for convenience, we set $d=3$ and $m_{1}=\cdots=m_{T}=m_{t}$, and choose different $\left(T, m, m_{t}\right)$. Then, given each triple $\left(T, m, m_{t}\right)$, we randomly generate a trial, where each distribution has dense weights and different support points. All results presented are the average of 10 independent trials. 
\begin{table}[]
\centering
\caption{Numerical results on Gaussian mixture distributions}
\label{tab:experiment-1}
\setlength{\tabcolsep}{5mm}{
\resizebox{\textwidth}{!}{%
\begin{tabular}{|ccc|cccccc|}
\hline
\multicolumn{3}{|c|}{}                                     & \multicolumn{1}{c|}{Gurobi}   & \multicolumn{1}{c|}{fast-ADMM} & \multicolumn{1}{c|}{HPR}      & \multicolumn{1}{c|}{HPR-hybrid} & \multicolumn{1}{c|}{IBP(0.01)} & IBP(0.001) \\ \hline
\multicolumn{1}{|c|}{m}   & \multicolumn{1}{c|}{mt}  & T   & \multicolumn{6}{c|}{\textbf{${\rm KKT_{res}}$}}                                                                                                                                         \\ \hline
\multicolumn{1}{|c|}{100} & \multicolumn{1}{c|}{100} & 100 & \multicolumn{1}{c|}{1.73E-07} & \multicolumn{1}{c|}{9.85E-06}  & \multicolumn{1}{c|}{9.80E-06} & \multicolumn{1}{c|}{9.77E-06}   & \multicolumn{1}{c|}{-}         & -          \\ \hline
\multicolumn{1}{|c|}{100} & \multicolumn{1}{c|}{100} & 200 & \multicolumn{1}{c|}{1.40E-07} & \multicolumn{1}{c|}{9.92E-06}  & \multicolumn{1}{c|}{9.84E-06} & \multicolumn{1}{c|}{9.66E-06}   & \multicolumn{1}{c|}{-}         & -          \\ \hline
\multicolumn{1}{|c|}{100} & \multicolumn{1}{c|}{100} & 400 & \multicolumn{1}{c|}{1.50E-07} & \multicolumn{1}{c|}{9.94E-06}  & \multicolumn{1}{c|}{9.67E-06} & \multicolumn{1}{c|}{9.82E-06}   & \multicolumn{1}{c|}{-}         & -          \\ \hline
\multicolumn{1}{|c|}{100} & \multicolumn{1}{c|}{100} & 800 & \multicolumn{1}{c|}{1.72E-07} & \multicolumn{1}{c|}{9.95E-06}  & \multicolumn{1}{c|}{9.90E-06} & \multicolumn{1}{c|}{9.88E-06}   & \multicolumn{1}{c|}{-}         & -          \\ \hline
\multicolumn{1}{|c|}{200} & \multicolumn{1}{c|}{100} & 100 & \multicolumn{1}{c|}{2.22E-07} & \multicolumn{1}{c|}{9.89E-06}  & \multicolumn{1}{c|}{9.81E-06} & \multicolumn{1}{c|}{9.66E-06}   & \multicolumn{1}{c|}{-}         & -          \\ \hline
\multicolumn{1}{|c|}{400} & \multicolumn{1}{c|}{100} & 100 & \multicolumn{1}{c|}{1.43E-07} & \multicolumn{1}{c|}{9.92E-06}  & \multicolumn{1}{c|}{9.78E-06} & \multicolumn{1}{c|}{9.67E-06}   & \multicolumn{1}{c|}{-}         & -          \\ \hline
\multicolumn{1}{|c|}{800} & \multicolumn{1}{c|}{100} & 100 & \multicolumn{1}{c|}{1.23E-07} & \multicolumn{1}{c|}{9.93E-06}  & \multicolumn{1}{c|}{9.92E-06} & \multicolumn{1}{c|}{9.71E-06}   & \multicolumn{1}{c|}{-}         & -          \\ \hline
\multicolumn{1}{|c|}{100} & \multicolumn{1}{c|}{200} & 100 & \multicolumn{1}{c|}{2.12E-07} & \multicolumn{1}{c|}{9.92E-06}  & \multicolumn{1}{c|}{9.77E-06} & \multicolumn{1}{c|}{9.78E-06}   & \multicolumn{1}{c|}{-}         & -          \\ \hline
\multicolumn{1}{|c|}{100} & \multicolumn{1}{c|}{400} & 100 & \multicolumn{1}{c|}{6.00E-08} & \multicolumn{1}{c|}{9.90E-06}  & \multicolumn{1}{c|}{9.57E-06} & \multicolumn{1}{c|}{9.68E-06}   & \multicolumn{1}{c|}{-}         & -          \\ \hline
\multicolumn{1}{|c|}{100} & \multicolumn{1}{c|}{800} & 100 & \multicolumn{1}{c|}{3.16E-08} & \multicolumn{1}{c|}{9.86E-06}  & \multicolumn{1}{c|}{9.73E-06} & \multicolumn{1}{c|}{9.55E-06}   & \multicolumn{1}{c|}{-}         & -          \\ \hline
\multicolumn{1}{|c|}{m}   & \multicolumn{1}{c|}{mt}  & T   & \multicolumn{6}{c|}{\textbf{relative obj gap}}                                                                                                                                 \\ \hline
\multicolumn{1}{|c|}{100} & \multicolumn{1}{c|}{100} & 100 & \multicolumn{1}{c|}{0}        & \multicolumn{1}{c|}{4.39E-05}  & \multicolumn{1}{c|}{9.31E-05} & \multicolumn{1}{c|}{6.74E-05}   & \multicolumn{1}{c|}{3.29E-01}  & 1.60E-02   \\ \hline
\multicolumn{1}{|c|}{100} & \multicolumn{1}{c|}{100} & 200 & \multicolumn{1}{c|}{0}        & \multicolumn{1}{c|}{5.79E-05}  & \multicolumn{1}{c|}{1.11E-04} & \multicolumn{1}{c|}{8.44E-05}   & \multicolumn{1}{c|}{4.11E-01}  & 1.92E-02   \\ \hline
\multicolumn{1}{|c|}{100} & \multicolumn{1}{c|}{100} & 400 & \multicolumn{1}{c|}{0}        & \multicolumn{1}{c|}{6.82E-05}  & \multicolumn{1}{c|}{1.26E-04} & \multicolumn{1}{c|}{9.64E-05}   & \multicolumn{1}{c|}{4.81E-01}  & 2.23E-02   \\ \hline
\multicolumn{1}{|c|}{100} & \multicolumn{1}{c|}{100} & 800 & \multicolumn{1}{c|}{0}        & \multicolumn{1}{c|}{8.01E-05}  & \multicolumn{1}{c|}{1.30E-04} & \multicolumn{1}{c|}{1.06E-04}   & \multicolumn{1}{c|}{5.34E-01}  & 2.40E-02   \\ \hline
\multicolumn{1}{|c|}{200} & \multicolumn{1}{c|}{100} & 100 & \multicolumn{1}{c|}{0}        & \multicolumn{1}{c|}{6.07E-05}  & \multicolumn{1}{c|}{1.10E-04} & \multicolumn{1}{c|}{9.84E-05}   & \multicolumn{1}{c|}{3.61E-01}  & 2.32E-02   \\ \hline
\multicolumn{1}{|c|}{400} & \multicolumn{1}{c|}{100} & 100 & \multicolumn{1}{c|}{0}        & \multicolumn{1}{c|}{8.31E-05}  & \multicolumn{1}{c|}{1.39E-04} & \multicolumn{1}{c|}{1.29E-04}   & \multicolumn{1}{c|}{3.80E-01}  & 3.10E-02   \\ \hline
\multicolumn{1}{|c|}{800} & \multicolumn{1}{c|}{100} & 100 & \multicolumn{1}{c|}{0}        & \multicolumn{1}{c|}{1.63E-04}  & \multicolumn{1}{c|}{1.94E-04} & \multicolumn{1}{c|}{2.10E-04}   & \multicolumn{1}{c|}{4.10E-01}  & 4.06E-02   \\ \hline
\multicolumn{1}{|c|}{100} & \multicolumn{1}{c|}{200} & 100 & \multicolumn{1}{c|}{0}        & \multicolumn{1}{c|}{5.81E-05}  & \multicolumn{1}{c|}{1.35E-04} & \multicolumn{1}{c|}{1.16E-04}   & \multicolumn{1}{c|}{3.77E-01}  & 1.56E-02   \\ \hline
\multicolumn{1}{|c|}{100} & \multicolumn{1}{c|}{400} & 100 & \multicolumn{1}{c|}{0}        & \multicolumn{1}{c|}{5.99E-05}  & \multicolumn{1}{c|}{1.72E-04} & \multicolumn{1}{c|}{1.65E-04}   & \multicolumn{1}{c|}{4.20E-01}  & 1.77E-02   \\ \hline
\multicolumn{1}{|c|}{100} & \multicolumn{1}{c|}{800} & 100 & \multicolumn{1}{c|}{0}        & \multicolumn{1}{c|}{7.55E-05}  & \multicolumn{1}{c|}{1.66E-04} & \multicolumn{1}{c|}{1.90E-04}   & \multicolumn{1}{c|}{4.37E-01}  & 1.75E-02   \\ \hline
\multicolumn{1}{|c|}{m}   & \multicolumn{1}{c|}{mt}  & T   & \multicolumn{6}{c|}{\textbf{iter}}                                                                                                                                             \\ \hline
\multicolumn{1}{|c|}{100} & \multicolumn{1}{c|}{100} & 100 & \multicolumn{1}{c|}{39}       & \multicolumn{1}{c|}{3558}      & \multicolumn{1}{c|}{1515}     & \multicolumn{1}{c|}{1320}       & \multicolumn{1}{c|}{190}       & 3060       \\ \hline
\multicolumn{1}{|c|}{100} & \multicolumn{1}{c|}{100} & 200 & \multicolumn{1}{c|}{54}       & \multicolumn{1}{c|}{3978}      & \multicolumn{1}{c|}{1615}     & \multicolumn{1}{c|}{1340}       & \multicolumn{1}{c|}{200}       & 4220       \\ \hline
\multicolumn{1}{|c|}{100} & \multicolumn{1}{c|}{100} & 400 & \multicolumn{1}{c|}{48}       & \multicolumn{1}{c|}{4368}      & \multicolumn{1}{c|}{1748}     & \multicolumn{1}{c|}{1395}       & \multicolumn{1}{c|}{210}       & 6120       \\ \hline
\multicolumn{1}{|c|}{100} & \multicolumn{1}{c|}{100} & 800 & \multicolumn{1}{c|}{45}       & \multicolumn{1}{c|}{4743}      & \multicolumn{1}{c|}{1988}     & \multicolumn{1}{c|}{1423}       & \multicolumn{1}{c|}{210}       & 7530       \\ \hline
\multicolumn{1}{|c|}{200} & \multicolumn{1}{c|}{100} & 100 & \multicolumn{1}{c|}{49}       & \multicolumn{1}{c|}{3270}      & \multicolumn{1}{c|}{1713}     & \multicolumn{1}{c|}{1363}       & \multicolumn{1}{c|}{200}       & 2110       \\ \hline
\multicolumn{1}{|c|}{400} & \multicolumn{1}{c|}{100} & 100 & \multicolumn{1}{c|}{51}       & \multicolumn{1}{c|}{3253}      & \multicolumn{1}{c|}{1860}     & \multicolumn{1}{c|}{1403}       & \multicolumn{1}{c|}{200}       & 1490       \\ \hline
\multicolumn{1}{|c|}{800} & \multicolumn{1}{c|}{100} & 100 & \multicolumn{1}{c|}{55}       & \multicolumn{1}{c|}{3070}      & \multicolumn{1}{c|}{2455}     & \multicolumn{1}{c|}{1560}       & \multicolumn{1}{c|}{170}       & 1270       \\ \hline
\multicolumn{1}{|c|}{100} & \multicolumn{1}{c|}{200} & 100 & \multicolumn{1}{c|}{36}       & \multicolumn{1}{c|}{3473}      & \multicolumn{1}{c|}{1565}     & \multicolumn{1}{c|}{1303}       & \multicolumn{1}{c|}{200}       & 2440       \\ \hline
\multicolumn{1}{|c|}{100} & \multicolumn{1}{c|}{400} & 100 & \multicolumn{1}{c|}{47}       & \multicolumn{1}{c|}{3030}      & \multicolumn{1}{c|}{1465}     & \multicolumn{1}{c|}{1255}       & \multicolumn{1}{c|}{200}       & 2400       \\ \hline
\multicolumn{1}{|c|}{100} & \multicolumn{1}{c|}{800} & 100 & \multicolumn{1}{c|}{49}       & \multicolumn{1}{c|}{2485}      & \multicolumn{1}{c|}{1605}     & \multicolumn{1}{c|}{1250}       & \multicolumn{1}{c|}{210}       & 2070       \\ \hline
\multicolumn{1}{|c|}{m}   & \multicolumn{1}{c|}{mt}  & T   & \multicolumn{6}{c|}{\textbf{relative primal feasibility   error}}                                                                                                              \\ \hline
\multicolumn{1}{|c|}{100} & \multicolumn{1}{c|}{100} & 100 & \multicolumn{1}{c|}{1.76E-09} & \multicolumn{1}{c|}{9.70E-06}  & \multicolumn{1}{c|}{9.68E-06} & \multicolumn{1}{c|}{8.60E-06}   & \multicolumn{1}{c|}{5.89E-08}  & 5.71E-07   \\ \hline
\multicolumn{1}{|c|}{100} & \multicolumn{1}{c|}{100} & 200 & \multicolumn{1}{c|}{1.70E-09} & \multicolumn{1}{c|}{9.66E-06}  & \multicolumn{1}{c|}{9.61E-06} & \multicolumn{1}{c|}{7.98E-06}   & \multicolumn{1}{c|}{3.57E-08}  & 6.65E-07   \\ \hline
\multicolumn{1}{|c|}{100} & \multicolumn{1}{c|}{100} & 400 & \multicolumn{1}{c|}{1.07E-09} & \multicolumn{1}{c|}{9.61E-06}  & \multicolumn{1}{c|}{9.44E-06} & \multicolumn{1}{c|}{8.43E-06}   & \multicolumn{1}{c|}{3.50E-08}  & 6.84E-07   \\ \hline
\multicolumn{1}{|c|}{100} & \multicolumn{1}{c|}{100} & 800 & \multicolumn{1}{c|}{3.88E-09} & \multicolumn{1}{c|}{9.40E-06}  & \multicolumn{1}{c|}{8.71E-06} & \multicolumn{1}{c|}{8.54E-06}   & \multicolumn{1}{c|}{2.09E-08}  & 1.16E-06   \\ \hline
\multicolumn{1}{|c|}{200} & \multicolumn{1}{c|}{100} & 100 & \multicolumn{1}{c|}{1.54E-09} & \multicolumn{1}{c|}{9.35E-06}  & \multicolumn{1}{c|}{9.74E-06} & \multicolumn{1}{c|}{8.83E-06}   & \multicolumn{1}{c|}{1.41E-09}  & 4.88E-07   \\ \hline
\multicolumn{1}{|c|}{400} & \multicolumn{1}{c|}{100} & 100 & \multicolumn{1}{c|}{1.57E-09} & \multicolumn{1}{c|}{9.32E-06}  & \multicolumn{1}{c|}{9.78E-06} & \multicolumn{1}{c|}{8.38E-06}   & \multicolumn{1}{c|}{1.78E-09}  & 3.27E-07   \\ \hline
\multicolumn{1}{|c|}{800} & \multicolumn{1}{c|}{100} & 100 & \multicolumn{1}{c|}{1.57E-09} & \multicolumn{1}{c|}{9.23E-06}  & \multicolumn{1}{c|}{9.39E-06} & \multicolumn{1}{c|}{8.81E-06}   & \multicolumn{1}{c|}{6.39E-08}  & 2.70E-07   \\ \hline
\multicolumn{1}{|c|}{100} & \multicolumn{1}{c|}{200} & 100 & \multicolumn{1}{c|}{1.52E-09} & \multicolumn{1}{c|}{9.85E-06}  & \multicolumn{1}{c|}{8.94E-06} & \multicolumn{1}{c|}{8.27E-06}   & \multicolumn{1}{c|}{2.69E-09}  & 5.67E-07   \\ \hline
\multicolumn{1}{|c|}{100} & \multicolumn{1}{c|}{400} & 100 & \multicolumn{1}{c|}{1.37E-09} & \multicolumn{1}{c|}{9.88E-06}  & \multicolumn{1}{c|}{9.08E-06} & \multicolumn{1}{c|}{8.03E-06}   & \multicolumn{1}{c|}{4.19E-09}  & 5.22E-07   \\ \hline
\multicolumn{1}{|c|}{100} & \multicolumn{1}{c|}{800} & 100 & \multicolumn{1}{c|}{2.62E-12} & \multicolumn{1}{c|}{9.72E-06}  & \multicolumn{1}{c|}{9.17E-06} & \multicolumn{1}{c|}{7.86E-06}   & \multicolumn{1}{c|}{7.86E-10}  & 4.20E-07   \\ \hline
\multicolumn{1}{|c|}{m}   & \multicolumn{1}{c|}{mt}  & T   & \multicolumn{6}{c|}{\textbf{time(s)}}                                                                                                                                          \\ \hline
\multicolumn{1}{|c|}{100} & \multicolumn{1}{c|}{100} & 100 & \multicolumn{1}{c|}{7.33}     & \multicolumn{1}{c|}{14.66}     & \multicolumn{1}{c|}{5.46}     & \multicolumn{1}{c|}{5.25}       & \multicolumn{1}{c|}{0.99}      & 16.28      \\ \hline
\multicolumn{1}{|c|}{100} & \multicolumn{1}{c|}{100} & 200 & \multicolumn{1}{c|}{23.23}    & \multicolumn{1}{c|}{36.70}     & \multicolumn{1}{c|}{13.67}    & \multicolumn{1}{c|}{12.00}      & \multicolumn{1}{c|}{2.20}      & 46.62      \\ \hline
\multicolumn{1}{|c|}{100} & \multicolumn{1}{c|}{100} & 400 & \multicolumn{1}{c|}{40.71}    & \multicolumn{1}{c|}{81.53}     & \multicolumn{1}{c|}{30.66}    & \multicolumn{1}{c|}{25.44}      & \multicolumn{1}{c|}{4.59}      & 130.67     \\ \hline
\multicolumn{1}{|c|}{100} & \multicolumn{1}{c|}{100} & 800 & \multicolumn{1}{c|}{70.43}    & \multicolumn{1}{c|}{176.82}    & \multicolumn{1}{c|}{70.14}    & \multicolumn{1}{c|}{52.50}      & \multicolumn{1}{c|}{8.71}      & 310.59     \\ \hline
\multicolumn{1}{|c|}{200} & \multicolumn{1}{c|}{100} & 100 & \multicolumn{1}{c|}{20.20}    & \multicolumn{1}{c|}{29.62}     & \multicolumn{1}{c|}{14.20}    & \multicolumn{1}{c|}{11.93}      & \multicolumn{1}{c|}{2.17}      & 22.91      \\ \hline
\multicolumn{1}{|c|}{400} & \multicolumn{1}{c|}{100} & 100 & \multicolumn{1}{c|}{44.74}    & \multicolumn{1}{c|}{59.64}     & \multicolumn{1}{c|}{32.23}    & \multicolumn{1}{c|}{24.83}      & \multicolumn{1}{c|}{4.20}      & 31.32      \\ \hline
\multicolumn{1}{|c|}{800} & \multicolumn{1}{c|}{100} & 100 & \multicolumn{1}{c|}{202.25}   & \multicolumn{1}{c|}{111.45}    & \multicolumn{1}{c|}{83.87}    & \multicolumn{1}{c|}{54.97}      & \multicolumn{1}{c|}{6.97}      & 51.68      \\ \hline
\multicolumn{1}{|c|}{100} & \multicolumn{1}{c|}{200} & 100 & \multicolumn{1}{c|}{12.77}    & \multicolumn{1}{c|}{31.11}     & \multicolumn{1}{c|}{12.72}    & \multicolumn{1}{c|}{11.15}      & \multicolumn{1}{c|}{2.05}      & 24.66      \\ \hline
\multicolumn{1}{|c|}{100} & \multicolumn{1}{c|}{400} & 100 & \multicolumn{1}{c|}{34.27}    & \multicolumn{1}{c|}{54.68}     & \multicolumn{1}{c|}{24.83}    & \multicolumn{1}{c|}{22.04}      & \multicolumn{1}{c|}{4.16}      & 49.90      \\ \hline
\multicolumn{1}{|c|}{100} & \multicolumn{1}{c|}{800} & 100 & \multicolumn{1}{c|}{224.91}   & \multicolumn{1}{c|}{90.35}     & \multicolumn{1}{c|}{56.09}    & \multicolumn{1}{c|}{44.46}      & \multicolumn{1}{c|}{8.68}      & 85.93      \\ \hline
\end{tabular}%
}}
\end{table}

\par A key purpose of the numerical experiments is to demonstrate that the proposed algorithms in this paper can obtain better solutions in terms of solution quality than the entropy regularization based algorithms, such as the IBP and its stabilized version, and in comparable computational time. Thus, we will focus on the comparison with the IBP. But we want to briefly mention the comparison between the fast-ADMM and the sGS-ADMM \cite{yang2021fast}. Based on our testing, in order to obtain solutions with comparable quality, in general, the sGS-ADMM needs about 20\% to 30\% more iterations than the fast-ADMM. In terms of the per-iteration computational time, the sGS-ADMM needs about 70\% more computational time than the fast-ADMM.  

The numerical results on the synthetic data are summarized in Table \ref{tab:experiment-1}. Next, we briefly discuss the numerical results. First, the numerical results show that, under reasonable feasibility error, the quality of the solutions obtained by the fast-ADMM, HPR, and HPR-hybrid is better, implied by the relative objective function value gap. We will further justify this point in the numerical experiments on the real datasets. Second, the IBP algorithm in terms of the per-iteration computational time is extremely fast, but the performance of the IBP is sensitive to the regularization parameter.
Third, the per-iteration computational cost of the fast-ADMM is very economical, which partially demonstrates the importance of a linear time complexity procedure in Algorithm \ref{alg:normal}. Forth, the HPR algorithm and the HPR-hybrid algorithm are faster than the fast-ADMM. We want to highlight that, due to our extensive numerical testing, we observe the HPR-hybrid usually outperforms the fast-ADMM and the HPR algorithm. One reason is the fast-ADMM usually performs quite well, and the HPR can accelerate the performance of the fast-ADMM when the fast-ADMM reaches the bottleneck since the HPR has a much better worst-case complexity guarantee than the ADMM.

Next, we compare the numerical performance of the fast-ADMM and the HPR algorithm with Gurobi. When $m, m_{t},$ and $T$ are relatively small, Gurobi can solve the problem efficiently. It is a little bit surprising that HPR and HPR-hybrid also have comparable performance in terms of time on small problems. For some large examples like $(m, m_{t}, T) = (100, 800, 100)$, HPR-hybrid can be $4$ times faster than Gurobi. To further compare the performances of Gurobi with fast-ADMM, HPR, and HPR-hybrid, we conduct more experiments on synthetic data. For triple $(m, m_t, T)$, we fixed two of them, and vary the other one. Figure \ref{Fig:experiment-2} shows that fast-ADMM, HPR, and HPR-hybrid always return a similar objective value as Gurobi and have a good feasibility accuracy. For the computational time, fast-ADMM and HPR, and HPR-hybrid increase almost linearly with respect to $m,m_t$ and $T$, which verifies the complexity result in Lemma \ref{compleixty:per}. But the computational time of Gurobi increases much more rapidly because the complexity of sparse Cholesky decomposition is not linearly related to the dimensionality of variables and it consumes too much memory for large-scale problems.
\begin{figure}[htbp]
\centering
\subfigure[$T$ varies with $m=50$, $m_t=10$]{
\begin{minipage}[p]{0.40\textwidth} 
    \centering 
    \resizebox{\linewidth}{!}{%
    \includegraphics{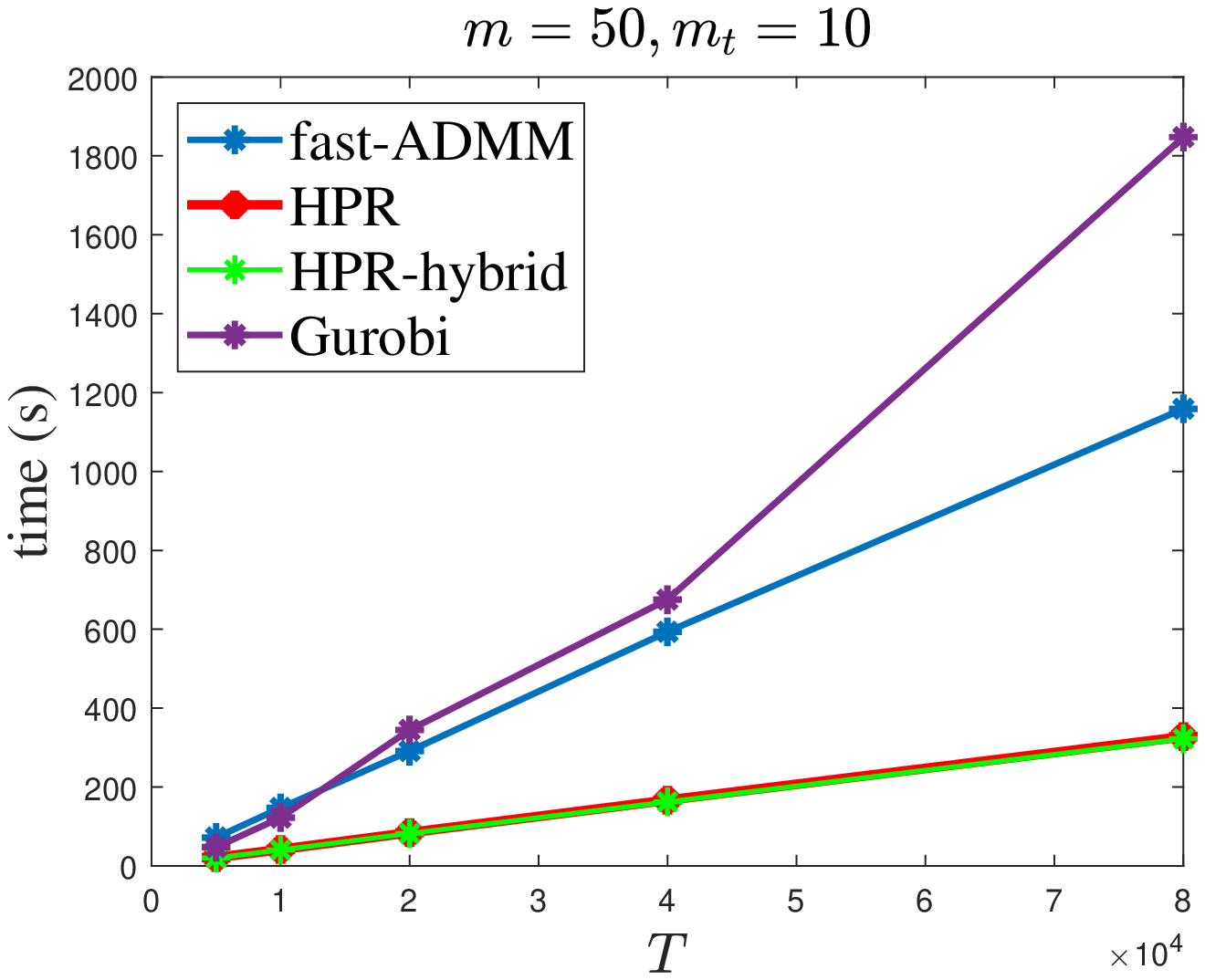} 
    }
 \end{minipage}
 
 \begin{minipage}[p]{0.6\textwidth} 
 {\renewcommand\arraystretch{1.3}
\resizebox{\textwidth}{!}{%
{\Huge
\begin{tabular}{|c|cccc|cccc|}
\hline
      & \multicolumn{4}{c|}{\textbf{relative   obj gap}}                                                                                                                                             & \multicolumn{4}{c|}{\textbf{relative   primal feasibility error}}                                                                                                                             \\ \hline
T     & \multicolumn{1}{c|}{Gurobi} & \multicolumn{1}{c|}{\begin{tabular}[c]{@{}c@{}}fast\\ ADMM\end{tabular}} & \multicolumn{1}{c|}{HPR}     & \begin{tabular}[c]{@{}c@{}}HPR\\ hybrid\end{tabular} & \multicolumn{1}{c|}{Gurobi}  & \multicolumn{1}{c|}{\begin{tabular}[c]{@{}c@{}}fast\\ ADMM\end{tabular}} & \multicolumn{1}{c|}{HPR}     & \begin{tabular}[c]{@{}c@{}}HPR\\ hybrid\end{tabular} \\ \hline
5000  & \multicolumn{1}{c|}{0}      & \multicolumn{1}{c|}{4.9E-06}                                             & \multicolumn{1}{c|}{1.2E-05} & 1.2E-05                                              & \multicolumn{1}{c|}{9.7E-10} & \multicolumn{1}{c|}{9.6E-06}                                             & \multicolumn{1}{c|}{8.9E-06} & 8.9E-06                                              \\ \hline
10000 & \multicolumn{1}{c|}{0}      & \multicolumn{1}{c|}{4.3E-06}                                             & \multicolumn{1}{c|}{1.1E-05} & 1.1E-05                                              & \multicolumn{1}{c|}{3.6E-09} & \multicolumn{1}{c|}{9.8E-06}                                             & \multicolumn{1}{c|}{8.8E-06} & 8.8E-06                                              \\ \hline
20000 & \multicolumn{1}{c|}{0}      & \multicolumn{1}{c|}{3.8E-06}                                             & \multicolumn{1}{c|}{1.1E-05} & 1.1E-05                                              & \multicolumn{1}{c|}{7.2E-09} & \multicolumn{1}{c|}{9.8E-06}                                             & \multicolumn{1}{c|}{9.4E-06} & 9.4E-06                                              \\ \hline
40000 & \multicolumn{1}{c|}{0}      & \multicolumn{1}{c|}{4.8E-06}                                             & \multicolumn{1}{c|}{1.2E-05} & 1.2E-05                                              & \multicolumn{1}{c|}{2.4E-09} & \multicolumn{1}{c|}{9.8E-06}                                             & \multicolumn{1}{c|}{9.3E-06} & 9.3E-06                                              \\ \hline
80000 & \multicolumn{1}{c|}{0}      & \multicolumn{1}{c|}{3.6E-06}                                             & \multicolumn{1}{c|}{9.2E-06} & 9.2E-06                                              & \multicolumn{1}{c|}{2.1E-08} & \multicolumn{1}{c|}{9.9E-06}                                             & \multicolumn{1}{c|}{9.4E-06} & 9.4E-06                                              \\ \hline
\end{tabular}%
}}
}
\end{minipage}
}
\subfigure[$m$ varies with $T=100$, $m_t=200$]{
\begin{minipage}[p]{0.4\textwidth} 
    \centering 
    \resizebox{\linewidth}{!}{%
    \includegraphics[]{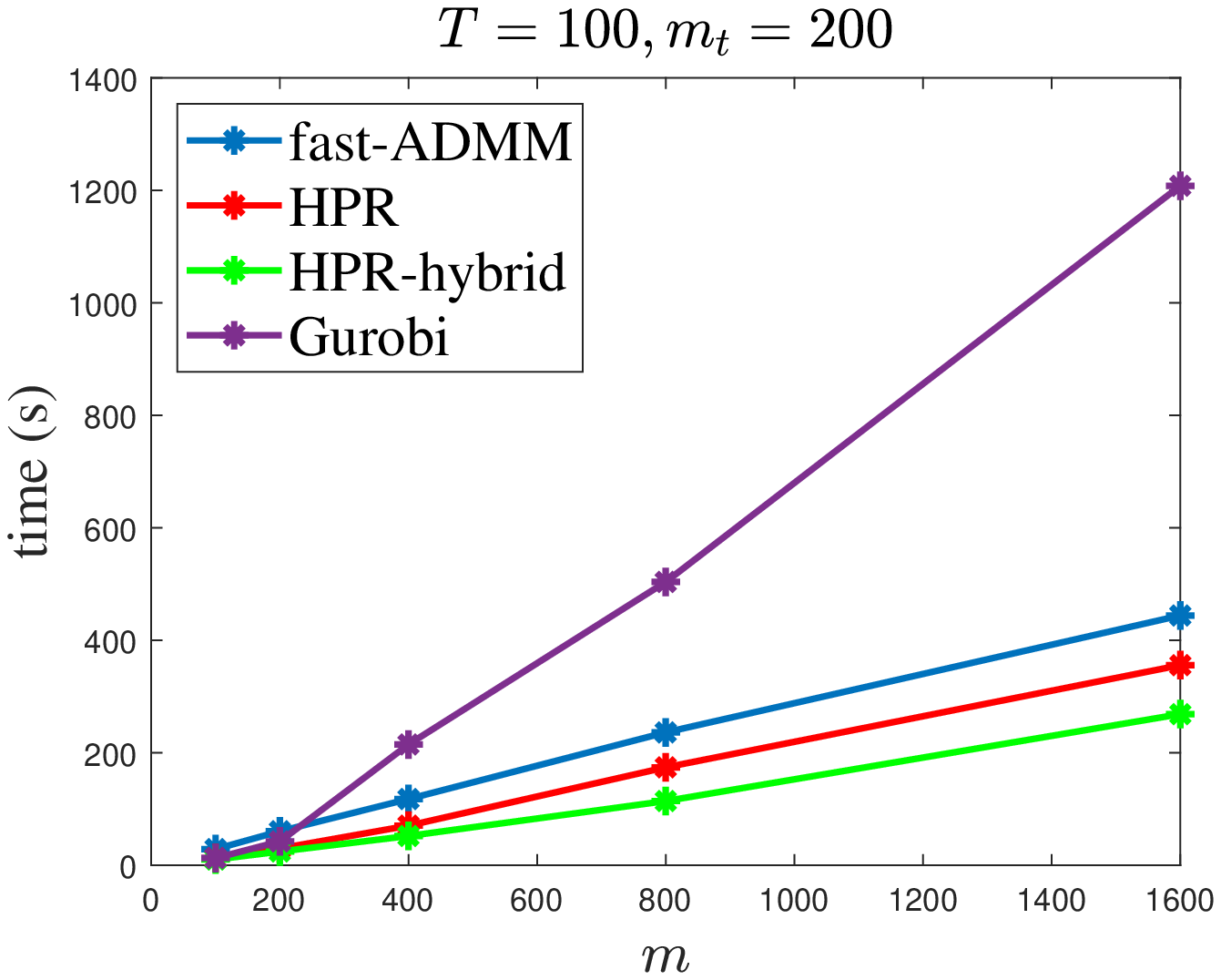} 
    }
 \end{minipage}
 \begin{minipage}[p]{0.6\textwidth} 
 {\renewcommand\arraystretch{1.3}
\resizebox{\textwidth}{!}{%
\Huge
\begin{tabular}{|c|cccc|cccc|}
\hline
     & \multicolumn{4}{c|}{\textbf{relative   obj gap}}                                                                                                                                             & \multicolumn{4}{c|}{\textbf{relative   primal feasibility error}}                                                                                                                             \\ \hline
m    & \multicolumn{1}{c|}{Gurobi} & \multicolumn{1}{c|}{\begin{tabular}[c]{@{}c@{}}fast\\ ADMM\end{tabular}} & \multicolumn{1}{c|}{HPR}     & \begin{tabular}[c]{@{}c@{}}HPR\\ hybrid\end{tabular} & \multicolumn{1}{c|}{Gurobi}  & \multicolumn{1}{c|}{\begin{tabular}[c]{@{}c@{}}fast\\ ADMM\end{tabular}} & \multicolumn{1}{c|}{HPR}     & \begin{tabular}[c]{@{}c@{}}HPR\\ hybrid\end{tabular} \\ \hline
100  & \multicolumn{1}{c|}{0}      & \multicolumn{1}{c|}{5.6E-05}                                             & \multicolumn{1}{c|}{1.4E-04} & 1.2E-04                                              & \multicolumn{1}{c|}{2.2E-09} & \multicolumn{1}{c|}{9.6E-06}                                             & \multicolumn{1}{c|}{9.1E-06} & 8.3E-06                                              \\ \hline
200  & \multicolumn{1}{c|}{0}      & \multicolumn{1}{c|}{1.0E-04}                                             & \multicolumn{1}{c|}{1.8E-04} & 1.6E-04                                              & \multicolumn{1}{c|}{1.7E-09} & \multicolumn{1}{c|}{9.5E-06}                                             & \multicolumn{1}{c|}{9.6E-06} & 8.6E-06                                              \\ \hline
400  & \multicolumn{1}{c|}{0}      & \multicolumn{1}{c|}{1.4E-04}                                             & \multicolumn{1}{c|}{2.0E-04} & 1.9E-04                                              & \multicolumn{1}{c|}{1.8E-09} & \multicolumn{1}{c|}{9.3E-06}                                             & \multicolumn{1}{c|}{9.6E-06} & 8.8E-06                                              \\ \hline
800  & \multicolumn{1}{c|}{0}      & \multicolumn{1}{c|}{2.2E-04}                                             & \multicolumn{1}{c|}{2.6E-04} & 3.1E-04                                              & \multicolumn{1}{c|}{2.2E-09} & \multicolumn{1}{c|}{9.0E-06}                                             & \multicolumn{1}{c|}{9.8E-06} & 9.0E-06                                              \\ \hline
1600 & \multicolumn{1}{c|}{0}      & \multicolumn{1}{c|}{3.5E-04}                                             & \multicolumn{1}{c|}{3.9E-04} & 5.3E-04                                              & \multicolumn{1}{c|}{2.3E-09} & \multicolumn{1}{c|}{8.7E-06}                                             & \multicolumn{1}{c|}{9.6E-06} & 8.9E-06                                              \\ \hline
\end{tabular}%
}
}
\end{minipage}
}
\subfigure[$m_t$ varies with $m=200$, $T=100$]{
\begin{minipage}[p]{0.4\textwidth} 
    \centering 
    \resizebox{\linewidth}{!}{%
    \includegraphics[width=50mm]{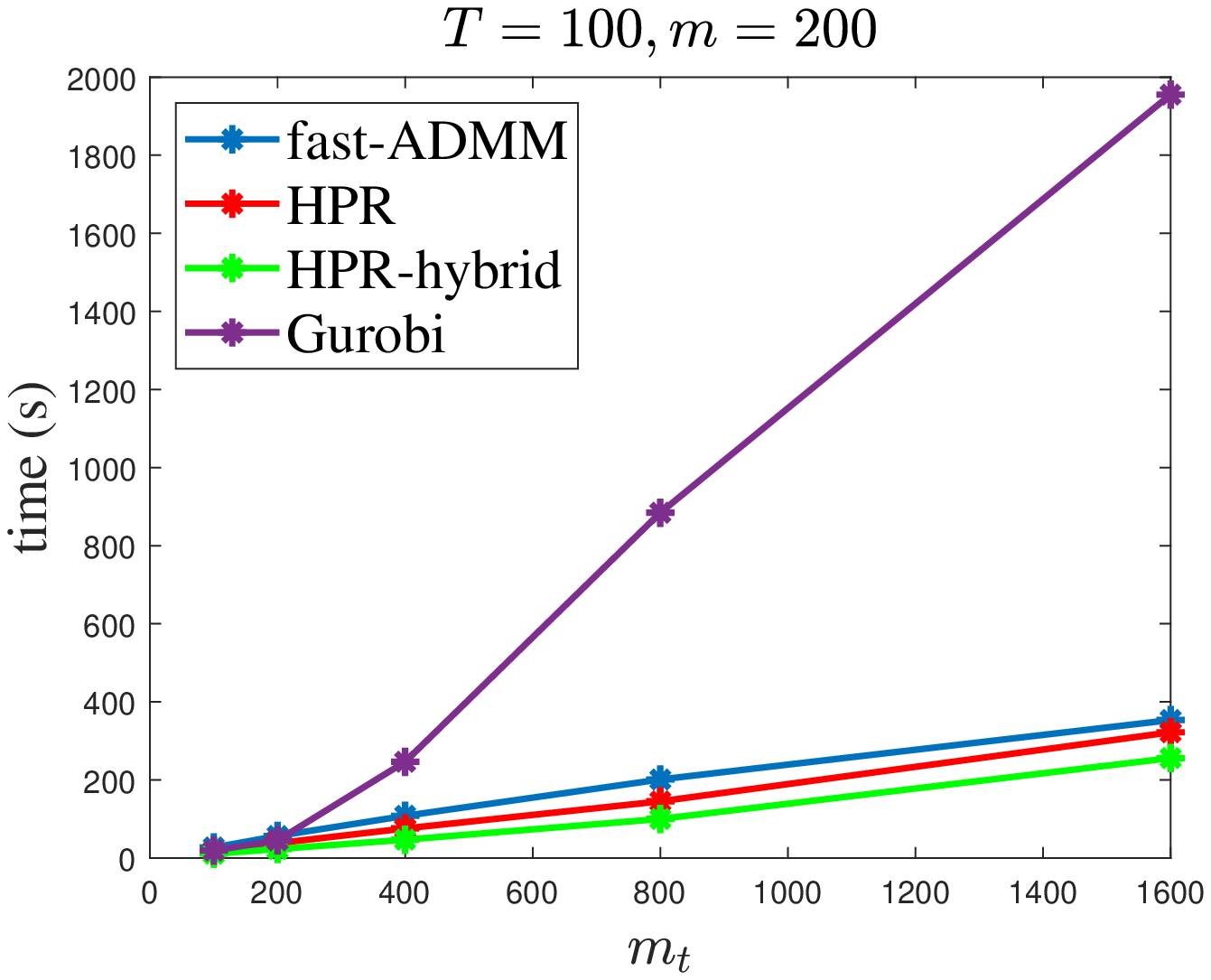} 
    }
 \end{minipage}
 \begin{minipage}[p]{0.6\textwidth} 
 {\renewcommand\arraystretch{1.3}
\resizebox{\textwidth}{!}{%
\Huge
\begin{tabular}{|c|cccc|cccc|}
\hline
     & \multicolumn{4}{c|}{\textbf{relative   obj gap}}                                                                                                                                             & \multicolumn{4}{c|}{\textbf{relative   primal feasibility error}}                                                                                                                             \\ \hline
mt   & \multicolumn{1}{c|}{Gurobi} & \multicolumn{1}{c|}{\begin{tabular}[c]{@{}c@{}}fast\\ ADMM\end{tabular}} & \multicolumn{1}{c|}{HPR}     & \begin{tabular}[c]{@{}c@{}}HPR\\ hybrid\end{tabular} & \multicolumn{1}{c|}{Gurobi}  & \multicolumn{1}{c|}{\begin{tabular}[c]{@{}c@{}}fast\\ ADMM\end{tabular}} & \multicolumn{1}{c|}{HPR}     & \begin{tabular}[c]{@{}c@{}}HPR\\ hybrid\end{tabular} \\ \hline
100  & \multicolumn{1}{c|}{0}      & \multicolumn{1}{c|}{6.5E-05}                                             & \multicolumn{1}{c|}{1.0E-04} & 1.0E-04                                              & \multicolumn{1}{c|}{2.0E-09} & \multicolumn{1}{c|}{9.4E-06}                                             & \multicolumn{1}{c|}{9.6E-06} & 8.6E-06                                              \\ \hline
200  & \multicolumn{1}{c|}{0}      & \multicolumn{1}{c|}{1.0E-04}                                             & \multicolumn{1}{c|}{1.9E-04} & 1.7E-04                                              & \multicolumn{1}{c|}{1.3E-09} & \multicolumn{1}{c|}{9.6E-06}                                             & \multicolumn{1}{c|}{9.4E-06} & 8.6E-06                                              \\ \hline
400  & \multicolumn{1}{c|}{0}      & \multicolumn{1}{c|}{1.5E-04}                                             & \multicolumn{1}{c|}{2.6E-04} & 2.5E-04                                              & \multicolumn{1}{c|}{1.6E-09} & \multicolumn{1}{c|}{9.3E-06}                                             & \multicolumn{1}{c|}{9.5E-06} & 8.7E-06                                              \\ \hline
800  & \multicolumn{1}{c|}{0}      & \multicolumn{1}{c|}{1.6E-04}                                             & \multicolumn{1}{c|}{2.8E-04} & 3.0E-04                                              & \multicolumn{1}{c|}{2.7E-11} & \multicolumn{1}{c|}{9.5E-06}                                             & \multicolumn{1}{c|}{9.2E-06} & 8.7E-06                                              \\ \hline
1600 & \multicolumn{1}{c|}{0}      & \multicolumn{1}{c|}{1.7E-04}                                             & \multicolumn{1}{c|}{3.2E-04} & 2.5E-04                                              & \multicolumn{1}{c|}{5.4E-13} & \multicolumn{1}{c|}{9.1E-06}                                             & \multicolumn{1}{c|}{9.5E-06} & 9.0E-06                                              \\ \hline
\end{tabular}%
}}
\end{minipage}
}

\centering
\caption{Comparisons between Gurobi, fast-ADMM, HPR, and HPR-hybrid}
\label{Fig:experiment-2} 
\end{figure}

\subsection{Numerical Comparison of the Constants in the Complexity Bounds}
Now, we will compare the constants of the complexity bounds between the HPR algorithm and the entropy regularized algorithms, such as the IBP. The constant of the complexity bounds of the IBP depends on $\max_{t = 1}^T \|\mathcal{D}(\mathcal{P}^{c},\mathcal{P}^{t})\|_{\infty}$, which is one in our synthetic data settings. On the other hand, the constant for the HPR depends on the distance between the initial point and the solution, as shown in Theorem \ref{complexity}. Since we choose zeros as the initial point of the HPR algorithm, we only need to compute $\|x^{*}\|$ and $\|s^{*}\|$. We test it on the synthetic data with two different settings. For $(T,m,m_t)=(100,100,100)$, we get $\|x^{*}\|=1.10$ and $\|s^{*}\|=3.52$. For $(T,m,m_t)=(100,100,800)$, we get $\|x^{*}\|=0.98$ and $\|s^{*}\|=0.36$. These results partially demonstrate that the constants in the complexity bounds in Table \ref{tab:complexity-summary} and Table \ref{tab:complexity-summary-ADMM} are comparable. 

\subsection{Experiments on Real Data}
To further compare the HPR, and HPR-hybrid to other methods, we conduct experiments on some Real Data. Our experiments include MNIST data set \citep{lecun1998gradient}, Coil20 data set \citep{nene1996columbia}, and Yale Face B data set \citep{georghiades2001few}. For the MNIST data set, we randomly select $50$ images for each digit $(0,\ldots, 9)$ and resize each image to $56 \times 56$. For the Coil20 data set, we select 3 representative objects: Car, Duck, and Pig, where each object has $10$ images. We resize each image to $64 \times 64$. For the Yale Face B data set, we include two human subjects: YaleB01 and YaleB02. We randomly select $5$ images for each human subject and resize it to  $96 \times 84$. A summary of each data set is shown in Table \ref{tab:dataset}. For all data sets, we normalize the resulting image so that all pixel values add up to $1$. We generate the distance matrix similarly to \eqref{cost:D}. At last, we set the weight vector $\left(\omega_{1}, \omega_{2}, \ldots, \omega_{T}\right)$ such that $\omega_{t}=1 / T$ for all $t=1,\ldots,T$.
\begin{table}[htbp]
\centering
\caption{Summary of Real Data set}
\label{tab:dataset}
\resizebox{0.6\textwidth}{!}{%
\setlength{\tabcolsep}{8mm}{
\begin{tabular}{|c|c|c|c|}
\hline
Dataset & m    & mt   & T  \\ \hline
MNIST   & 3136 & 3136 & 50 \\ \hline
Car     & 4096 & 4096 & 10 \\ \hline
Duck    & 4096 & 4096 & 10 \\ \hline
Pig     & 4096 & 4096 & 10 \\ \hline
YaleB01 & 8064 & 8064 & 5  \\ \hline
YaleB02 & 8064 & 8064 & 5  \\ \hline
\end{tabular}
}
}
\end{table}
\par 
Since Gurobi is out of memory for this experiment, we do not compare with Gurobi in this subsection. For MNIST, we visualized the results in Figure \ref{Fig:experiment-3}, where the Wasserstein barycenters are obtained by different methods for the 50s and 100s respectively. One can see that HPR and HPR-hybrid can provide a clear “smooth” barycenter just like POT with regularization parameter $ \{0.001\}$ within a fixed time.

For the Coil20 and Yale Face B data sets, we first apply POT to get a barycenter as a  benchmark. We find that POT with regularization parameter $\{0.0005\}$ gives the best possible result. Although we try to use the smaller parameter, IBP implemented in POT encounters the numerical error even with the stabilized version. In this experiment, we want to know how many iterations of fast-ADMM, HPR, and HPR-hybrid are needed to reach the quality of the best solution returned by POT. The result of the Coil20 data set and Yale Face B data set is presented in Figure \ref{Fig:Coil20} and Figure \ref{Fig:YaleB}, respectively. We list the computational time of different methods on these data sets in Table \ref{tab:time-Coil20YaleB}. Figures \ref{Fig:Coil20} and \ref{Fig:YaleB} show that the quality of solutions produced by fast-ADMM, HPR, and HPR-hybrid is better than that of POT within 300 iterations. In particular, HPR-hybrid is faster and more stable than ADMM and HPR. HPR-hybrid can return a better result than POT in the 100th iteration, whose computational time is comparable to the time of POT from Table \ref{tab:time-Coil20YaleB}. Hence, we recommend using HPR-hybrid for computing the Wasserstein barycenter with high accuracy requirements in practice.  

\subsection{Summary of Experiments}

From the numerical results reported in the previous subsections, one can see that HPR and HPR-hybrid outperform the powerful commercial solver Gurobi in terms of the computational time for solving large-scale LPs arising from Wasserstein barycenter problems. Moreover, to get a high-quality solution, one needs to use a small regularization parameter for IBP-type methods, which will easily suffer from numerical instability issues. Compared with IBP-type methods, HPR-hybrid is more stable and can return a good solution in a comparable time. Hence, for computing a high-quality Wasserstein barycenter, we recommend applying HPR-hybrid to get a high-quality solution without the need to implement sophisticated stabilization techniques as in the case of IBP.

\begin{figure}[htbp]
\centering
\subfigure{
\begin{minipage}[t]{1\linewidth}
\centering
\includegraphics[width=12cm]{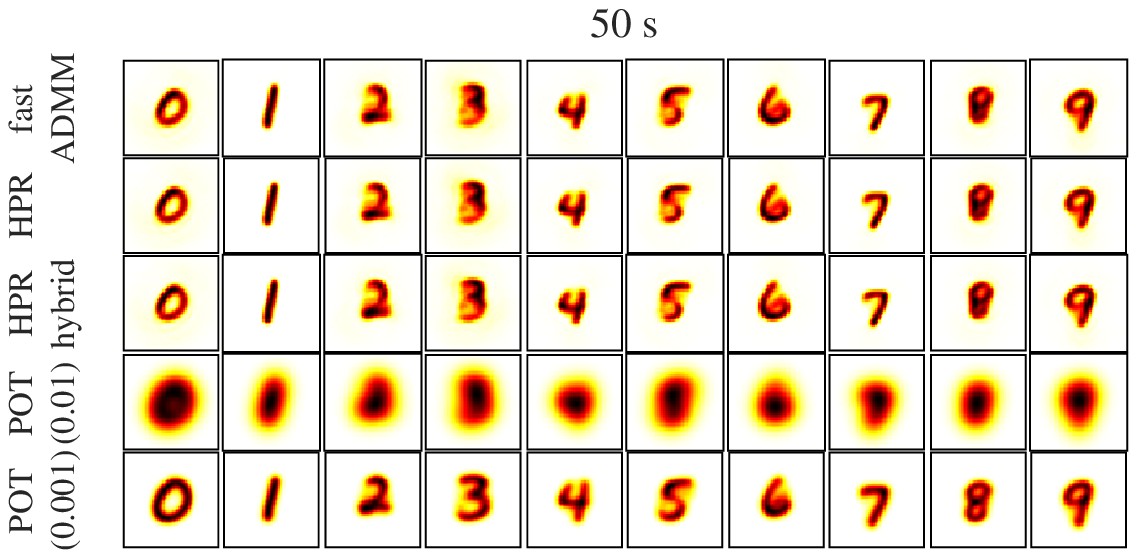}
\end{minipage}%
}%

\subfigure{
\begin{minipage}[t]{1\linewidth}
\centering
\includegraphics[width=12cm]{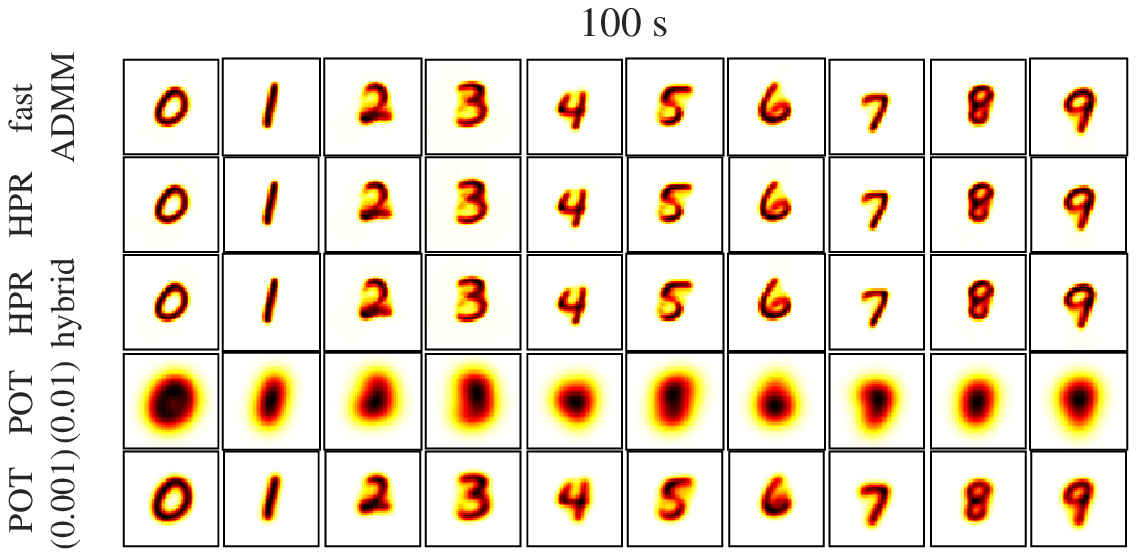}
\end{minipage}%
}%

\centering
\caption{Barycenters obtained by running different methods for 50s, 100s, respectively.}
\label{Fig:experiment-3}
\end{figure}

\begin{figure}[htbp]
\centering
\subfigure{
\begin{minipage}[t]{1\linewidth}
\centering
\includegraphics[width=11cm]{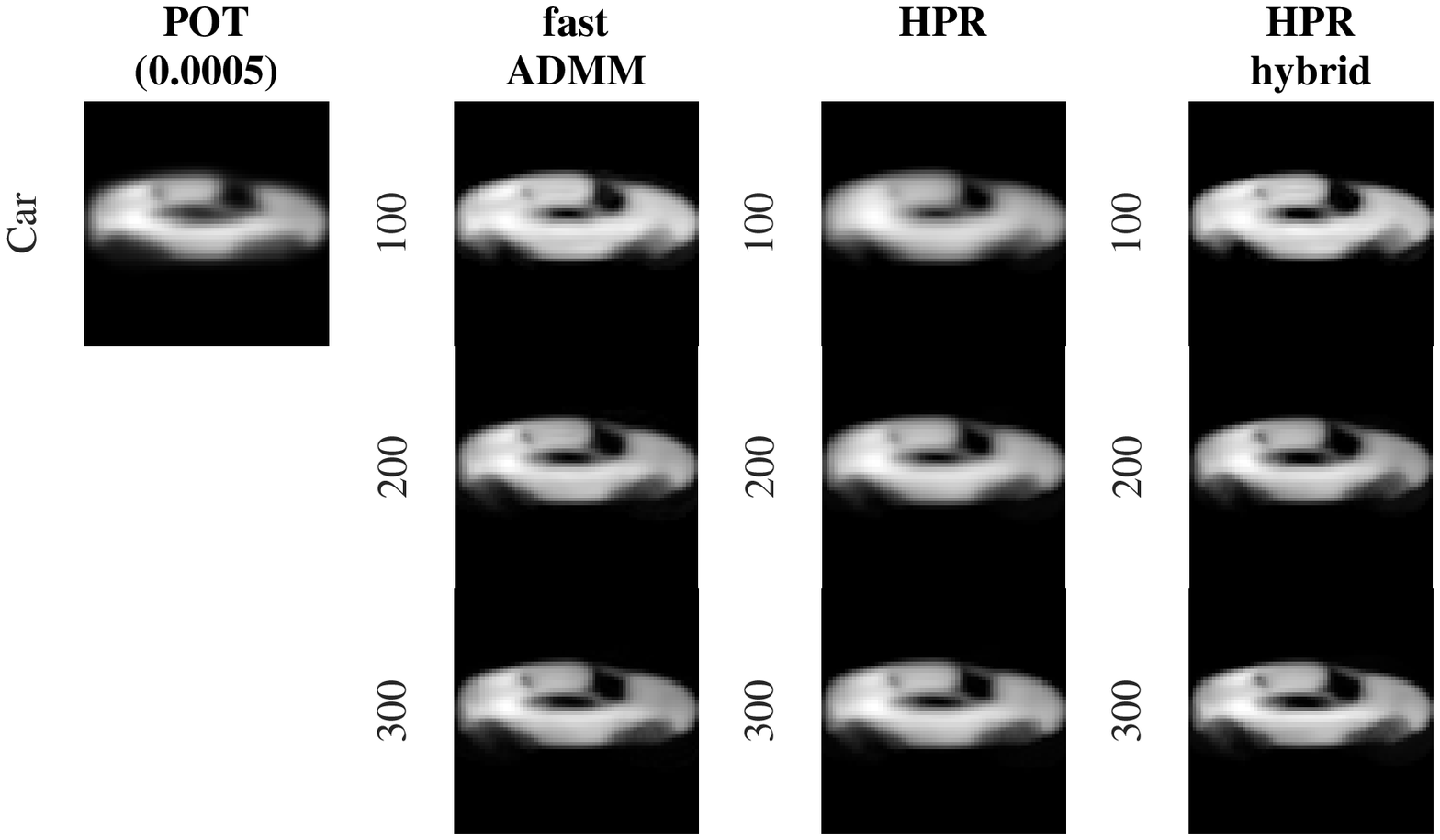}
\end{minipage}%
}%

\subfigure{
\begin{minipage}[t]{1\linewidth}
\centering
\includegraphics[width=11cm]{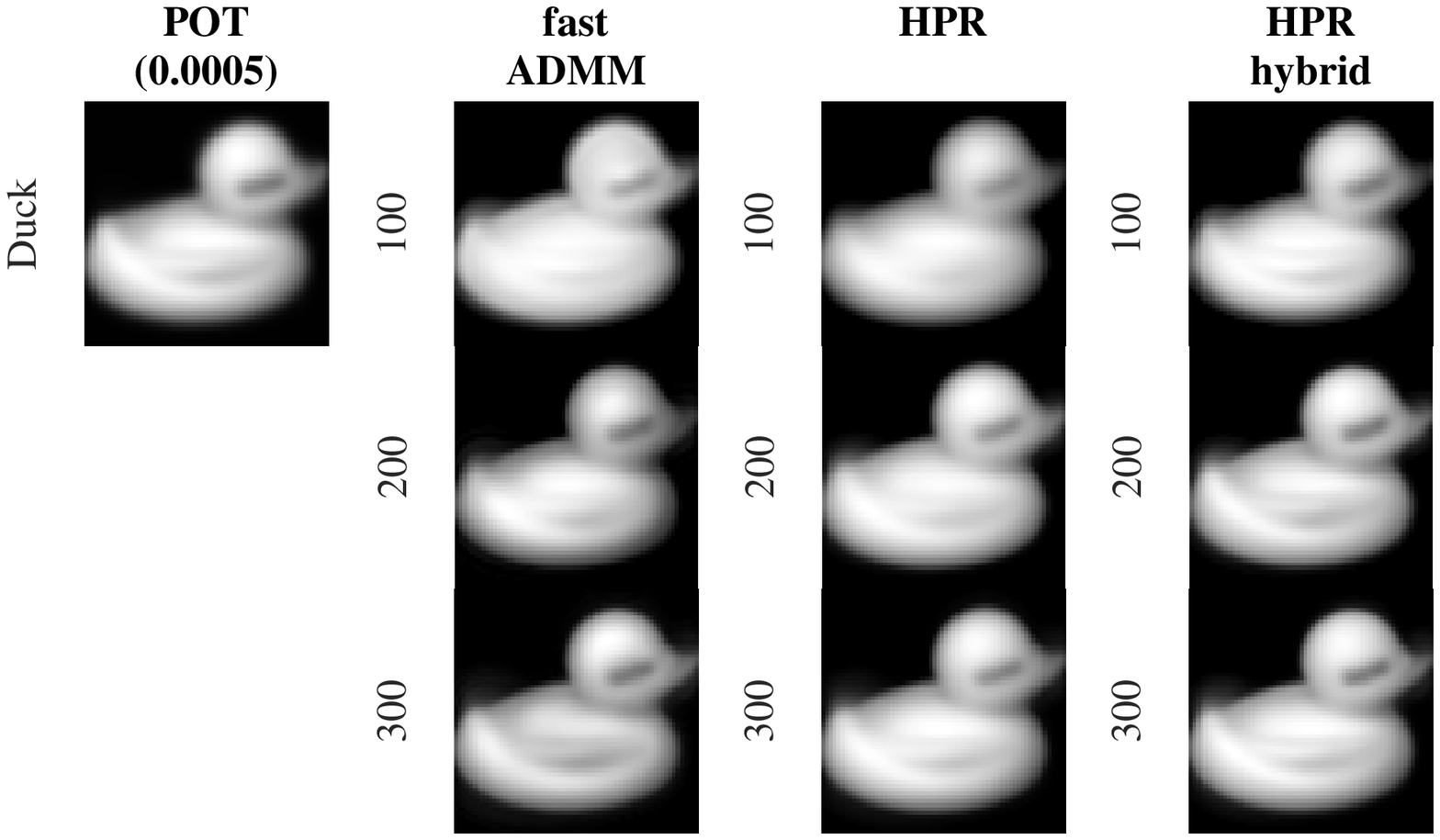}
\end{minipage}%
}%


\subfigure{
\begin{minipage}[t]{1\linewidth}
\centering
\includegraphics[width=11cm]{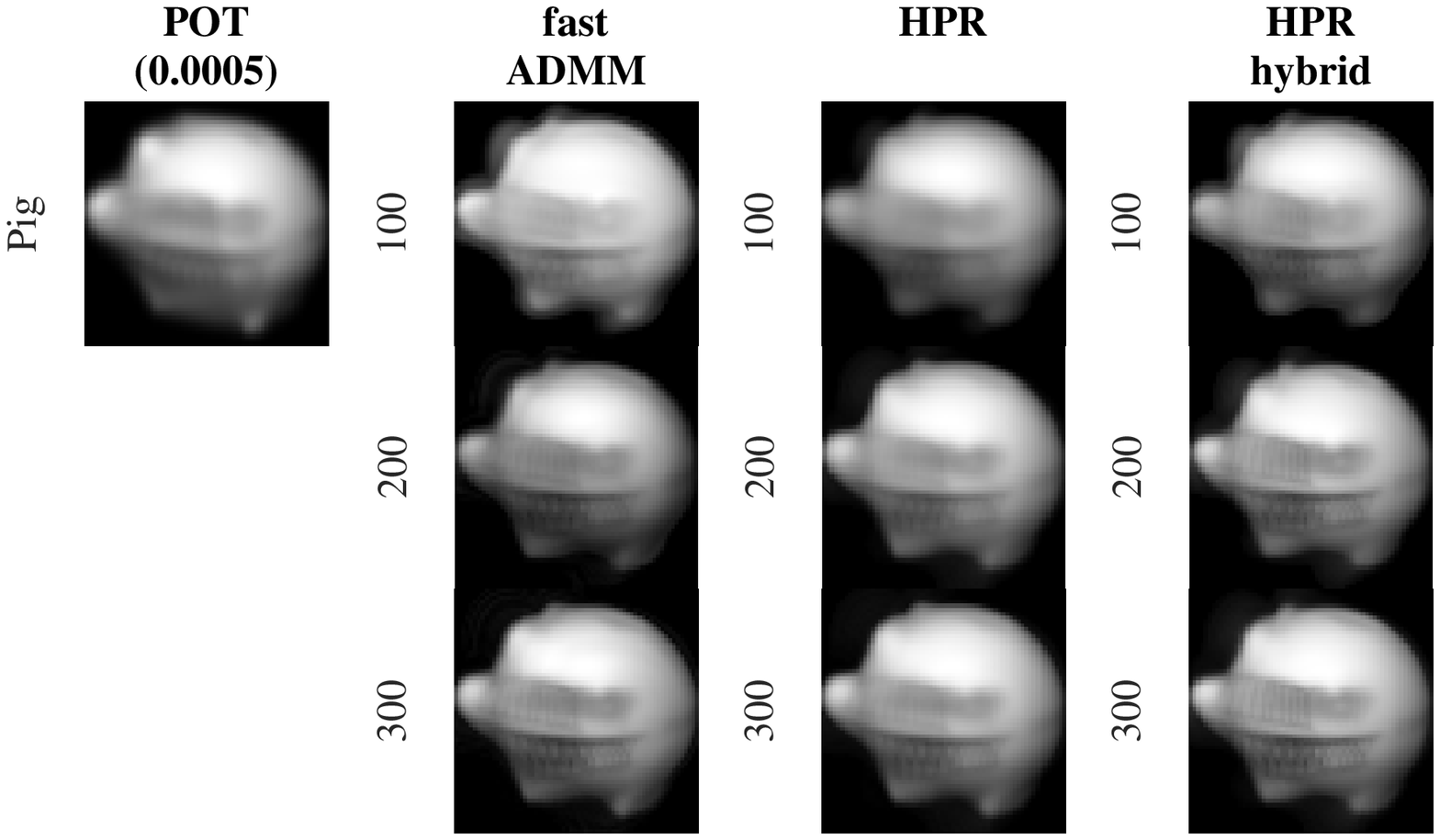}
\end{minipage}%
}%

\centering
\caption{Barycenters obtained by running different methods on the Coil20 data set for the 100th, 200th, and 300th iterations.}
\label{Fig:Coil20} %
\end{figure}

\begin{figure}[htbp]
\centering
\subfigure{
\begin{minipage}[t]{1\linewidth}
\centering
\includegraphics[width=16cm]{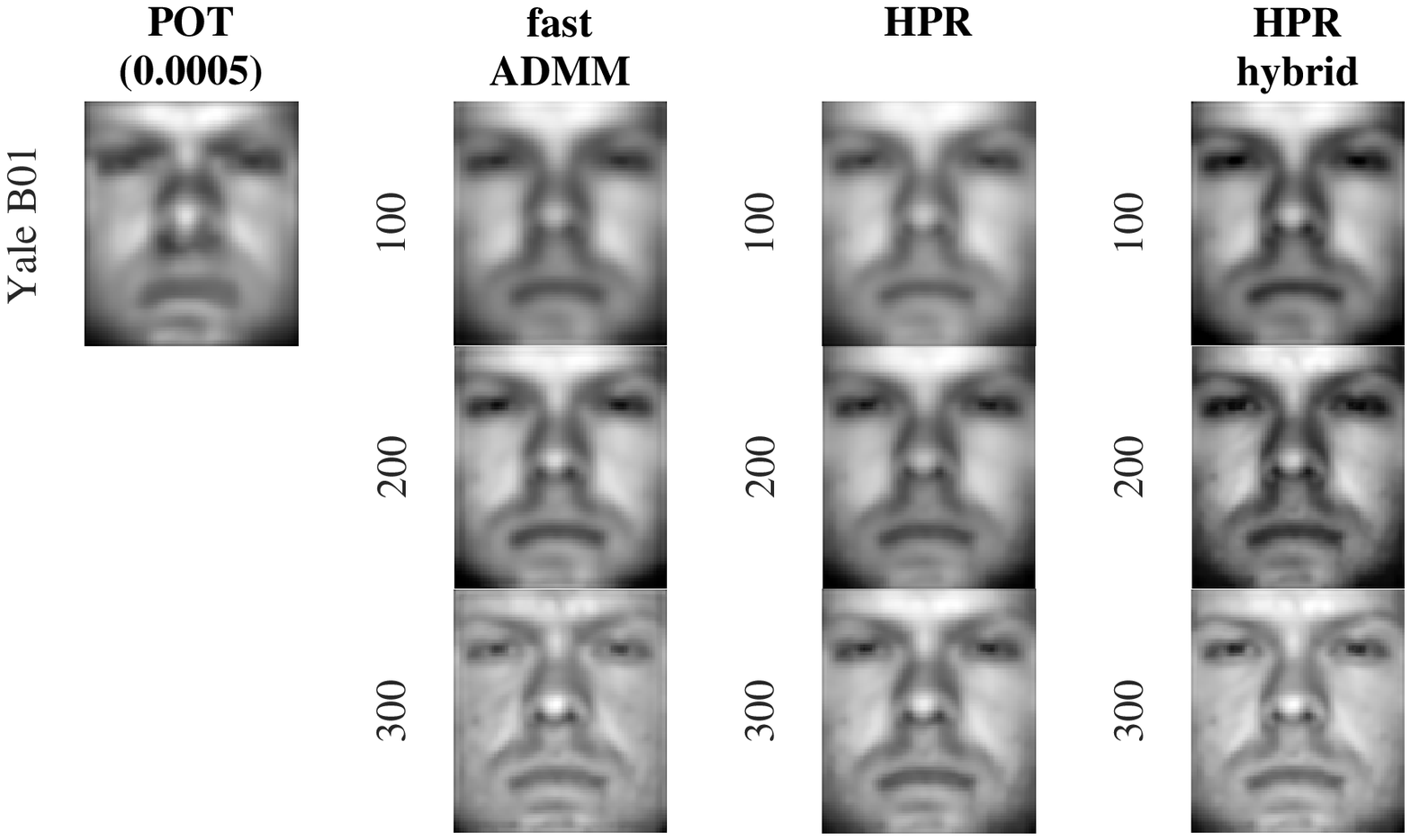}
\end{minipage}
}

\subfigure{
\begin{minipage}[t]{1\linewidth}
\centering
\includegraphics[width=16cm]{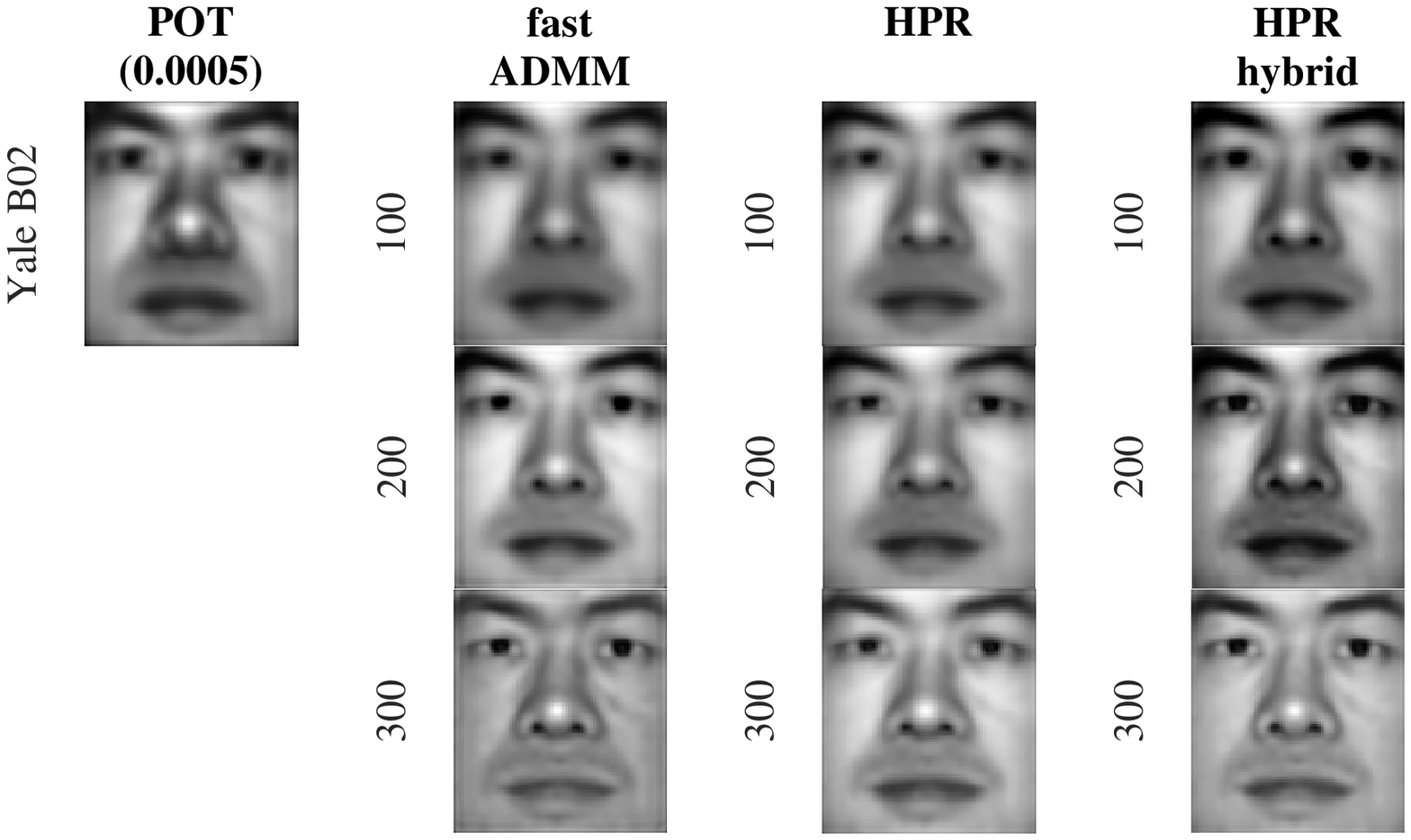}
\end{minipage}%
}%

\centering
\caption{Barycenters obtained by running different methods on the Yale Face B data set for the 100th, 200th, and 300th iterations.}
\label{Fig:YaleB} %
\end{figure}

\begin{table}[]
\centering
\caption{The computational time of different methods for computing the Wasserstein barycenter on the Coil20 data set and the Yale B face data set. (Unit: s)}
\label{tab:time-Coil20YaleB}
\resizebox{\textwidth}{!}{%
\begin{tabular}{|c|c|ccc|ccc|ccc|}
\hline
          & POT(0.0005) & \multicolumn{3}{c|}{fast-ADMM}                                     & \multicolumn{3}{c|}{HPR}                                           & \multicolumn{3}{c|}{HPR-hybrid}                                    \\ \hline
iter      & -           & \multicolumn{1}{c|}{100}    & \multicolumn{1}{c|}{200}    & 300    & \multicolumn{1}{c|}{100}    & \multicolumn{1}{c|}{200}    & 300    & \multicolumn{1}{c|}{100}    & \multicolumn{1}{c|}{200}    & 300    \\ \hline
Car       & 44.59       & \multicolumn{1}{c|}{27.61}  & \multicolumn{1}{c|}{55.22}  & 82.82  & \multicolumn{1}{c|}{25.60}  & \multicolumn{1}{c|}{51.20}  & 76.81  & \multicolumn{1}{c|}{25.79}  & \multicolumn{1}{c|}{51.58}  & 77.38  \\ \hline
Duck      & 48.89       & \multicolumn{1}{c|}{63.51}  & \multicolumn{1}{c|}{127.02} & 190.52 & \multicolumn{1}{c|}{60.68}  & \multicolumn{1}{c|}{121.36} & 182.03 & \multicolumn{1}{c|}{62.13}  & \multicolumn{1}{c|}{124.25} & 186.38 \\ \hline
Pig       & 33.25       & \multicolumn{1}{c|}{90.92}  & \multicolumn{1}{c|}{181.83} & 272.75 & \multicolumn{1}{c|}{87.33}  & \multicolumn{1}{c|}{174.66} & 262.00 & \multicolumn{1}{c|}{87.90}  & \multicolumn{1}{c|}{175.81} & 263.71 \\ \hline
YaleB01  & 164.62      & \multicolumn{1}{c|}{58.30}  & \multicolumn{1}{c|}{116.60} & 174.90 & \multicolumn{1}{c|}{55.76}  & \multicolumn{1}{c|}{111.51} & 167.27 & \multicolumn{1}{c|}{56.80}  & \multicolumn{1}{c|}{113.60} & 170.41 \\ \hline
YaleB02 & 153.86      & \multicolumn{1}{c|}{178.52} & \multicolumn{1}{c|}{357.04} & 535.57 & \multicolumn{1}{c|}{176.35} & \multicolumn{1}{c|}{352.70} & 529.06 & \multicolumn{1}{c|}{177.34} & \multicolumn{1}{c|}{354.68} & 532.03 \\ \hline
\end{tabular}%
}
\end{table}

\section{Concluding Remark}
In this paper, we introduce an efficient HPR algorithm for solving the WBP, which enjoys an appealing $O(1/\varepsilon)$ non-ergodic iteration complexity with respect to the KKT residual. We want to emphasize that the KKT residual is important since it is widely used as a reliable stopping criterion for the primal-dual algorithms. We also proposed a linear time complexity procedure to solve the linear system involved in the HPR algorithm for solving the WBP. As a consequence, the HPR algorithm enjoys an $O({\rm Dim(P)}/\epsilon)$ computational complexity in terms of flops to obtain an $\epsilon$-optimal solution to the WBP measured by the KKT residual. This result shows that the computational complexity of the HPR algorithm depends on the dimension of the WBP linearly. As a byproduct, we also get an efficient procedure for solving the OT problem. Extensive numerical experiments demonstrate the superior performance of the HPR algorithm for obtaining high-quality solutions to the WBP on both synthetic datasets and real image datasets. 

  In the future, we will develop a highly efficient GPU solver based on the HPR algorithm for solving the WBP and OT. It has been shown that the proposed efficient procedure for solving the linear system can directly benefit to the ADMM algorithm. It is noted that \cite{Monteiro2017} proposed a dynamic regularized ADMM algorithm that enjoys a non-ergodic iteration complexity of $O(1/\varepsilon \log(1/\varepsilon))$ in terms of the KKT residual. We will further study to extend the proposed procedure to solve the subproblems of the dynamic regularized ADMM algorithms for solving the WBP. As an open exploration research question, we will further investigate other acceleration techniques to see whether it is possible to design a new algorithm to solve the WBP with a better computational complexity than $O({\rm Dim(P)}/\varepsilon)$.

\acks{The research of Yancheng Yuan is supported by the Hong Kong Polytechnic University under  grant P0038284. The research of Defeng Sun is supported in part by the Hong Kong Research Grant Council under grant 15303720.}

\appendix
\section{The Proof of Proposition \ref{prop:alg1-2-3}}\label{proof:alg1-2-3}
\begin{proof}
We prove the first statement in the proposition by induction. Given $\eta^{0}$ in $\mathbb{X}$. For $k=0$, 
by the definition of $s^{1}$ in Algorithm \ref{alg:HPR-OP-0}, we have 
$$
0\in \partial f_2(s^{1})+B_{2}^{*}\eta^{0}+\sigma B_{2}^{*}B_{2}s^{1}.
$$
Thus, 
$$
-B_{2}^{*}(\eta^{0}+\sigma B_{2}s^{1})\in \partial f_2(s^{1}).
$$
It follows from \citep[Theorem 23.5]{rockafellar1970convex} that 
$$
s^{1}\in \partial f^{*}_2(-B_{2}^{*}(\eta^{0}+\sigma B_{2}s^{1})).
$$
Hence,
$$
-\sigma B_{2}s^{1}\in -\sigma B_{2}\partial f^{*}_2(-B_{2}^{*}(\eta^{0}+\sigma B_{2}s^{1})).
$$
This means
$$
\eta ^{0} \in  \eta^{0} +\sigma B_{2}s^{1}+ \sigma \partial (f_2^* \circ (-B_2^*))(\eta^{0} +\sigma B_{2}s^{1} ).
$$
Since $\mathbf{M}_2 = \partial (f_2^* \circ (-B_2^*))$, we have 
$$
\eta ^{0} \in  \eta^{0} +\sigma B_{2}s^{1}+ \sigma \mathbf{M}_2   (\eta^{0} +\sigma B_{2}s^{1} )   ,
$$
which implies 
\begin{equation}\label{prop:equ-w}
 w^{1}:= \boldsymbol{J}_{\sigma \boldsymbol{M}_{2}}(\eta^{0}) = \eta^{0}+\sigma B_{2}s^{1}.  
\end{equation}
Similarly, by the definition of $y^{1}$, we have 
$$
0\in \partial f_1(y^{1})+B_{1}^{*}\left(\eta^{0}+2\sigma B_{2}s^{1}+\sigma (B_{1}y^{1}-c)\right).
$$
It follows from \citep[Theorem 23.5]{rockafellar1970convex} that 
$$
y^{1}\in \partial f^{*}_1\left(-B_{1}^{*}\left(\eta^{0}+2\sigma B_{2}s^{1}+\sigma (B_{1}y^{1}-c)\right)\right).
$$
Hence,
$$
-\sigma B_{1}y^{1}\in -\sigma B_{1} \partial f^{*}_1\left(-B_{1}^{*}\left(\eta^{0}+2\sigma B_{2}s^{1}+\sigma (B_{1}y^{1}-c)\right)\right).
$$
This implies that
$$
2(\eta^{0}+\sigma B_{2}s^{1})-\eta ^{0}\in 2(\eta^{0}+\sigma B_{2}s^{1})-\eta ^{0}+ \sigma B_{1}y^{1} -\sigma B_{1} \partial f^{*}_1\left(-B_{1}^{*}\left(\eta^{0}+2\sigma B_{2}s^{1}+\sigma (B_{1}y^{1}-c)\right)\right).
$$
That is 
$$
\begin{array}{ll}
     2(\eta^{0}+\sigma B_{2}s^{1})-\eta ^{0}\in &\eta^{0}+2\sigma B_{2}s^{1}+ \sigma (B_{1}y^{1} -c)+\dots\\
   &\sigma  \left(\partial (f_1^* \circ (-B_1^*))\left(\eta^{0}+2\sigma B_{2}s^{1}+\sigma (B_{1}y^{1}-c)\right)   + c\right).
\end{array}
$$
Since $\mathbf{M}_1 = \partial (f_1^* \circ (-B_1^*)) + c$, we have
$$
2(\eta^{0}+\sigma B_{2}s^{1})-\eta ^{0} \in  \eta^{0}+2\sigma B_{2}s^{1}+ \sigma (B_{1}y^{1} -c)+ \sigma \mathbf{M}_1 \left(\eta^{0}+2\sigma B_{2}s^{1}+ \sigma (B_{1}y^{1} -c)\right).
$$
which implies 
\begin{equation}\label{prop:equ-x}
x^{1}:= \boldsymbol{J}_{\sigma \boldsymbol{M}_{1}}\left(2 w^{1}-\eta^{0}\right) = \eta^{0}+2\sigma B_{2}s^{1}+ \sigma (B_{1}y^{1} -c).    
\end{equation}
From \eqref{prop:equ-w} and \eqref{prop:equ-x} we have  
$$
x^{1}-w^{1}=\sigma (B_{1}y^{1}+B_{2}s^{1}-c). 
$$
Hence
\begin{equation}\label{prop:equ-v}
  v^{1}:=\eta^{0}+2\left(x^{1}-w^{1}\right)=\eta^{0}+2\sigma (B_{1}y^{1}+B_{2}s^{1}-c).  
\end{equation}
It follows that the update of $\eta^{1}$ is the same in Algorithm \ref{alg:HPR} as in Algorithm \ref{alg:HPR-OP-0}. Hence, we prove the statement for $k=0$. Assume that the statement holds for some $k\geq 1$. For $k := k+1$, we can prove that the statement holds similarly to the case $k=0$. Thus, we prove the statement holds for any $k \geq 0$ by induction. This completes the first part.

Now we show the proof of the second part by induction. Let $\eta^{0}:=\hat{x}^{0}+\sigma (B_{1}y^{0}-c)$. For $k=0$, from Algorithm \ref{alg:HPR-OP-0}, we have 
$$
\begin{array}{ll}
s^{1} &:=\underset{s\in \mathbb{Z}}{\arg \min }\left\{f_{2}(s)+\langle\eta^{0}, {B}_{2} s\rangle+\frac{\sigma}{2}\|{B}_{2} s\|^{2}\right\}\\
&=\underset{s\in \mathbb{Z}}{\arg \min }\left\{f_{2}(s)+\langle\hat{x}^{0}+\sigma (B_{1}y^{0}-c), {B}_{2} s\rangle+\frac{\sigma}{2}\|{B}_{2} s\|^{2}\right\}
\\
&=\underset{s\in \mathbb{Z}}{\arg \min }\left\{f_{2}(s)+\langle\hat{x}^{0} , B_{1}y^{0}+{B}_{2}s-c \rangle+\frac{\sigma}{2}\|B_{1}y^{0}+{B}_{2} s-c\|^{2}\right\}.
\end{array}
$$
For $y^{1}$, from Algorithm \ref{alg:HPR-OP-0}, we have 
$$
\begin{array}{ll}
{y}^{1} &:=\underset{y\in \mathbb{Y}}{\arg \min }\left\{f_{1}(y)+\langle\eta^{0}+2 \sigma {B}_2 s^{1}, {B}_{1}{y}-{c}\rangle+\frac{\sigma}{2}\|{B}_{1}y -{c}\|^{2}\right\}\\
&=\underset{y\in \mathbb{Y}}{\arg \min }\left\{f_{1}(y)+\langle\hat{x}^{0}+\sigma (B_{1}y^{0}-c)+2 \sigma {B}_2 s^{1}, {B}_{1}{y}-{c}\rangle+\frac{\sigma}{2}\|{B}_{1}y -{c}\|^{2}\right\}\\
&=\underset{y\in \mathbb{Y}}{\arg \min }\left\{f_{1}(y)+\langle\hat{x}^{0}+\sigma (B_{1}y^{0}+{B}_2 s^{1} -c) , {B}_{1}{y}+B_{2}s^{1}-{c}\rangle+\frac{\sigma}{2}\|{B}_{1}y+B_{2}s^{1} -{c}\|^{2}\right\}.
\end{array}
$$
Define $x^{\frac{1}{2}}:=\hat{x}^0+\sigma\left(B_1 y^0+B_2 s^{1}-c\right)$. We have 
$$
\begin{array}{ll}
{y}^{1} =\underset{y\in \mathbb{Y}}{\arg \min }\left\{f_{1}(y)+\langle x^{\frac{1}{2}} , {B}_{1}{y}+B_{2}s^{1}-{c}\rangle+\frac{\sigma}{2}\|{B}_{1}y+B_{2}s^{1} -{c}\|^{2}\right\}
\end{array}.
$$
For $x^{1}$, from Algorithm \ref{alg:HPR-OP-0}, we have 
$$
\begin{array}{ll}
{x}^{1} &:=\eta^{0}+\sigma({B}_{1} {y}^{1}-c)+2 \sigma {B}_{2} s^{1}\\
&=\hat{x}^0+\sigma\left(B_1 y^0-c\right)+\sigma({B}_{1} {y}^{1}-c)+2 \sigma {B}_{2} s^{1}\\
&=\hat{x}^0+\sigma\left(B_1 y^0+{B}_{2} s^{1}-c\right)+\sigma({B}_{1} {y}^{1}+\sigma {B}_{2} s^{1}-c) \\
&=x^{\frac{1}{2}}+\sigma({B}_{1} {y}^{1}+\sigma {B}_{2} s^{1}-c).
\end{array}
$$
For $\hat{x}^{1}$ defined in \eqref{x&hatx}, We have 
$$
\begin{array}{ll}
\hat{x}^{1} &:=\eta^{1}-\sigma\left(B_1 y^{1}-c\right)\\
&=\frac{1}{2} \eta^0+\frac{1}{2} v^{1}-\sigma\left(B_1 y^{1}-c\right)\\
&=\frac{1}{2}(\hat{x}^{0}+\sigma (B_{1}y^{0}-c))+\frac{1}{2}\left(\hat{x}^{0}+\sigma (B_{1}y^{0}-c)+2 \sigma(B_1 y^{1}+B_2 s^{1}-c)\right)-\sigma\left(B_1 y^{1}-c\right)
\\
&=\frac{1}{2}(\hat{x}^{0}+\sigma (B_{1}y^{0}-c))+\frac{1}{2}(x^{1})-\frac{1}{2}\sigma\left(B_1 y^{1}-c\right). 
\end{array}
$$
Hence, we prove the statement for $k=0$. Assume that the statement holds for some $k \geq 0$. Then, we can prove the statement for $k:= k+1$ similarly to the case for $k=0$. Thus, by induction, we have completed the proof.
\end{proof}


\end{document}